\documentclass[12pt,reqno]{amsart}

\usepackage{lineno}

\usepackage[abs]{overpic}

\usepackage{amsmath, amscd, amssymb, amsthm,amsfonts}
\usepackage{stmaryrd,euscript, mathrsfs, latexsym,mathabx}
\usepackage{color}
\usepackage{hyperref} % links to refs

\usepackage{lmodern}
\usepackage[T1]{fontenc} % things like \"o

\usepackage{multicol}

\ifx\pdfoutput\undefined
\usepackage{graphicx}
\else
\fi

\ifx\xetex
\usepackage{fontspec}
\usefont{EU1}{lmr}{m}{n}%Latin Modern in OpenType format.
\else
\fi

\usepackage{url} % hyperref works too
\usepackage{tikz}
\usetikzlibrary{matrix,arrows,positioning}
\tikzset{node distance=2cm, auto}

\usepackage[nodayofweek]{datetime}

\pdfoutput=1
\usepackage[activate={true,nocompatibility},final,kerning=true,tracking=true,spacing=true,factor=1100,stretch=10,shrink=10]{microtype}
\SetTracking{encoding={*}, shape=sc}{0}
\microtypecontext{spacing=nonfrench}

\parskip=\medskipamount
\newcommand{\conj}[1]{\quad\textnormal{ #1 }\quad}

\newcommand{\inp}[1]{\ensuremath{\langle #1 \rangle}}

\newcommand{\module}[0]{\operatorname{-mod}}

\newcommand{\normaltext}[1]{\textnormal{#1}}

\makeatletter
\def\imod#1{\allowbreak\mkern2.5mu({\operator@font mod}\,#1)}
\makeatother

\renewcommand{\a}{\alpha}
\renewcommand{\b}{\beta}
\newcommand{\e}{\epsilon}
\newcommand{\vp}{\varphi}
\newcommand{\opp}{\oplus}
\newcommand{\ott}{\otimes}

\newcommand{\s}{\sigma}

\newcommand{\Ga}{\Gamma}
\newcommand{\ga}{\gamma}
\renewcommand{\th}{\theta}
\renewcommand{\d}{\delta}

\newcommand{\CPic}[1]{
\begin{minipage}{.45in}
\includegraphics[scale=.75]{#1}
\end{minipage}
}

\newcommand{\CPPic}[1]{
\begin{minipage}{.8in}
\includegraphics[scale=1.1]{#1}
\end{minipage}
}

\newcommand{\BCPPic}[1]{
\begin{minipage}{1.6in}
\includegraphics[scale=2]{#1}
\end{minipage}
}

\newcommand{\MPic}[1]{
\begin{minipage}{.35in}
\includegraphics[scale=.45]{#1}
\end{minipage}
}

\newcommand{\BPic}[1]{
\begin{minipage}{1in}
\includegraphics[scale=1.5]{#1}
\end{minipage}
}

\newcommand{\aA}{\mathcal{A}}

\newcommand{\aC}{\mathcal{C}}
\newcommand{\aD}{\mathcal{D}}

\newcommand{\aK}{\mathcal{K}}

\newcommand{\aM}{\mathcal{M}}
\newcommand{\aN}{\mathcal{N}}

\newcommand{\aP}{\mathcal{P}}
\newcommand{\aQ}{\mathcal{Q}}
\newcommand{\aR}{\mathcal{R}}

\newcommand{\aW}{\mathcal{W}}

\newcommand{\aZ}{\mathcal{Z}}

\newcommand{\ba}{\mathbf{a}}

\newcommand{\fZ}{\mathfrak{Z}}

\newcommand{\fii}{\mathfrak{i}}

\newcommand{\fz}{\mathfrak{z}}

\usepackage{bbm}

\newcommand{\CC}{\mathbb{C}}

\newcommand{\FF}{\mathbb{F}}

\newcommand{\LL}{\mathbb{L}}

\newcommand{\NN}{\mathbb{N}}

\newcommand{\RR}{\mathbb{R}}

\newcommand{\TT}{\mathbb{T}}

\newcommand{\ZZ}{\mathbb{Z}}

\newcommand{\eA}{\EuScript{A}}
\newcommand{\eB}{\EuScript{B}}
\newcommand{\eC}{\EuScript{C}}
\newcommand{\eD}{\EuScript{D}}
\newcommand{\eE}{\EuScript{E}}

\newcommand{\eR}{\EuScript{R}}
\newcommand{\eS}{\EuScript{S}}

\newcommand{\eX}{\EuScript{X}}

\theoremstyle{plain}

\newtheorem{thm}{Theorem}[section]

\newtheorem{theorem}[thm]{Theorem}

\newtheorem{question}[thm]{Question}
\newtheorem{assumption}[thm]{Assumption}
\newtheorem{conjecture}[thm]{Conjecture}

\newtheorem{upropertyy}[thm]{Universal property}

\newtheorem{proposition}[thm]{Proposition}
\newtheorem{prop}[thm]{Proposition}

\newtheorem{corollary}[thm]{Corollary}
\newtheorem{cor}[thm]{Corollary}
\newtheorem{lemma}[thm]{Lemma}

\theoremstyle{remark}

\theoremstyle{definition}

\newtheorem{example}[thm]{Example}

\newtheorem{defn}[thm]{Definition}

\newtheorem{definition}[thm]{Definition}

\newtheorem{rmk}[thm]{Remark}
\newtheorem{remark}[thm]{Remark}

\numberwithin{equation}{section}

\let\emptyset\varnothing

\makeatletter
\def\imod#1{\allowbreak\mkern2.5mu({\operator@font mod}\,#1)}
\makeatother

\newcommand{\Ad}{\aZ}

\newcommand{\lns}{Z}
\newcommand{\pts}{\ba}
\newcommand{\asuf}{F(\aZ)}
\newcommand{\asurf}{\asuf}
\newcommand{\param}{\aP}

\newcommand{\zigzag}{\aM}
\newcommand{\dr}{\tilde{I}}

\newcommand{\alg}{\aA}

\newcommand{\z}{\fz}%{{\bf r}}
\newcommand{\an}{an}

\newcommand{\Zi}{\fZ}

\newcommand{\vnp}[1]{\lvert #1 \rvert}
\newcommand{\I}{[0,1]}
\newcommand{\xto}[1]{\xrightarrow{#1}}
\newcommand{\xfrom}[1]{\xleftarrow{#1}}
\newcommand{\from}{\leftarrow}
\usepackage{extarrows}
\newcommand{\xlto}[1]{\xlongrightarrow{#1}}

\newcommand{\spc}{\mathfrak{e}}
 
\newcommand{\Hqe}{\normaltext{Hqe}}
\newcommand{\Hmo}{\normaltext{Hmo}}
\newcommand{\Kom}{Kom}
\newcommand{\Ho}{Ho}
\newcommand{\dgcat}{\normaltext{dgcat}_k}
\newcommand{\Forget}{\normaltext{Forget}}
\newcommand{\Mat}{\textrm{Mat}}
\newcommand{\op}{\normaltext{op}}

\theoremstyle{plain}

\newcommand{\Sp}{R_+}
\newcommand{\Sm}{R_-}

\newcommand{\Co}{\mathcal{C}\hspace{-.1475em}o}

\newcommand{\Yi}{\mathcal{Y}}
\newcommand{\Ko}[0]{\aK\hspace{-.1475em}o}

\newcommand{\PKo}[0]{\normaltext{Pre-}\Ko}%{\sK\!\!\sK}
\newcommand{\PPKo}[0]{\normaltext{Pre-Pre-}\Ko}%{\sK\!\!\sK}
\newcommand{\Si}[0]{\Sigma}
\newcommand{\KoS}[0]{\Ko(\Sigma)}

\newcommand{\KoSm}[0]{\Ko(\Sigma,m)}
\newcommand{\KoSmn}[0]{\Ko^n(\Sigma,m)}
\newcommand{\PKoSm}[0]{\PKo(\Sigma,m)}
\newcommand{\PKoSmn}[0]{\PKo^n(\Sigma,m)}
\newcommand{\PKoS}[0]{\PKo(\Sigma)}
\newcommand{\PPKoS}[0]{\PPKo(\Sigma)}

\newcommand{\HoKoS}[0]{\Ho(\Ko(\Sigma))}
\newcommand{\mcgS}[0]{\Gamma(\Sigma)}

\newcommand{\Vect}{\normaltext{Vect}}
\newcommand{\Hom}{Hom}
\newcommand{\im}{im}

\newcommand{\Aut}{Aut}
\newcommand{\Endo}{End}

\newcommand{\pt}{\normaltext{pretr}}
\newcommand{\perf}{\normaltext{perf}}
\newcommand{\Obj}{Ob}
\newcommand{\Ob}{\Obj}

\renewcommand{\s}{\sigma}

\begin{document}

\title[Formal Contact Categories]{Formal Contact Categories}
\author[Benjamin Cooper]{Benjamin Cooper}
\address{University of Iowa, Department of Mathematics, 14 MacLean Hall, Iowa City, IA 52242-1419}
\email{ben-cooper\char 64 uiowa.edu}

\begin{abstract}
  To each oriented surface $\Sigma$, we associate a differential graded
  category $\KoS$. The homotopy category $\HoKoS$ is a triangulated category
  which satisfies properties akin to those of the contact categories studied
  by K. Honda. These categories are also related to the algebraic contact
  categories of Y. Tian and to the bordered sutured categories of R. Zarev. 
\end{abstract}

\maketitle
\setcounter{tocdepth}{1}
\setcounter{secnumdepth}{2}
\tableofcontents

\section{Introduction}\label{introsec}

The purpose of this paper is to associate a differential graded category
$\KoS$ to each oriented surface $\Si$. This category is used to study
comparison problems between the categories associated to surfaces by
Seiberg-Witten-type manifold invariants. For example, we prove that the
categories associated to the disk $(D^2,2n)$ with $2n$ marked points by each
theory are equivalent and there is a functorial relationship between the
categories associated to a surfaces with boundary when they can be defined.

\subsection{The unicity of Floer-type invariants of 3-manifolds}

In \cite{OZ1,OZ2} P. Ozsv\'{a}th and Z. Szab\'{o} introduced invariants of
3-manifolds known as the Heegaard-Floer homologies. Depending upon the
setting of a parameter $U$, there are homology groups: $HF^{-}_*(M)$,
$HF^{+}_*(M)$, $HF^{\infty}_*(M)$ which fit into a long exact sequence:
\begin{equation}\label{hfleseqn}
\cdots \to HF^-_*(M) \to HF^{\infty}_*(M) \to HF^+_*(M) \to \cdots.
\end{equation}
When the parameter $U=0$, there are simpler invariants $\widehat{HF}_*(M)$.
The Heegaard-Floer theory has had a profound effect on the study of
3-manifolds and 4-manifolds \cite{AJ}. This is in part because it was
originally conceived of as a means by which one can obtain information in
the Seiberg-Witten invariants \cite{Donaldson, KM, W}. The relationship
between the Heegaard-Floer homology theory and the Seiberg-Witten Floer
homology was recently articulated by two independent groups of researchers:
\c{C}. Kutluhan, Y-J. Lee, C. H. Taubes \cite{TaubesGroup} and V. Colin,
P. Ghiggini, K. Honda \cite{HondaGroup}. Both teams built upon the earlier
work of C. H. Taubes \cite{Taubes1}, which identified the Seiberg-Witten
Floer homologies $\widehat{HM}_*(M)$ with the Embedded Contact Homology
$ECH_*(M)$ due to M. Hutchings \cite{H1} and M. Hutchings and C. H. Taubes
\cite{H2,H3}:
$$\Omega : ECH_*(M) \xto{\sim} \widehat{HM}_*(M).$$
Using the Embedded Contact Homology as an intermediary, both groups
completed the diagram:
$$\begin{tikzpicture}[scale=10, node distance=2.5cm]
\node (A1) {$ECH_*(M)$};
\node (X) [right=1.25cm of A1] {};
\node (B1) [above=1.25cm of X] {$HF^-_*(M)$};
\node (C1) [below=1.25cm of X] {$\widehat{HM}_*(M),$};
\draw[->] (A1) to node {} (B1);
\draw[->] (A1) to node [swap]  {$\Omega$} (C1);
\draw[->] (B1) to node  {} (C1);
\end{tikzpicture}$$
in a fashion which preserved essential properties of the three homology
theories. In particular, the maps defined respect decompositions with respect to
$Spin^{\CC}$ structures, carry invariants of contact structures to
invariants of contact structures, preserve the long exact
sequence \eqref{hfleseqn} and reductions to the simpler, $U=0$, theory:
\begin{equation}\label{heq}
\widehat{ECH}_*(M) \cong \widehat{HF}_*(M) \cong \widetilde{HM}_*(M).
\end{equation}

Intuitively, each component in the equation above corresponds to a
codimension 1 piece of a 4-dimensional topological field theory. It is
evident that such a theory satisfies the following properties. In
codimension 1, a topological field theory associates a chain complex $C(M)$
to each oriented 3-manifold $M$. The homology of this chain complex $H_*(C(M))$ is an
invariant of the diffeomorphism type of the 3-manifold.  In codimension 2, a
topological field theory associates a differential graded category
$\aC(\Si)$ to each oriented surface $\Si$. The derived category
$D(\aC(\Si))$ of right $\aC(\Si)$-modules \cite{Kellerder, Kellerdg} is an
invariant of the diffeomorphism type of the surface and reversing the
orientation of the surface produces the opposite dg category:
$$\aC(\bar{\Si}) \cong \aC(\Si)^{\op}.$$ 
To each 3-manifold $X$ with boundary $\partial X = \Si$, there is a right $\aC(\Si)$-module $X_*$. When a 3-manifold $M$ is split along a surface $M = X \cup_\Si Y$, the invariant $C(M)$ corresponding to $M$ is quasi-isomorphic to the tensor product,
$$C(M) \xto{\sim} X_* \ott^{\LL}_{\aC(\Si)} (Y_*)^{\op},$$
of the modules associated to each piece. If the identifications made by Equation \eqref{heq} result from an
equivalence between topological field theories then the codimension 2
extensions of these topological field theories must be equivalent as
well. 

\begin{question}
  Is there an equivalence between codimension 2 extensions of Seiberg-Witten
  Floer, Heegaard-Floer and Embedded Contact Homology?
\end{question}

In this paper, we study the simpler question of establishing a relationship
between the categories associated to oriented surfaces $\Si$ by
Heegaard-Floer theory and contact topology.

The Heegaard-Floer homology $\widehat{HF}^*(M)$ was extended to surfaces and
3-manifolds with boundary, in the manner described above, by the authors
P. Ozsv\'{a}th, R. Lipshitz and D. Thurston \cite{LOT}. The theory was
further developed by R. Zarev \cite{Zarev1, Zarev2}. In particular, when an
oriented surface $\Si$ sports a handle decomposition, determined by
combinatorial data $\Ad$ called an {\em arc parameterization}, there is a dg
category $\alg(-\Ad)$ which is associated to the surface $\Si$. The Morita
homotopy class of the corresponding categories of dg modules are independent
of the handle decomposition $\Ad$.

On the contact side, K. Honda has conjectured the existence of a family of
triangulated categories $\Co(\Si)$ associated to oriented surfaces $\Si$
called {\em contact categories} \cite{KoHo}. These categories might function
as part of a codimension 2 component of the Embedded Contact homology.  The
morphisms of contact categories are isotopy classes of tight contact
structures on a thickened surface $\Si\times [0,1]$. Maps in $\Co(\Si)$ are
composed by gluing $\Si\times [0,1]$ to $\Si \times [0,1]$ and rescaling.
The contact categories $\Co(\Si)$ are conjectured to contain distinguished
triangles associated to special contact structures called bypass
moves. Unfortunately, this construction is not yet available in its full
generality. For disks and annuli, algebraic analogues of these categories
were introduced and studied by Y. Tian \cite{YT1,YT2}.

\subsection{Summary of main results}

In this paper, we associate a $\mathbb{Z}/2$-linear dg category $\Ko(\Si)$
to each oriented surface $\Si$. This category satisfies a universal property
which guarantees the existence of a unique map to a dg enhancement of any
contact category $\Co(\Si)$, when it exists.

\begin{upropertyy} If $\eX$ is a pretriangulated dg category for which there are choices of maps $\th : \ga \to \ga'$ corresponding to bypass moves between dividing sets $\ga,\ga'\subset\Si$ and these maps satisfy four properties:
\begin{enumerate}
\item Bypass moves are cycles.
\item Trivial bypass moves are equal to identity.
\item Disjoint bypass moves commute.
\item Associated to each bypass move is an exact triangle of the form:
\begin{center}
\begin{tikzpicture}[scale=10, node distance=2.5cm]
\node (A2) {\CPic{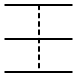}};
\node (B2) [right=1cm of A2] {};
\node (C2) [right=1cm of B2] {\CPic{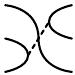}};
\node (D2) [below=1.41421356cm of B2] {\CPic{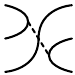}};
\draw[->, bend left=35] (A2) to node {$\theta_A$} (C2);
\draw[->, bend left=35] (C2) to node {$\theta_B$} (D2);
\draw[->, bend left=35] (D2) to node {$\theta_C$} (A2);
\end{tikzpicture}
\end{center}
\end{enumerate}
then there is a unique map $\KoS \to \eX$ in the homotopy category of
differential graded categories. See Section \ref{formalcatsec} for details.
\end{upropertyy}

Section \ref{algsec} contains algebraic background necessary to produce and
study the categories $\KoS$. The definition of pre-triangulated hull and a
review of Drinfeld-To\"en localization construction for dg categories is
included. A variation of this localization construction is introduced and
related to the standard localization.

Section \ref{formalcatsec} contains a discussion of surface topology needed
for the main construction. The construction of the formal contact categories
$\KoS$ follows immediately by combining these topological considerations and
the localization construction introduced in Section \ref{algsec}. The
remainder of the paper is dedicated to the study of formal contact
categories.

In Section \ref{propertiessec}, we check that the categories satisfy several
elementary properties which were outlined by K. Honda. In particular,
Corollary \ref{tightcor} shows that non-trivial boundary conditions are
necessary for Giroux's tightness criterion to be satisfied. Theorem
\ref{surfacetheorem} shows that when such boundary conditions are present,
the triangulated structure allows one to simplify the category by writing
dividing sets which do not interact with the boundary in terms of those
which do, up to homotopy equivalence. In Section \ref{possec}, formal
contact categories $\KoS$ are split into a product of two isomorphic copies
of a subcategory $\Ko_+(\Si)$ called the positive half of the formal contact
category.

In Section \ref{mcgMainsec}, Theorem \ref{mcgthm} shows that the mapping
class group $\Ga(\Si)$ of $\Si$ acts naturally on the category
$\KoS$. Theorem \ref{zarevgenthm} shows that when the surface $\Si$ supports
a handle decomposition, determined by an arc parameterization $\Ad$, this
produces a collection of generators $\Zi(\Ad)$ for the category $\KoS$.
After proving the second statement above, in Section \ref{kodecatsec} we
study additive invariants of $\Ko_+(\Si_{g,1},2)$.
%% Conjecture \ref{kththm} uses this
%% result to identify the Grothendieck group $K_0(\Ko_+(\Si_{g,1},2))$ of the
%% positive half of the formal contact category with the exterior algebra
%% $\Lambda^*H_1(\Si_{g,1})$ of the first homology group $H_1(\Si_{g,1})$.

The remainder of the paper is dedicated to an investigation of the
comparison problem between two codimension 2 extensions: contact categories and
Heegaard-Floer categories. The strategy pursued is illustrated by the diagram below:
$$\begin{tikzpicture}[scale=10, node distance=2.5cm]
\node (A1) {$\alg(-\Ad)\module$};
\node (X) [left=1.25cm of A1] {};
\node (B1) [left=1.25cm of X] {$\Ko(\Si)$};
\node (C1) [below=1.25cm of X] {$\Co(\Si)$};
\draw[->] (B1) to node {} (A1);
\draw[->] (A1) to node {} (B1);
\draw[->,dashed] (C1) to node   {} (A1);
\draw[->,dashed] (B1) to node  {} (C1);
\end{tikzpicture}$$
When a reasonable candidate for the geometric contact category $\Co(\Si)$ exists, the dashed lines should be taken to be solid.

In Section \ref{tiansec}, we study the relationship between three categories
associated to the disk $(D^2,2n)$ with $2n$ points fixed along its
boundary. In \cite{YT1}, Y. Tian constructed a candidate $\Yi_n$ for
$\Co(D^2,2n)$ and we introduce an arc parameterization $\zigzag_{n}$ of the
disk $(D^2,2n)$ which gives a dg category $\alg(-\zigzag_n)$ associated to
the Heegaard-Floer package \cite{Zarev1}. The main result of this section is
to prove that the three dg categories are Morita equivalent:
\begin{equation}
\Co(D^2,2n) \cong \Ko_+(D^2, 2n) \cong \alg(-\zigzag_n).
\end{equation}

The category $\alg(-\zigzag_n)$ is a $k$-linear category because the differential
$d$ is always equal to zero.  There are several other instances in which 
categories with this property can be associated to surfaces. In
Section \ref{BFalgsec}, we show that functors from these categories to the
homotopy categories of the appropriate formal contact categories can be
defined.

Section \ref{geosec} applies the universal property, discussed above, in a
much broader context.  The section begins with a discussion of the
relationship between the formal contact categories $\KoS$ and the contact
categories $\Co(\Si)$. The main theorem leverages the universal property to
construct a map:
$$\Ko_+(\Si) \to \alg(-\Ad)\module$$
in the homotopy category of dg categories, from the formal contact category
associated to $\Si$ to the Heegaard-Floer category associated to $\Si$, when
$\Si$ is parameterized by $\Ad$, for any oriented surface $\Si$ with
sufficient boundary conditions.

\subsection*{Acknowledgments}
The construction of contact categories was inspired by the ideas of
K. Honda, Y. Tian and K. Walker \cite{KoHo, YT1, YT2, Walker2}.  The author
would especially like to thank Y. Tian for his helpful correspondence and
the Simons Center for facilitating our discussion. Also Y. Huang,
R. Lipshitz, A. Manion and I. Petkova for several helpful emails. My colleagues
C. Frohman, A. Kaloti, K. Kawamuro and R. Kinser for their cordiality and
their conversation. 

In more detail, the author's involvement in this subject began because of
the mention of a contact TQFT in \S 9.4 of K. Walker's 2006 notes
\cite{Walker2}.  These notes were discussed at length with J. Roberts
between 2006-2009.  Several years later the author spoke to K. Walker about
the possibility of fitting Heegaard Floer theory into his framework in
\cite{Walker2}.  In 2014 the author found Y. Tian's papers \cite{YT1, YT2}
and a recording of K. Honda at MSRI discussing his ideas \cite{KoHo}.  This
project began after several conversations with A. Kaloti.

The author would like to thank the referees for their readings and very
helpful feedback.

The author would also like to thank the organizers and participants of the ``Categorifications of Quantum Groups, Representations and Knot Invariants'' session at the AMS-EMS-SPM meeting in June 2015 where some of these results were announced.

After this paper was posted, a few papers with complementary results have
appeared, see \cite{HondaTian} and \cite{Mathews}.

\newpage
\section{Algebraic constructions}\label{algsec}
\newcommand{\OT}{\th_{1,2}}
\renewcommand{\TT}{\th_{2,3}}
\newcommand{\tTT}{\tilde{\th}_{2,3}}
\newcommand{\TO}{\th_{3,1}}

\newcommand{\tS}{{\tilde{S}}}
\newcommand{\Tp}{D'}
\newcommand{\DI}{\bar{D}}
\newcommand{\cDI}{\widetilde{D}}
\newcommand{\un}{\normaltext{un}}
\newcommand{\dtri}{\DI}
\newcommand{\cdtri}{\cDI}

In this section a discussion of localizations follows a review of
pretriangulated hulls.  Section \ref{tolocsec} reviews the standard
localization procedure for dg categories. Section \ref{postsec} introduces a
form of localization which creates formal extensions among objects in a dg
category: rather than creating homotopy equivalences amongst objects, this
{\em Postnikov} localization introduces distinguished triangles. In Section
\ref{proppostsec}, properties of Postnikov localizations are discussed.

Most of the materials in this section are standard. Some review is found in
the Appendix \ref{dgcatsec}. A review of differential graded categories can
be found in \cite{Kellerdg, Toen} or \cite[\S 1]{DK}; consult \cite{TabAdd,
  Tabuada, T} for technical details. The language of model categories is
reviewed in the reference \cite[\S A.2]{LT}, more details can be found in
the references \cite{Hovey, Quillen}.

\subsection{Pretriangulated hull}\label{trisec}
This section contains a brief discussion of pretriangulated hulls of dg
categories. The key ideas were introduced in \cite[\S 4]{BK}; see also
\cite{BV,D}.

\begin{defn}{(\cite[\S 2.4]{D})}\label{ptdef}
If $\eC$ is a dg category then there exists a dg category $\eC^\pt$ called the {\em pretriangulated hull} of $\eC$. The objects of $\eC^\pt$ are one-sided twisted complexes; i.e. formal expressions
$$x = (\bigoplus_{i=1}^n x_i[r_i],p)\conj{ such that  } dp + p^2 = 0$$
and  $n\geq 0, x_i \in\Obj(\eC)\cup\{0\},r_i\in \ZZ$. The map $p = (p_{i,j})$ is a matrix such that $\vert p_{i,j} \vert = 1$ and
$$p_{i,j} = \left\{ \begin{array}{ll} x_i[r_i] \to x_j[r_j]  & j > i \\ 0 & j \leq i \end{array}\right.$$
If $x,x'\in\Obj(\eC^\pt)$ so that $x = (\bigoplus_{i=1}^n x_i[r_i],p)$ and $x' = (\bigoplus_{i=1}^n x'_i[r'_i],p')$ then $\Hom(x,x')$ consists of matrices $f = (f_{i,j})$, $f_{i,j}\in\Hom^{r'_j-r_i}(x_i,x'_j)$, the composition is given by  matrix multiplication and the differential $d : \Hom(x,x') \to \Hom(x,x')$ is determined by the formula: 
$$(df)_{i,j} = (df)_{i,j} + (p' f)_{i,j} - (-1)^{\vert f_{i,j} \vert} (f p)_{i,j}.$$
\end{defn}

\begin{remark}{(\cite[\S 2.4]{D})}\label{conerem}
If $x,y\in\Ob(\eC)$ and $f : x\to y$ is a closed map of degree zero then the cone
of $f$ exists in $\eC^{\pt}$ by construction: $C(f) = (x \opp y[-1], p)\in \Ob(\eC)$ where
$p_{1,2} = f$ and $p_{1,1} = p_{2,1} = p_{2,2} = 0$.
The objects in $\eC^\pt$ can be obtained by iterated applications of the cone construction.
\end{remark}

A referee notes that the above construction in Remark \ref{conerem} be called ``cocone.''

By construction, the pretriangulated dg category $\eC^{\pt}$ associated to
a $k$-linear category $\eC$ factors through its additive closure
$\Mat(\eC)$: 
$$\Mat(\eC)^\pt \cong \eC^{\pt}.$$
(Or set $p=0$ in Def. \ref{ptdef} above.)
The canonical inclusion $\eC \hookrightarrow \eC^{\pt}$ is fully faithful. A
dg category $\eC$ is {\em pretriangulated} when the functor
$\Ho(\eC) \to \Ho(\eC^\pt)$ induced by inclusion between the associated
homotopy categories is an equivalence of categories. The {\em category of pretriangulated
  dg categories} will be denoted by $\dgcat^\pt$.

Unfamiliar readers may wish to recall that $\Ob(\eC\coprod\eD) := \Ob(\eC)\sqcup \Ob(\eD)$ and
$$\Hom_{\eC\coprod \eD}(x,y) := \left\{ \begin{array}{ll}
\Hom_\eC(x,y) & \normaltext{ if } x,y\in \Ob(\eC)\\
\Hom_\eD(x,y) & \normaltext{ if } x,y\in \Ob(\eD)\\
0 & \normaltext{ otherwise }
\end{array}
\right.$$

The proposition below shows how the pretriangulated hull operation
distributes over coproducts of dg categories. This is a $p\ne 0$
generalization of the analogous statement about additive closures.  It will be
used in Theorem \ref{splitthm}.

\begin{prop}\label{ptmultprop}
If  $\eC, \eD$ are $k$-linear then $(\eC \coprod \eD)^\pt \cong \eC^\pt \prod \eD^\pt$.
\end{prop}
\begin{proof}
Since there are no non-zero maps between $\eC$ and $\eD$, thought of as subcategories of $\eC\coprod \eD$,
a twisted complex $(\bigoplus_{i=1}^n x_i[r_i],p) \in (\eC \coprod \eD)^\pt$ splits into a direct sum of twisted complexes in $\eC^{\pt}$ and $\eD^{\pt}$ respectively. Likewise, matrices $(f_{i,j})$ of maps between twisted complexes in $(\eC \coprod \eD)^\pt$ consist of blocks. It follows that there are functors $\pi_\eC : (\eC \coprod \eD)^\pt \to \eC^\pt$ and $\pi_\eD : (\eC \coprod \eD)^\pt \to \eD^\pt$ which satisfy the universal property of the product.
\end{proof}

The following proposition is well-known, see \cite[\S 1.5]{BV}.

Many of the constructions to follow in this section use ideas which are touched on in Appendix \ref{dgcatsec}.

\begin{prop}\label{pretriprop}
The pretriangulated hull $-^\pt : \dgcat \to \dgcat^\pt$ is left adjoint to the forgetful functor:
$$\Hom_{\dgcat^\pt}(\eC^{\pt},\eD) \cong \Hom_{\dgcat}(\eC,\Forget(\eD)).$$
If $f : \eC \xto{\sim} \eD$ is a quasi-equivalence then $f^\pt : \eC^\pt \to \eD^\pt$ is a quasi-equivalence of dg categories.
\end{prop}

The category $\Hqe$ is a localization of $\dgcat$ in which
quasi-equivalences between dg categories are isomorphisms.  The Morita
homotopy category $\Hmo$ is a localization of the homotopy category $\Hqe$
of dg categories in which derived equivalences are isomorphisms. In $\Hmo$,
the homotopy idempotent completion $\eC^\perf$ of the pretriangulated hull
$\eC^\pt$ is fibrant replacement, see \cite{TabAdd}.

\subsection{Inverting maps in dg categories}\label{tolocsec}
This section contains a brief review of the localization construction for dg
categories. Many authors have studied this problem, see \cite{D,Kellerder,
  Kellerdg, TabuadaLoc} and \cite[\S 8.2]{T}.

\begin{definition}\label{intcatdef}
The symbol $I$ will be used to denote the dg category freely generated
by a cycle $f : 1\to 2$ of degree $0$ and $I'$ will be used to denote the dg category freely generated by cycles $f : 1\to 2$ and $g : 2\to 1$ of degree $0$.
$$I = 1 \xrightarrow{f} 2 \conj{ and } I' = 1 \rightleftarrows 2$$ 
The symbol $\bar{I}$ denotes the dg category with a unique degree zero isomorphism
$f : 1\xto{\sim} 2$ with $df = 0$. There are canonical inclusions:
$$\kappa : I \hookrightarrow \bar{I} \conj{ and } \kappa' : I' \hookrightarrow \bar{I}.$$
These maps are determined by the assignments $\kappa(f) = f$, $\kappa'(f) = f$ and $\kappa'(g) = f^{-1}$.
\end{definition}

\begin{definition}\label{toenlocdef}
Suppose that $\eC$ is a dg category and $R : \coprod_{r\in \eR} I \to \eC$ is a dg functor. Then the {\em localization of $\eC$ with respect to $R$} is a dg functor:
$$P : \eC \to L_R\eC$$
which satisfies properties (1) and (2) below.
\begin{enumerate}
\item The pullback map $P^* : \Hom_{\Hqe}(L_R\eC, \eX) \to \Hom_{\Hqe}(\eC,\eX)$ is injective.
\item The image of the map $P^*$ consists of maps $f : \eC \to \eX$ for which there is a map $\a$ making the diagram below commute.
$$\begin{tikzpicture}[scale=10, node distance=2.5cm]
\node (A1) {$\coprod_{r\in \eR} Ho(I)$};
\node (B1) [right=2.5cm of A1] {$\Ho(\eX)$};
\node (C1) [below=1cm of A1] {$\coprod_{r\in \eR} \Ho(\bar{I})$};
\draw[->] (A1) to node {$\Ho(R^* f)$} (B1);
\draw[->] (A1) to node [swap] {$\Ho(\kappa)$} (C1);
\draw[->,dashed] (C1) to node [swap] {$\a$} (B1);
\end{tikzpicture}$$
\end{enumerate}
The image $\im(P^*)$ may be denoted by $\Hom_\Hqe^I(\eC,\eX)$.
\end{definition}

Corollary 8.8 in \cite{T} shows that for any dg category $\eC$ and any functor
$R : \coprod_{r\in \eR} I \to \eC$, there exists a functor
$P : \eC \to L_R\eC$ in the homotopy category $\Hqe$ of dg categories which
satisfies the two properties in Definition \ref{toenlocdef} above. The
functor $P : \eC \to L_R\eC$ is defined to be the homotopy pushout:
\begin{center}
\begin{tikzpicture}[scale=10, node distance=2.5cm]
\node (A1) {$\coprod_{r\in \eR} I$};
\node (B1) [right of=A1] {$\eC$};
\node (C1) [below of=A1] {$\coprod_{r\in \eR} \bar{I}$};
\node (D1) [right of=C1] {$L_R\eC.$};
\draw[->] (A1) to node {$R$} (B1);
\draw[->] (B1) to node {$P$} (D1);
\draw[->] (A1) to node [swap] {$\kappa$} (C1);
\draw[->] (C1) to node {} (D1);
\end{tikzpicture}
\end{center}

When the category $\eC$ is cofibrant, this homotopy pushout
$$L_R \eC = \underset{r\in\eR}{\amalg} \bar{I} \coprod_{R}^{\mathbb{L}} \eC$$ 
can be computed by replacing the inclusion $\kappa : I \hookrightarrow \bar{I}$ by a well-known cofibration $I \hookrightarrow \dr$. The dg category $\dr$ appears in Drinfeld where it is denoted by $\aK$ \cite[\S 3.7.1]{D}.

\begin{definition}\label{Kdef}
  The category $\dr$ has two objects: $1$ and $2$. Its maps are generated by the elements:
  $f\in\Hom_{\dr}^0(1,2)$, $g\in\Hom_{\dr}^0(2,1)$,
  $h_{1,1}\in\Hom_{\dr}^{-1}(1,1)$,
  $h_{2,2}\in\Hom_{\dr}^{-1}(2,2), h_{1,2}\in\Hom_{\dr}^{-2}(1,2)$:
\begin{center}
\begin{tikzpicture}[->,>=stealth',shorten >=1pt,auto,node distance=3cm]

  \node (1) {$1$};
  \node (2) [right of=1] {$2$};

  \path
    (1) edge [loop left] node {$h_{1,1}$} (1)
        edge [bend left] node  {$f,h_{1,2}$} (2);
  \path
    (2) edge [loop right] node {$h_{2,2}.$} (2)
        edge [bend left] node  {$g$} (1);
\end{tikzpicture}
\end{center}
The differential is determined by the Leibniz rule together with the equations:
$$df=0,\quad dg=0,\quad dh_{1,1} = gf - 1_1,\quad dh_{2,2} = fg - 1_2\quad \normaltext{ and }\quad dh_{1,2} = h_{2,2}f - f h_{1,1}.$$
and the maps are subject to no relations.
\end{definition}

\begin{remark}
In Definition \ref{toenlocdef}, the category $I$ and the map $\kappa : I \hookrightarrow \bar{I}$ can be replaced by the category $I'$ and the map $\kappa' : I' \hookrightarrow \bar{I}$.
Suppose that $R : I \to \eC$ and a candidate $R(f)^{-1}$ for the inverse of the map $R(f)$ already exists in the category $\eC$. Then $R$ can be extended to a functor $R' : I' \to \eC$ so that $R'(f) = R(f)$ and $R'(g) = R(f)^{-1}$ and there is an analogous localization:
$$ P : \eC \to L_{R'}\eC \conj{ where } L_{R'}\eC =  \dr \coprod^{\LL}_{R'} \eC.$$
\end{remark}

\subsection{Postnikov localization}\label{postsec}
A variation of the localization procedure discussed in the previous section
is introduced. This {\em Postnikov localization} introduces distinguished
triangles rather than homotopy equivalences. In particular, given a sequence
$$1 \to 2\to 3 \to 1$$
of maps $S$ in a dg category $\eC$, there is a dg category $L_S\eC$ in which
this sequence forms a distinguished triangle. 

The dg categories considered in this section are $\ZZ/2$-graded for
simplicity.  The the equivalences discussed below commute with the forgetful
functor to the ungraded setting introduced in Section \ref{grading}. On the
other hand, $\ZZ$-graded lifts determined by grading conventions for distinguished triangles can be found in \cite[\S 2.4.1]{DK}. See also \cite[\S 4.3]{Toen}.

Historically, Postnikov systems appear in the study of triangulated
categories \cite{GM}.  The name Postnikov may be attached to that
construction because it is a generalization of the Postnikov decomposition
of topological spaces to algebraic triangulated categories.

First we introduce a dg category $\Tp$ which corepresents triangles, see
Equation \eqref{trieq1} and Proposition \ref{symtriprop}. Then Definition
\ref{didef} introduces dg categories $\bar{D}$ and $\cdtri$ which
corepresent distinguished triangles. A dg functor
$\kappa : \Tp \hookrightarrow \cdtri$ will be used to construct the
Postnikov localization in Definition \ref{anlocdef}.

\begin{defn}\label{Ttridef}
  The symbol $\Tp$ will be used to denote the dg category freely generated by
  cycles: $\OT : 1\to 2$, $\TT : 2\to 3$ and $\TO : 3\to 1$.
$$\begin{tikzpicture}[scale=10, node distance=2.5cm]
\node (A1) {$1$};
\node (X) [right=1.25cm of A1] {};
\node (B1) [right=1.25cm of X] {$2$};
\node (C1) [below=1.25cm of X] {$3$};
\draw[->] (A1) to node {$\OT$} (B1);
\draw[->] (C1) to node  {$\TO$} (A1);
\draw[->] (B1) to node  {$\TT$} (C1);
\end{tikzpicture}$$
The degrees are determined by $\vnp{\OT}=1$,$\vnp{\TT} = 1$ and $\vnp{\TO} = 1$.
\end{defn}

Since a dg functor $f : \Tp \to \eC$ is uniquely determined by where it maps
the generators in the definition above, there is a bijection between the set
of such functors and (symmetric) triangles in $\eC$.
\begin{equation}\label{trieq1}
\Hom_{\dgcat}(\Tp,\eC)  \xto{\sim} \{\normaltext{ symmetric triangles in $\eC$ } \}
\end{equation}

\begin{defn}
  If $f, g : \Tp \to \eC$ are two triangles in $\eC$ then {\em $f$ is
    isomorphic to $g$} when $\Ho(f) \cong \Ho(g)$ as objects in the functor
  category $\Hom(\Ho(\Tp),\Ho(\eC))$.
\end{defn}

The proposition below states that in the homotopy category $\Hqe$ of dg
categories the lefthand side of Equation \eqref{trieq1} above is in
canonical bijection with isomorphism classes of triangles.

\begin{prop}{(\cite[Prop. 2.4.7]{DK})}\label{symtriprop}
  For any dg category $\eC$, there is a one-to-one correspondence between
  homotopy classes of functors $f : \Tp \to \eC$ and isomorphism classes of
  triangles in $\eC$:
$$\Hom_{\Hqe}(\Tp,\eC)  \leftrightarrow \{\normaltext{ symmetric triangles in $\eC$ } \}/iso.$$
\end{prop}

Just as isomorphisms are distinguished types of maps, distinguished
triangles are distinguished types of triangles. A distinguished triangle is
a recipe for constructing one of its objects in terms of the other two.

\begin{defn}\label{disttridef}
If $S$ is a symmetric triangle $1 \xto{\OT} 2 \xto{\TT} 3 \xto{\TO} 1$ in a dg category $\eC$ then $S$ is a {\em distinguished triangle} if and only if $S$ is isomorphic to the distinguished triangle $S'$ given by $1 \xto{\OT} 2 \to C(\OT) \to 1$ in the homotopy category of $\eC^\pt$.
\end{defn}

In keeping with Section \ref{tolocsec}, the distinguished property
of triangles is formulated as a lifting problem. An innocuous looking dg
category $\DI$ which corepresents distinguished triangles and a
quasi-equivalent cofibrant replacement $\cDI\xto{\sim} \DI$ are introduced below.

\begin{defn}{(\cite[\S 2.4.1]{DK})}\label{didef}
  The dg category $\DI$ consists of objects $\Obj(\DI)= \{1,2,3\}$. The maps are generated by
  cycles: $\OT : 1\to 2$, $\TT : 2\to 3$ and $\TO : 3\to 1$, of degree $1$ and homotopies $h_{2,1} : 2\to 1$, $h_{3,2} : 3\to 2$ and $h_{1,3} : 1 \to 3$ of degree $1$
$$\begin{tikzpicture}[scale=10, node distance=2.5cm]
\node (A1) {$1$};
\node (X) [right=2.5cm of A1] {};
\node (B1) [right=2.5cm of X] {$2$};
\node (C1) [below=2.5cm of X] {$3$};
\draw[->,bend left=10] (A1) to node {$\OT$} (B1);
\draw[->,bend left=10] (C1) to node  {$\TO$} (A1);
\draw[->,bend left=10] (B1) to node  {$\TT$} (C1);
\draw[->,bend left=10] (B1) to node {$h_{2,1}$} (A1);
\draw[->,bend left=10] (A1) to node  {$h_{1,3}$} (C1);
\draw[->,bend left=10] (C1) to node  {$h_{3,2}$} (B1);
\end{tikzpicture}$$
with $dh_{2,1} = \th_{3,1}\th_{2,3}$, $dh_{3,2} = \th_{1,2}\th_{3,1}$ and $dh_{1,3} = \th_{2,3}\th_{1,2}$
and the relations:
$$\th_{2,3} h_{3,2} + h_{1,3}\th_{3,1} = 1_3, \quad \th_{1,2} h_{2,1} + h_{3,2}\th_{2,3} = 1_2, \quad \th_{3,1} h_{1,3} + h_{2,1}\th_{1,2} = 1_1.$$

The dg category $\cDI$ consists of objects $\Obj(\cDI)= \{1,2,3\}$. The maps $\th_{i,j} : i\to j$ in this category are clockwise-oriented paths between vertices, from $i$ to $j$, in the triangular graph featured in Definition \ref{Ttridef} above.

The differential is zero on paths of length zero or one, when $\th_{i,i}$ is a cycle, a path of topological degree one (a loop), then
$$d\th_{i,i} = 1_i - \sum_k \th_{k,i}\th_{i,k}$$
otherwise $d\th_{i,j}$ is the sum over compositions of all possible factorizations of the path:
$$d\th_{i,j} =  \sum_k \th_{k,i}\th_{j,k}.$$
The projection $p :  \cDI \to \DI$ given by mapping cycles of length $1$ to their respective $\th$-maps is a quasi-equivalence \cite[Prop. 2.4.13]{DK}.
In the other direction, there is an inclusion $\kappa' : \Tp \hookrightarrow \cDI$ given by sending the $\th$-maps to their respective length $1$ cycles. There is also an inclusion $\kappa' : \Tp\hookrightarrow \DI$ given by the same formula. A $\ZZ$-graded analogue of $\cDI$ is discussed in \cite{K}. 
This dg category is the Cobar-Bar construction on the partially wrapped Fukaya category of the disk with three stops \cite{Nadler}.
\end{defn}

The proposition below states that the dg category $\cDI$ corepresents
distinguished triangles and satisfies the key properties necessary for the
localization construction.

\begin{prop}{(\cite[Prop. 2.4.14]{DK})}\label{anrelprop}
\begin{enumerate}
\item  For any dg category $\eC$ the set of homotopy classes of dg functors from
  $\cDI$ to $\eC$ is in bijection with the set of isomorphism classes of
  distinguished triangles in $\eC$:
$$\Hom_{\Hqe}(\cDI,\eC) = \{ 1\xto{\th_{1,2}} 2 \xto{\th_{2,3}} C(\th_{1,2}) \to 1 \}/iso.$$

\item The image of the pullback induced by the map $\kappa'$ appearing in Definition \ref{didef} coincides with the subset of triangles which are distinguished:
$$(\kappa')^* : \Hom_{\Hqe}(\cDI, \eC) \to \Hom_{\Hqe}(\Tp, \eC).$$

\item The set $\Hom_{\Hqe}(\cDI,\eC)$ is equal to the set of maps $f \in \Hom_{\Hqe}(\Tp,\eC)$ for which there is a map $\a : \Ho(\cDI) \to \Ho(\eC)$ such that $\Ho(f) = \a \circ \Ho(\kappa')$.
\end{enumerate}
\end{prop}

We are now ready to discuss a generalization of the localization procedure
presented earlier in Section \ref{tolocsec}. Instead of inverting maps in
the associated homotopy category, this new operation creates distinguished
triangles in the associated homotopy category.

\begin{definition}\label{anlocdef}
  Suppose that $\eC$ is a dg category and
  $S : \coprod_{s\in \eS} \Tp \to \eC$ is a dg functor. Then the {\em Postnikov localization of $\eC$ with respect to $S$} is a dg functor:
$$Q : \eC \to L_S\eC$$
such that for any dg category $\eX$ the following properties are satisfied.
\begin{enumerate}
\item The pullback map $Q^* : \Hom_{\Hqe}(L_S\eC, \eX) \to \Hom_{\Hqe}(\eC,\eX)$ is injective and
\item The set of maps $\Hom_{\Hqe}(L_S\eC, \eX)$ in the image of $Q^*$ is equal to the set of maps $f \in \Hom_{\Hqe}(\eC,\eX)$ such that there is a map $\a$ making the diagram below commute.
$$\begin{tikzpicture}[scale=10, node distance=2.5cm]
\node (A1) {$\coprod_{s\in \eS} Ho(\Tp)$};
\node (B1) [right=2.5cm of A1] {$\Ho(\eX)$};
\node (C1) [below=1cm of A1] {$\coprod_{s\in \eS} \Ho(\cDI)$};
\draw[->] (A1) to node {$\Ho(f\circ S)$} (B1);
\draw[->] (A1) to node [swap] {$\kappa'$} (C1);
\draw[dashed,->] (C1) to node [swap] {$\a$} (B1);
\end{tikzpicture}$$
The image $\im(Q^*)$ may also be denoted by $\Hom_{\Hqe}^T(\eC,\eX)$.
\end{enumerate}
\end{definition}

Recall from above that a functor from $ S : \Tp \to \eC$ is determined by the
choice of cycles $f : 1 \to 2$, $g : 2 \to 3$ and $h : 3 \to 1$. The
Postnikov localization $L_S\eC$ associated to the functor $S$ requires that
the sequence:
$$1\xto{f} 2\xto{g} 3 \xto{h} 1$$
is a distinguished triangle in the sense of Definition \ref{disttridef}. The category
$L_S\eC$ is uniquely determined up to homotopy by the property that a
functor $f : \eC\to \eX$ factors through $Q : \eC \to L_S\eC$ in $\Hqe$ when
it maps triangles in the image of $S$ to distinguished triangles in the homotopy
category $\Ho(\eX)$ of $\eX$.

When $\eC$ is a cofibrant dg category, the category $L_S\eC$ is a pushout,
obtained by gluing a copy of $\cDI$ along the subcategory determined by the
image of a functor $S$.  If $\eC$ is not cofibrant then $L_S\eC$ is a
homotopy pushout: the pushout of a cofibrant replacement
$\widetilde{\eC}\xto{\sim}\eC$ of $\eC$ \cite[\S A.2.4.4]{LT}.

The next proposition states that Postnikov localizations always exist.

\begin{prop}\label{anlocexprop}
  For any dg category $\eC$ and any collection
  $S : \coprod_{s\in \eS} \Tp \to \eC$, there is a Postnikov localization
  $Q : \eC \to L_S\eC$ in $\Hqe$.
\end{prop}
\begin{proof}
  It follows from Proposition \ref{anrelprop} that the functor
  $\kappa' : \Tp \to \cDI$ is a Postnikov localization in the sense of
  Definition \ref{anlocdef}. Therefore, any coproduct of inclusions:
  $\coprod_{s\in \eS} \Tp \to \coprod_{s\in\eS} \cDI$, is an Postnikov localization.  For any
  dg category $\eC$, the localization $Q : \eC \to L_S\eC$ is given by the
  homotopy pushout:
\begin{center}
\begin{tikzpicture}[scale=10, node distance=2.5cm]
\node (A1) {$\coprod_{s\in \eS} \Tp$};
\node (B1) [right of=A1] {$\eC$};
\node (C1) [below of=A1] {$\coprod_{s\in \eS} \cDI$};
\node (D1) [right of=C1] {$L_S\eC$};
\draw[->] (A1) to node {$S$} (B1);
\draw[->] (B1) to node {$Q$} (D1);
\draw[->] (A1) to node {} (C1);
\draw[->] (C1) to node {} (D1);
\end{tikzpicture}
\end{center}
That $L_S\eC$ is a Postnikov localization follows Proposition \ref{anlocdef}
and properties of homotopy pushouts \cite{Hovey}.
\end{proof}

\subsection{Properties of Postnikov localization}\label{proppostsec}

In this section we explore properties of the Postnikov localization
procedure, establish a relationship between it and the ordinary localization of dg
categories, and introduce an analogue of Heller's lemma which facilitates the
computation of additive invariants such as the Grothendieck group.

\subsubsection{Triangle insertion}\label{insdelsec}

The appendix  \S\ref{dgcatsec} reviews relevant concepts such quasi-fully faithful embedding.

The first proposition below assures us that, after having added a triangle,
it persists in the pretriangulated hull.

\begin{prop}\label{prestriprop}
Suppose that $S : \Tp \to \eC$, $Q : \eC \to L_S\eC$ and $R : L_S\eC \to \eX$ is a quasi-fully faithful embedding of the Postnikov localization of $\eC$ into a pretriangulated category $\eX$. If $f = RQS(1\to 2)$ and $c = RQS(3)$ then  $c$ is isomorphic to the cone $C(f)$ of $f$ in the homotopy category of $\eX$.
$$c \cong C(f) \conj{ in } \Ho(\eX).$$
\end{prop}
\begin{proof}
  For the sake of notation, everything to follow takes place inside of the
  category $\Ho(\eX)$.  By TR3 there is a map $h$ in $\eX$ which yields a
  map $(1,1,h,1)$ from the triangle $S(1)\to S(2) \to S(3)\to S(1)$ to the triangle 
  $S(1) \to S(2) \to C(f) \to S(1)$. For all $x\in\eX$, both triangles determine long
  exact sequences after applying the functor $\Hom(x,-)$. By the Five Lemma
  $h_* : \Hom^*(x, c)\to \Hom^*(x,C(f))$ is an isomorphism. Therefore,
  Yoneda's lemma implies the result.
\end{proof}

\subsubsection{Decategorification of localizations}\label{decatsec}
\newcommand{\kr}{K}

For references concerning short exact sequences of dg categories see \cite[\S 4.6]{Kellerdg}. 

\begin{lemma}\label{clem}
Suppose that $S : \Tp \to \eC$ is a triangle, $\th_{1,2} = S(1\to 2)$ and $c = S(3)$ in a dg category $\eC$. Then
  $S$ is isomorphic to a distingished triangle if and only if the double cone complex $\kr = C(C(\th_{1,2}) \xto{\tilde{\th}_{2,3}} c)$ is contractible where $\tilde{\th}_{2,3}$ is the extension of the map $\th_{2,3} : S(2)\to c$ to the cone $C(\th_{1,2})$.
\end{lemma}
\begin{proof}
If $S$ is distinguished then the triangle $S(1)\to S(2) \to S(3)\to S(1)$ is isomorphic to $1\to 2\to C(\th_{1,2}) \to 1$ in the homotopy category via the map $(1,1,\tilde{\th}_{2,3},1[1])$, so $\tilde{\th}_{2,3}$ is a homotopy equivalence and $C(\tilde{\th}_{2,3})$ is contractible. Conversely, $C(\tilde{\th}_{2,3}) \simeq 0$ implies $\tilde{\th}_{2,3}$ is a homotopy equivalence and the map above determines an equivalence of triangles.
\end{proof}

Recall that if $a\in\Obj(\eC)$ then Drinfeld's dg quotient $\eC/\inp{a}$ can
be formed by adding a homotopy $h$ which satisfies $dh = 1_a$ to a cofibrant
replacement of $\eC$, see \cite{D}. This makes the object contractible in
the homotopy category of the Drinfeld quotient. (This can be reformulated as
a homotopy pushout \cite[Thm. 4.0.1]{TabuadaLoc}.)

The proposition below constructs a short exact sequence of dg categories by
relating the Postnikov localization $L_S\eC$ of a dg category $\eC$ to a
Drinfeld quotient $\eC/\inp{\kr}$. The subcategory $\inp{\kr}$ is generated
by the object $\kr$ in Lemma \ref{clem} above.

\begin{prop}\label{sesprop}
Suppose that $S : \Tp \to \eC$ is a triangle, $f = S(1\to 2)$ and $c = S(3)$ in a dg category $\eC$. Then there is a short exact sequence of dg categories:
$$\inp{\kr} \to \eC \to L_S(\eC)$$
in the Morita homotopy category $\Hmo$, where $\inp{\kr}$ is the dg category determined by the cone $\kr = C(C(f) \to c)$ of the natural map from the cone on $f$ to $c$ in $\eC^\pt$.
\end{prop}
\begin{proof}
First assume that $\kr$ is represented by an object in $\eC$. By Definition \ref{anlocdef}, the Postnikov localization $L_S\eC$ satisfies the universal property,
\begin{equation}\label{eq:A}
\Hom_{\Hqe}(L_S\eC, \eX) \xto{\sim} \Hom^T_{\Hqe}(\eC,\eX),
\end{equation}
the set of homotopy classes of functors from $L_S\eC$ to any dg category $\eX$ is in bijection with the set of homotopy classes of functors $f : \eC\to \eX$ which map $\im(S)$ to distinguished triangles in the homotopy categories: $\Ho(f) : \Ho(\eC) \to \Ho(\eX)$. By the lemma above, the condition that $\Ho(fS) : \Tp \to \Ho(\eX)$ maps to a distinguished triangle is equivalent to the condition that a certain double cone complex $\kr$ is contractible. If $\tTT : C(\OT)\to 3$ is given by $\tTT = (0,\TT)$ then set
 $\kr = C(\tTT)$ so that:
$$\kr = C(\tTT) = (1[2]\opp 2[1]\opp 3, d_\kr)\conj{ where } d_\kr = \left(\begin{array}{ccc} d_1 & \OT & \\  & -d_2 & \TT \\ &  & d_3 \end{array} \right)$$ 
is contractible in $\eX$. So there is a bijection of sets:
\begin{equation}
  \label{eq:B}
\Hom^T_{\Hqe}(\eC,\eX) \xto{\sim} \Hom^{\inp{\kr}}_\Hqe(\eC,\eX)
\end{equation}
where $\Hom^{\inp{\kr}}(\eC,\eX)$ is the set of maps $f : \eC \to \eX$ which send $\kr$ to a contractible object in $\eX$.  Since
\begin{equation}
  \label{eq:C}
\Hom_{\Hqe}(\eC/\inp{\kr},\eX) \xto{\sim} \Hom^{\inp{\kr}}_\Hqe(\eC,\eX)
\end{equation}
see \cite[Thm. 4.0.1]{TabuadaLoc}. The maps in Equations \eqref{eq:A}, \eqref{eq:B} and \eqref{eq:C} combine to show that the Postnikov localization satisfies the same universal property as the Drinfeld quotient. Therefore, $\eC/\inp{\kr}$ and $L_S\eC$ are isomorphic in $\Hqe$. Associated to any such Drinfeld quotient, there is a short exact sequence:
$$\inp{\kr} \hookrightarrow \eC \to \eC/\inp{\kr}$$
in the Morita homotopy category $\Hmo$ \cite[Rmk. 4.0.2]{TabuadaLoc}. Since $\Hmo$ is a quotient of $\Hqe$, the isomorphism $\eC/\inp{\kr} \cong L_S\eC$ in $\Hqe$ implies the isomorphism  $\eC/\inp{\kr} \cong L_S\eC$ in $\Hmo$, and there is a short exact sequence of dg categories:
$$\inp{\kr} \hookrightarrow \eC \to L_S\eC.$$

Now suppose that $\kr$ is {\em not} representable by object in $\eC$.  In the
Morita homotopy category $\Hmo$, the fibrant replacement $\eC^\perf$ of
$\eC$ is the category of perfect modules over $\eC$: an idempotent
completion of the pretriangulated hull. The object $\kr$ is representable in
$\eC^\perf$, (see Remark \ref{conerem}), and so, by the argument above, there
is a short exact sequence:
$$\inp{\kr} \to \eC^\perf \to L_S(\eC^\perf).$$

In the homotopy category of any model category, every object $\eC$ is isomorphic to
its fibrant replacement $\b : \eC \xto{\sim} \eC^\perf$. Since
cofibrations in $\Hmo$ and $\Hqe$ are identical, a homotopy pushout in
$\Hqe$ is a homotopy pushout in $\Hmo$. The map $\b$ determines an
equivalence of pushout diagrams from $\cDI \leftarrow \amalg_s \Tp \to \eC$ to
$\cDI \leftarrow \amalg_s \Tp \to \eC^{\perf}$ from which it follows that the map
$L_S\b : L_S\eC \to L_S(\eC^\perf)$ is an isomorphism in $\Hmo$.

There is a commuting diagram extending the righthand side of the short exact sequence in which
all of the vertical maps are isomorphisms in $\Hmo$.
$$\begin{tikzpicture}[every node/.style={midway},node distance=2cm]
  \matrix[column sep={4cm,between origins}, row sep={2cm}] at (0,0) {
  \node(R) {$\eC$}  ; & \node(S) {$L_S\eC$}; \\
  \node(R/I) {$\eC^\perf$}; & \node (T) {$L_S(\eC^\perf)$};\\
  };
  \draw[->] (R) -- (R/I) node[anchor=east]  {$\b$};
  \draw[->] (R) -- (S) node[anchor=south] {};
  \draw[->] (S) -- (T) node[anchor=west] {$L_S\b$};
  \draw[->] (R/I) -- (T) node[anchor=south]  {};
\end{tikzpicture}$$
So there is a short exact sequence: $E \to \eC \to L_S\eC$ where $E$ is a dg category Morita equivalent to $\inp{\kr}$.
\end{proof}

A short exact sequence of dg categories in $\Hmo$ induces a long exact
sequence among additive invariants of dg categories
\cite{Kellerdg,TabAdd}. The corollary below is the first part of the long
exact sequence associated to Hochschild homology.

\begin{cor}
Suppose that $S$, $\inp{\kr}$ and $\eC$ are as in the proposition above. Then there is an exact sequence of abelian groups:
%$$K_0(\inp{\kr}) \to K_0(\eC) \to K_0(L_S(\eC)) \conj{ and } 
$$HH_0(\inp{\kr}) \to HH_0(\eC) \to HH_0(L_S(\eC))\to 0$$
\end{cor}

\subsubsection{A Postnikov localization as a module}\label{braidgroupsec}

In this section we explain how Postnikov localizations inherit the structure of a module category over $\Endo(\cDI)$ in $\Hmo$.

If $\eC \cong L_S\eX$ is a Postnikov localization of a dg category $\eX$,
then the map $\iota : \coprod_{s\in\eS} \cDI\to \eC$ from the proof of
Proposition \ref{anlocexprop} yields a map
$\iota^{\pt} : (\coprod_{s\in\eS}\cDI)^\pt \to \eC^{\pt}$. Therefore, by
Proposition \ref{ptmultprop}, there is a map
$\hat{\iota}^\pt : \prod_{s\in\eS}\cDI^{\pt} \to \eC^\pt$. The pullback of
the map $\hat{\iota}^\pt$ along the diagonal map
$\Delta_\eS : \cDI^\pt \to \prod_{s\in\eS}\cDI^{\pt}$ is a functor:
$j : \cDI^\pt \to \eC^\pt$. The map $j$ determines an action of $\Endo(\cDI^\pt)$ on $\eC^\pt$.
\begin{center}
\begin{tikzpicture}[scale=10, node distance=2.5cm]
\node (A1) {$\cDI^\pt$};
\node (B1) [right of=A1] {$\eC^\pt$};
\node (C1) [below of=A1] {$\cDI^\pt$};
\node (D1) [right of=C1] {$\eC^\pt.$};
\draw[->] (A1) to node {$j$} (B1);
\draw[->,dashed] (B1) to node {$\bar{g}$} (D1);
\draw[->] (A1) to node [swap] {$g$} (C1);
\draw[->] (C1) to node {$j$} (D1);
\end{tikzpicture}
\end{center}
The universal property in Definition
\ref{anlocdef} gives us a lift $\bar{g}$ of $j\circ g$ for each $g\in \Endo(\cDI^\pt)$ and uniqueness of lifts implies that lifts
commute with compositions.

\subsection{Ungraded dg categories}\label{grading}
The main body of the paper will use the trivial grading, a more
sophisticated $G$-grading will be introduced at a later time \cite{C}. Here we
require $k$ to be a field of characteristic $2$.

There is a category $\Kom^\un_k$ of ungraded chain complexes.  In more
detail, An {\em ungraded chain complex} is a $k$-vector space $C$ and a
differential $d_C : C\to C$ which satisfies $d^2_C = 0$. A map $f : C \to D$
of ungraded chain complexes is a map of vector spaces. If $\Hom(C,D)$
denotes the vector space of such maps from $C$ to $D$ then there is an
associative composition and for each $C$ there is an identity map
$1_{C} : C\to C$.  This determines the category $\Kom^\un_k$.

The monoidal structure in $\Kom^\un_k$ is the tensor product; the differential is defined by:
$$d_{C\ott D} (x\ott y) = d_Cx \ott y + x \ott d_D y.$$
If $f \in \Hom(C,D)$ then the formula $df = fd_C - d_D f$ defines a
differential which makes $(\Hom(C,D),d)$ an ungraded chain complex and
$\Kom^\un_k$ is a category which is enriched over itself.

If $\Kom_k^{\ZZ/2}$ denotes the dg category of $\ZZ/2$-graded chain complexes then there is an adjunction
$$\iota : \Kom_k^\un \leftrightarrow \Kom_k^{\ZZ/2} : \rho$$
in which $\iota$ maps $(C,d)$ to the chain complex $(C_n, d_n)_{n\in \ZZ/2}$ where $C_n = C$ and $d_n = d$ for each $n\in \ZZ/2$. If $(C_n, d_n)_{n\in\ZZ/2}$ is a chain complex then $C = \oplus_n C_n$ and $d = \sum_n d_n$ determine a forgetful functor $\rho : \Kom_k^{\ZZ/2} \to \Kom_k^\un$.

An {\em ungraded dg category} $\eC$ is a category which is enriched over
$\Kom^\un_k$. The adjunction above induces an adjunction between the
category $\dgcat^\un$ of ungraded dg categories and the category
$\dgcat^{\ZZ/2}$ of $\ZZ/2$-graded categories. This extends to a Quillen
adjunction which induces model structures corresponding to $\Hqe$ and $\Hmo$
on $\dgcat^\un$, for analogous details see \cite[\S 5.1]{Dyckerhoff}.

\newpage
\section{Formal contact categories}\label{formalcatsec}
In this section, a contact category $\KoS$ is associated to each oriented
surface $\Sigma$.  The remainder of the paper will assume that $k$ is a
field of characteristic $2$ and use the trivial grading.

\subsection{Bypass moves}\label{curvessec}

In what follows surfaces will always be pointed in the sense defined below.

\begin{defn}\label{psurfdef}
  A {\em pointed surface $\Si$} is a compact connected surface $\Si$ in
  which the connected components of the boundary have been ordered and each boundary component $\partial_i\Si$ contains a marked point $z_i\in\partial_i\Si$:
$$\partial \Si = \partial_1\Si \cup \cdots \cup \partial_n \Si,\quad\quad z = \{ z_1,\ldots,z_n\} \conj{ and } z_i\in \partial_i\Si.$$
Every closed surface is canonically pointed. 

A pointed oriented surface $\Si$ in which a collection of points
$m\subset \partial\Si$ satisfy the conditions:
$$m\cap z = \emptyset \conj{ and } \vnp{m} \in 2\ZZ_+$$
will be denoted by $(\Si,m)$. We write $m=\cup_i m_i$ where
$m_i\subset \partial_i \Si$. Often notation will be abused and $m$ will be
used to denote both the set $m$ and the cardinality $\vnp{m}$.
\end{defn}

An orientation on a pointed surface $\Si$ induces an orientation of each
boundary component. The points $m_i \subset \partial_i\Si$ inherit an
ordering by starting from the basepoint $z_i\in \partial_i\Si$ and 
 traversing the boundary circle in this direction. Combining the order on each $m_i\subset \partial_i\Si$ with the ordering of the
boundary components $\{\partial_1\Si,\partial_2\Si,\ldots,\partial_n\Si\}$ produces a total ordering on the set $m$.

Recall that an arc $\ga$ is properly embedded in a pointed surface when
$\partial \ga \subset \partial\Si\backslash z$ and
$int(\ga)\cap \partial\Si = \emptyset$. Arcs $\ga$ are required
to intersect the boundary transversely.

\begin{definition}
Let $\Sigma$ be a pointed orientable surface possibly with boundary.
Then a properly embedded collection of smooth curves and arcs $\gamma$ on $\Sigma$ is
a {\em multicurve}.

If $\ga$ is a multicurve on $(\Si,m)$ then we require that the set
$\ga\cap\partial\Si$ coincides with the  points $m$ chosen on the
boundary $\partial\Si$.
\end{definition}

\begin{definition}\label{divsetdef}%[Dividing set]
A non-empty multicurve $\ga$ is said to be a {\em dividing set on the surface $\Sigma$} when there are disjoint subsurfaces $\Sp$ and $\Sm$ of $\Sigma$ so that 
$$\Si \backslash \ga = \Sp \cup \Sm \conj{ and as sets } \ga = \partial \Sp \backslash \partial \Si = \partial \Sm  \backslash \partial \Si.$$
If $\Sigma$ is a surface with boundary then we require that the intersection number $i(\ga, \partial \Sigma)$ is a positive even integer. In particular, when $\Si$ has boundary we {\em require} that $m\geq 2$.
\end{definition}

The subsets $\Sp$ and $\Sm$ of $\Sigma$ are the {\em positive region} and the {\em negative region} of $\ga$ on $\Sigma$ respectively. These regions may be labelled by $+$ and $-$ signs in illustrations. 

If a multicurve $\ga$ is a dividing set then for each boundary component $\partial_i\Si$,
the number of points $\ga\cap \partial_i \Si$ must be even.

\begin{defn}\label{dualdef}
For any dividing set $\ga$ on $\Sigma$, there is a {\em dual dividing set} $\ga^\vee$ on $\Sigma$ that is obtained by exchanging the positive and negative regions.
\end{defn}

The {\em equator} $\ell = \{ (x,y) : y = 0 \} \subset D^2 = \{ x\in\RR^2 : \vert x \vert < 1\}$ of a disk is the line formed by the $x$-axis in the standard embedding: $D^2 \subset \RR^2$. The equator $\ell$ divides the disk $D^2$ into two {\em half-disks}: a bottom $B$ and a top $T$. 
$$D^2 = B \cup T \conj{ and } B \cap T = \ell$$
The boundary $\partial T$ of the top half-disk $T$ consists of the equator $\ell$ and the northern hemisphere $\nu \subset \partial D^2$ of the boundary circle:
$$\partial T = \ell \cup \nu.$$

\begin{definition}%[Bypass disk]
\label{divdiskdef}
Suppose that $\ga$ is a dividing set on an oriented surface $\Sigma$. Then a {\em bypass disk on $\ga$} is a smoothly embedded oriented half-disk $(T,\ell) \subset (\Sigma\times [0,1], \Sigma \times 0)$ which satisfies the following properties:
\begin{enumerate}
 \item The equatorial arc $\ell$ intersects $\ga$ at exactly three points: $a,b$ and $c$. So that 
$$\ell = [a,b] \cup [b,c] \conj{ and } a < b < c.$$
where the order of the points is induced by the orientation.
 \item The boundary points of the arcs $\ell$ and $\nu$ are the points $a$ and $c$.
\end{enumerate}
A {\em dividing set} $\b$ of a bypass disk $T$ is a properly embedded arc starting at a point $x$ between $a$ and $b$ and ending at a point $y$ between $b$ and $c$.
\end{definition}

Definition \ref{divdiskdef}  above is illustrated below.
\begin{center}\label{fig:Bypass_Disk}
\begin{overpic}%[grid,tics=10]
{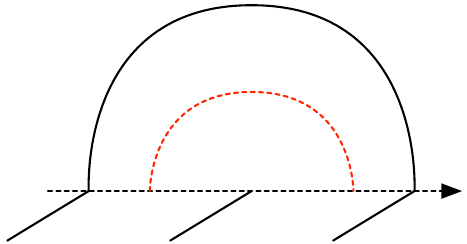}
 \put(12.5, 22.5){$\ell$}
 \put(51,7){$+$}
 \put(130,7){$-$}
 \put(45,30){$a$}
 \put(117.5,30){$b$}
 \put(202,30){$c$}
\put(68,15){$x$}
\put(165,15){$y$}
 \put(55,100){$\nu$}
\put(117.5,80){$\b$}

\end{overpic}
\end{center}
The picture above shows a bypass disk $T$ embedded in a thickened surface
$\Sigma\times [0,1]$. The boundary of the half-disk consists of the dashed
equatorial arc $\ell$ and the boundary of the northern hemisphere $\nu$. The
dashed red curve $\b$ is the dividing set for the bypass disk. The three
straight lines at the bottom are part of a dividing set $\ga$ on the surface
$\Sigma$. The labels $a,b,c$ indicate the intersection points of the arc
$\ell$ with the dividing set $\gamma$. The orientation of $T$ is determined
by fixing the direction of the equator $\ell$ and using the standard
orientation along the normal axis. The equator $\ell$ is drawn beyond the
boundary of $T$ for aesthetic reasons.

\begin{remark}
  If $\Sigma\subset (M,\xi)$ is a convex surface in a contact 3-manifold
  then $\xi$ determines a dividing set $\ga$ on $\Sigma$. A bypass disk
  $T$, embedded into a regular neighborhood of $\Sigma$, determines an
  operation on the dividing set called {\em bypass attachment} that changes
  the dividing set and the contact structure in a well-understood way
  \cite{KoOn}. These operations generate the contact structures on $M=\Si\times
  [0,1]$ in a sense which has been made precise by K. Honda
  \cite[Lem. 3.10 (Isotopy discretization)]{KoGluing}.  
\end{remark}

If $\Si$ is an oriented surface then the space $\Sigma\times [0,1]$ will be
always be oriented by appending the vertical direction to the orientation of
$\Sigma$.

\begin{defn}\label{orpresdef}
  A bypass disk $(T,\ell)$ in $\Sigma\times [0,1]$ determines the product orientation
  on $\Sigma\times [0,1]$. In more detail, if $\ell$ represents the direction of the
  equator and $n$ is the direction of the disk normal to the surface then
  the three vectors $(\ell,\ell\times n,n)$ determine this
  orientation of $\Si\times [0,1]$. If the orientation induced by $T$ agrees with that of
  $\Sigma\times [0,1]$ then the bypass disk is said to be {\em orientation
    preserving}, otherwise it is {\em orientation reversing}.
\end{defn}

\begin{definition}[Bypass move]
\label{bypassattdef}
Suppose that $\ga$ is a dividing set on an oriented surface $\Sigma$, $T$
is a bypass disk on $\ga$ and $N(T)$ is a regular neighborhood of the
half-disk $T\subset \Sigma\times [0,1]$. The boundary $\partial N(T)$
contains two copies of the half-disk $T$ which we will call {\em
  faces}. Each face, being a parallel copy of the half-disk $T$, contains a collection of
points: 
$$a < x < b < y < c$$
ordered along an equator $\ell$, a dividing set $\b$ and a northern hemisphere
$\nu$. Moreover, there are three line segments $\ga_a$, $\ga_b$ and $\ga_c$ from
$\ga$, on either side, meeting the points $a$,$b$ and $c$ respectively. The face in the $\ell\times n$ direction of $T\times \{\frac{1}{2}\}\subset T\times [0,1]$ is called the {\em positive face}, the other face is the {\em negative face}.

There is a dividing set $\eta$ on the surface $\Sigma' = \partial(\Sigma \cup N(T))$ which is
constructed by regluing the curves $\ga$ according to the prescription below.
\begin{enumerate}
\item If $T$ is orientation preserving then on the positive face
 attach $\ga_b$ to the point $x$ of $\b$ and attach $\ga_c$ to the point $y$ of $\b$ 
and  on the negative face attach $\ga_a$ to the point $x$ and attach $\ga_b$ to the point $y$.

\item If $T$ is not orientation preserving then on the positive face 
attach $\ga_a$ to the point $x$ of $\b$ and attach $\ga_b$ to the point $y$ of $\b$
and on the negative face attach $\ga_b$ to the point $x$ and attach $\ga_c$ to the point $y$.

\item Attach the curve $\ga_a$ on the latter face to the curve
   $\ga_c$ on the former face by an interval that crosses over the
  $\nu \times [0,1]\subset \partial N(T)$ boundary component along the
  diagonal.
\end{enumerate}

After smoothing the corners, the surface $\Sigma'$ is diffeomorphic to $\Sigma$ by a diffeomorphism $\psi$ which is isotopic to identity. If $\ga' = \psi(\eta)$ then the {\em bypass move} $\theta : \ga \to \ga'$ is the tuple:
$$\ga \xto{\theta} \ga' = (T, \ga,\ga')$$
given by the bypass disk $T$, the dividing set $\ga$ and the curve $\ga'$ determined by the operation described above.
\end{definition} 

\begin{remark}\label{smoothingrem}
  The definition of bypass move requires a choice of smoothing. We fix one choice and use it consistently.
Any two such choices will produce equivalent categories.
\end{remark}

The picture below shows the orientation preserving bypass move defined
above. On the lefthand side, the dividing set $\ga$ consists of three
horizontal lines and the equator $\ell$ of the bypass disk $T$ is indicated
by the vertical line. The rest of the bypass disk $T$ is assumed to come out
of the page.  The positive and negative regions on the right are determined
by the positive and negative regions on the left.
$$\CPPic{one2} \xlto{\th}\quad \CPPic{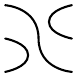}$$
In the contact category, bypass moves are required to be orientation preserving. Since the orientation of a bypass disk $T$ is determined by the direction of the equator, we will always choose orientations which are compatible with the ambient orientation of the surface. So it is not necessary to denote the orientation in most illustrations.

\subsubsection{Special types of bypass moves}\label{specialbypassec}

The two special types of bypass moves isolated below correspond precisely to
the relations (1) and (2) in Definition \ref{pkosdef}.

\begin{defn}
  A bypass move $\th : \ga \to \ga'$ is {\em capped} when either the subset
  $[a,b]$ or the subset $[b,c]$ of the associated equator $\ell$ is the
  equator $\rho$ of an embedded half-disk
  $(T, T\backslash \rho) \to (\Si, \ga)$
  which does not intersect the equator at any other point.
\begin{center}
\begin{overpic}%[grid,tics=10]
{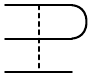}
 \put(26,22.5){$T$}
\end{overpic}
\end{center}
Intercardinal directions will be used to locate caps. For instance, a bypass featuring a cap $T$ in its northeastern corner is pictured above. 
\end{defn}

\begin{example}
The picture below contains one cap $T$ in the southeastern corner. The half-disk labelled $S$ is not a cap because it intersects the equator twice.
\begin{center}
\begin{overpic}%[grid,tics=10]
{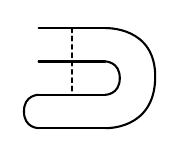}
 \put(41,33){$T$}
 \put(61,33){$S$}
\end{overpic}
\end{center}

  \end{example}

Capped bypass moves are the least interesting bypass moves because,
depending upon where the cap is found, a capped bypass must be either nullhomotopic
or equal to the identity map in the formal contact category.

\begin{defn}\label{disjdef}
  Two distinct bypass moves $\th : \ga \to \ga'$ and $\th' : \ga\to \ga''$ are {\em disjoint}, up to isotopy with end points fixed in the dividing set, when the equators of their bypass disks have geometric intersection number zero. 
\end{defn}

If a collection of bypass moves $\{\th_i\}_{1 \leq i \leq n}$ on a dividing set $\ga$ is
pairwise disjoint then performing the moves in any order produces the same
result: $\ga'$. So the union 
$$\amalg_{i=1}^n \th_i : \ga \to \ga'$$
may be viewed as kind of bypass combo-move.

\subsubsection{Isotopy of curves and disks}\label{isotopysec}

\begin{defn}\label{isodef}
  If $\ga$ and $\ga'$ are dividing sets on a surface $\Sigma$ then they are
  {\em isotopic:} $\ga \simeq \ga'$, when they are isotopic as multicurves
  on $\Sigma$. If $\Si$ is a pointed surface then the isotopy is required to
  fix the basepoints $z\subset \partial\Si$. If $(\Si,m)$ is a surface with
  points $m$ on each boundary component then the isotopy is required to fix
  the points at which the dividing sets attach to each boundary component.

Two bypass moves $\theta =(T,\ga,\ga')$ and $\theta' = (S,\d,\d')$ are {\em
  isotopic:} $\theta \simeq \theta'$, when the graph $\ga \cup \ell$ is
isotopic to $\d \cup \rho$ where $\ell$ and $\rho$ are equators of $T$ and
$S$ respectively.
\end{defn}

\begin{remark}\label{tmanrem}
  If $\Sigma$ is realized as a convex surface in the 3-manifold
  $M = \Sigma\times [0,1]$ and the dividing sets $\ga$ and $\ga'$
  corresponding to two contact structures $\xi$ and $\xi'$ are isotopic then
  $\xi$ and $\xi'$ are contactomorphic \cite{KoOn}. Since our motivation is
  to produce a category in which morphisms behave like contact structures up
  to contactomorphism, isotopic dividing sets are identified in Definition
  \ref{pkosdef} below.
\end{remark}

\subsection{The contact category}\label{contactdefsec}

\begin{defn}\label{pkosdef}
  The {\em pre-formal contact category $\PKoS$} is the ungraded $k$-linear
  category with objects corresponding to isotopy classes of dividing sets on
  $\Sigma$ and maps generated by isotopy classes of orientation preserving
  bypass moves subject relations below.
\begin{enumerate}
\item If $\th$ is a capped bypass move then $\th = 1$ when the cap can be found in the northwest or southeast:
$$\CPPic{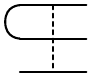} = 1 \conj{ and } \CPPic{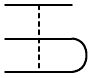} = 1.$$
\item If $\th$ and $\th'$ are disjoint bypass moves then the maps that they determine commute:
$$\th\th' = \th\amalg \th' = \th'\th.$$
\end{enumerate}
\end{defn}

The relations above are required for the formal contact category, defined
below, to have any bearing on contact geometry, see Remark \ref{tmanrem}
above. In Section \ref{overtwistedsec}, we will show that the first relation
implies that $\th=0$ in the associated homotopy category when the
corresponding bypass is capped in the northeast or the southwest:
$$\CPPic{NEcap} = 0 \conj{ and } \CPPic{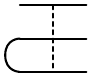} = 0.$$

The next proposition shows that every bypass move determines a triple of composable morphisms. This determines a functor from the category $\Tp$ in Def. \ref{Ttridef} to the category $\PKoS$.
This proposition is due to K. Honda and K. Walker, see \cite{KoHo,Walker2}.

\begin{prop}\label{triprop}
  For each oriented surface $\Sigma$ and each dividing set $\ga$ on
  $\Sigma$, each bypass move $\theta$ on $\ga$ determines a functor
  $\tilde{\theta} : \Tp \to \PKoS$.
\end{prop}
\begin{proof}
Set $\ga_A = \ga$ and $\th_A = \th$.
  By definition, a bypass move $\theta_A = (T_A, \ga_A,\ga_B)$ is locally modelled
  on a bypass disk $T_A$ in $\Sigma\times [0,1]$ which intersects $\ga_A$ in
  three points. There is a bypass disk $T_B$ on the dividing set $\ga_B$ which
  results from the bypass move $\th_A$. The disk $T_B$ determines a bypass
  move $\th_B = (T_B,\ga_B,\ga_C)$ and there is a bypass disk $T_C$ on the
  dividing set $\ga_C$. The disk $T_C$ determines a bypass move
  $\theta_C = (T_C,\ga_C,\ga_A)$; the result of the bypass $T_C$ is the
  original dividing set $\ga = \ga_A$. These choices are unique up to isotopy.
\end{proof}

The construction above is illustrated below. Each of the arrows in the diagram is
a bypass move. The solid lines represent dividing sets on the surface
$\Sigma$ and the dashed lines represent the equators of bypass disks.
\begin{center}
\begin{tikzpicture}[scale=10, node distance=2.5cm]
\node (A2) {\CPic{one2}};
\node (B2) [right=1cm of A2] {};
\node (C2) [right=1cm of B2] {\CPic{two2}};
\node (D2) [below=1.41421356cm of B2] {\CPic{three2}};
\draw[->, bend left=35] (A2) to node {$\theta_A$} (C2);
\draw[->, bend left=35] (C2) to node {$\theta_B$} (D2);
\draw[->, bend left=35] (D2) to node {$\theta_C$} (A2);
\end{tikzpicture}
\end{center}
The icon at the source of a given arrow represents a dividing set $\ga$ on
the surface $\Sigma$. The icon at the target of the arrow represents the
dividing set obtained by performing the bypass move with equator given by
the dashed line in the source.

The proposition above allows us to associate a functor
$\tilde{\theta} : \Tp \to \PKoS$ to each bypass move $\th : \ga \to \ga'$
between dividing sets on $\Sigma$. Composing the coproduct 
$\coprod_{\theta}\tilde{\theta} : \coprod_\theta \Tp \to \coprod_\theta \PKoS$
of all such functors with the fold map $\coprod_\th \PKoS \to \PKoS$ yields the functor:
\begin{equation*}\label{locfun}
\Xi : \coprod_\theta \Tp \to \PKoS.
\end{equation*}

\begin{defn}\label{formaldef}
  The {\em formal contact category} $\KoS$ is the pretriangulated hull of
  the Postnikov localization of the pre-formal contact category $\PKoS$
  along the functor $\Xi$ above.
$$\KoS = L_{\Xi}\PKoS^{\pt}$$
\end{defn}

By Proposition \ref{prestriprop}, the bypass triangles introduced by the
Postnikov localization remain distinguished triangles in the homotopy
category of the hull. The formal contact category $\KoS$ is the universal
pretriangulated category generated by bypass moves, containing bypass
triangles and satisfying the relations $(1)$ and $(2)$.

\begin{conjecture}\label{cofibrem}
  A cofibrant-fibrant replacement for $\KoS$ can be constructed without
  homotopy pushouts. Note that, before relations $(1)$ and $(2)$ are applied to the pre-formal contact category:
$$\PKoS = \PPKoS/\inp{(1),(2)},$$
the ``pre-pre-formal contact category'' is freely generated by bypass
moves. Any freely generated category is cofibrant as it can be obtained by a
series of pushouts along generating cofibrations in $\Hqe$. One can then
adjoin copies of Drinfeld's category $\dr$ via pushout and copies of a
resolution for the symmetric algebra for each instance of relations $(1)$
and $(2)$ respectively. The result is cofibrant in $\Hqe$, so the homotopy
pushout which underlies the Postnikov localization in Definition
\ref{formaldef} is now an ordinary pushout and the result of this pushout is both cofibrant and fibrant in $\Hqe$. The idempotent completion $L_{\Xi}\PKoS^{\perf}$ of $\KoS$ is cofibrant and fibrant in the Morita category $\Hmo$.
\end{conjecture}

\newpage
\section{Elementary properties of contact categories}\label{propertiessec}
In this section many of the properties which should hold for the contact
categories \cite{KoHo} are shown to hold for the formal contact
categories. The formal contact category associated to a surface decomposes
into a product of formal contact categories with fixed Euler invariant. The
category with Euler invariant $n$ is equivalent to the category with Euler
invariant $-n$. Reversing the orientation of the surface is equivalent to
forming the opposite category. A dividing set featuring a homotopically
trivial curve is contractible and dividing sets featuring regions which are
disconnected from the boundary are shown to be homotopy equivalent to
convolutions of dividing sets which are connected to the boundary.

\subsection{Decompositions of contact categories}\label{decompsec}
The contact categories $\Ko(\Sigma)$ consist of non-interacting subcategories
$\Ko^n(\Sigma, m)$. Each subcategory is determined by fixing some points $m$ on each
boundary component and the Euler number $n = \spc(\ga)$ of the dividing sets $\ga$ on $\Sigma$.

\subsubsection{Euler decomposition}\label{eulerdecompsec}
If $(\Sigma\times [0,1),\xi)$ is a contact 3-manifold and $e(\xi)$ is the Euler class of $\xi$ then the Euler number of $\xi$ is  $\spc(\xi) = \inp{e(\xi),[\Si]}.$ This number can be computed from the dividing set $\ga \subset \Si$.

\begin{defn}\label{eudef}
  If $\ga$ is a dividing set on an orientable surface $\Sigma$ then the {\em
    Euler number} $\spc(\ga)$ of $\ga$ is the Euler characteristic of the
  positive region minus the Euler characteristic of the negative region:
$$\spc(\ga) = \chi(\Sp) - \chi(\Sm).$$
\end{defn}

The proposition below shows that this is a reasonable thing to consider. 

\begin{prop}
The Euler number satisfies the following properties:
\begin{enumerate}
\item If two dividing sets are isotopic then the corresponding Euler numbers are equal:
$$\ga \simeq \ga' \conj{ implies that } \spc(\ga) = \spc(\ga').$$
\item If $\theta : \ga \to \ga'$ is a bypass move then the Euler numbers of $\ga$ and $\ga'$ must be equal.
\end{enumerate}
\end{prop}
\begin{proof}
  The first statement follows from the observation that $\ga \simeq \ga'$ implies that $\Sp \simeq \Sp'$ and $\Sm \simeq \Sm'$. 

The second statement follows from computing each Euler characteristic as a union of the region in which the bypass move is performed and its complement. Suppose that $B\subset \Si$ is a small ball containing the bypass moves. If $X_\pm = R_\pm \backslash B$ and $Y_\pm = R_\pm\cap B$ then $Y_\pm$ is homeomorphic to the disjoint union of two disks and $X_\pm \cap Y_\pm$ is homeomorphic to the disjoint union of three intervals. See the illustration following Definition \ref{bypassattdef}.
\end{proof}

\begin{remark}
  If $\ga$ is a dividing set on a surface $(\Si_{g,1},2)$ of genus $g$ with one boundary component and two points on the boundary then $\chi(R_+\cap R_-) = 1$ because $\ga$ consists of a disjoint union of circles and one interval connecting the two points which are fixed on the boundary. So
$2-2g = \chi(R_+) + \chi(R_-)$. If $\spc(\ga) = 2(g-k)$ then $\chi(R_+) = 1-k$ and $\chi(R_-) = 1-l$ where $k+l=2g$ for $0\leq k \leq 2g$.
\end{remark}

Since the pre-formal contact category $\PKoSm$ in Definition \ref{pkosdef} is generated by bypass moves, the proposition above is equivalent to the statement that the Euler number yields a well-defined map: $\spc : \Ob(\PKoSm) \to \ZZ$ which determines a decomposition:
$$\PKoSm \cong \coprod_{n\in\ZZ} \PKoSmn$$
in which $\PKoSmn$ is the full subcategory of $\PKoSm$ such that $\spc(\ga) = n$ for all $\ga \in \Ob(\PKoSmn)$. The theorem below shows that this decomposition extends to the formal contact category $\KoSm$.

\begin{theorem}\label{splitthm}
The formal contact category $\KoSm$ splits into a product of categories $\KoSmn$:
$$\KoSm \cong \prod_{n\in\ZZ} \Ko^n(\Sigma,m)$$
where $\KoSmn$ is the full subcategory of $\KoSm$ with objects that satisfy
$\spc(\ga) = n$.
\end{theorem}
\begin{proof}
By the proposition above, $\Xi : \coprod \Tp \to \PKoSm$ splits into a union $\Xi = \coprod_n \Xi_n$ where
$\Xi_n : \coprod \Tp \to \PKoSmn$ corresponds to the bypass triangles contained in $\PKoSmn$. The localization functor $Q : \PKoSm \to L_{\Xi}\PKoSm$ splits into a union of localizations:
$$\PKoSm \cong \coprod_n \PKoSmn \to L_{\Xi} \coprod_n \PKoSmn \cong \coprod_n L_{\Xi_n} \PKoSmn.$$
The theorem follows from Proposition \ref{ptmultprop}.
\end{proof}

\subsection{Dualities of contact categories}\label{dualitiessec}
Two forms of duality are introduced, corresponding to switching the
labellings of the regions and the ambient orientation of the surface
respectively.

\subsubsection{Euler duality}\label{eulerdualsec}
Definition \ref{dualdef} introduced an operation $\ga \mapsto \ga^\vee$ on
dividing sets which exchanged the positive and negative regions: $\Sp \leftrightarrow \Sm$. This
reverses the sign of the Euler number: $\spc(\ga^\vee) = -\spc(\ga)$.
Here this operation is extended to an involution 
$$-^\vee : \KoSm\to\KoSm$$
of the formal contact category which exchanges $\Ko^n(\Si,m)$ and $\Ko^{-n}(\Si,m)$ from Theorem \ref{splitthm}.

\begin{prop}\label{dualityfunprop}
The Euler duality map on dividing sets: $-^\vee : \Ob(\PKo^n(\Si,m))\to\Ob(\PKo(\Si,m))$ extends to an involution of dg categories:
$$-^\vee : \Ko^{n}(\Si,m) \to \Ko^{-n}(\Si,m) \conj{ and } (-^\vee)^\vee\cong 1.$$
\end{prop}
\begin{proof}
If $\ga$ is a dividing set on $\Si$ then for any bypass move $\theta : \ga \to \ga'$ the positive and negative regions of $\ga$ determine positive and negative regions of $\ga'$; see the illustration after Definition \ref{bypassattdef}. Therefore, on the generators $\theta$ of $\PKo^n(\Si,m)$: 
$$\theta : \ga \to\ga'\conj{ $\mapsto$ } \theta^\vee : \ga^\vee \to \ga'^\vee.$$
This extends to an involution of $\PKo(\Si,m)$ which takes triangles to triangles and so descends to a functor: $-^\vee : \Ko^n(\Si,m)\to Ko^{-n}(\Si,m)$. The uniqueness of this extension implies the relation $(-^\vee)^\vee\cong 1$. The map $-^\vee$ is an equivalence as it is its own inverse.
\end{proof}

\subsubsection{Orientation reversal}\label{orrevsec}

The formal contact category $\Ko(\bar{\Si})$ of a surface with reversed orientation is identified with the opposite formal contact category $\Ko(\Si)^\op$ of the surface.

\begin{proposition}\label{orrevprop}
There is an equivalence of formal contact categories,
$$\Ko^n(\Sigma,m)^\op \xto{\sim} \Ko^{n}(\bar{\Si},m).$$
\end{proposition}
\begin{proof}
  It is a consequence Definition \ref{bypassattdef} that reversing the orientation of
  the surface is equivalent to reversing the orientation of each bypass
  half-disk or equator. It suffices to analyze the correspondence between bypass
  triangles. In the eyeglass-shaped diagram below, reversing the orientation of each bypass
  disk, $\theta \mapsto \bar{\theta}$ in a triangle fixes the source and changes the sink of each map.
\begin{center}
\begin{tikzpicture}[scale=10, node distance=2.5cm]
\node (A2) {\CPic{one2}};
\node (B2) [right=1cm of A2] {};
\node (C2) [right=1cm of B2] {\CPic{two2}};
\node (D2) [below=1.41421356cm of B2] {\CPic{three2}};
\draw[->, bend left=35] (A2) to node {$\theta$} (C2);
\draw[->, bend left=35] (C2) to node {$\theta'$} (D2);
\draw[->, bend left=35] (D2) to node {$\theta''$} (A2);

\node (ZZ) [right=3.5cm of B2] {};

\node (A3) [right=1cm of ZZ] {\CPic{one2}};
\node (B3) [right=1cm of A3] {};
\node (C3) [right=1cm of B3] {\CPic{two2}};
\node (D3) [below=1.41421356cm of B3] {\CPic{three2}};
\draw[->, bend right=35] (A3) to node [swap] {$\bar{\theta}$} (D3);
\draw[->, bend right=35] (D3) to node [swap] {$\bar{\theta''}$} (C3);
\draw[->, bend right=35] (C3) to node [swap] {$\bar{\theta'}$} (A3);

\draw[|->, bend left=35] (C2) to node {$\bar{\cdot}$} (A3);
\end{tikzpicture}
\end{center}
Reversing the arrows on the lefthand side of the diagram produces the bypass
triangle for $\Ko^n(\Si,m)^\op$. The assignment $\ga\mapsto \ga$ on objects
and $\th^\op \mapsto \bar{\th'}$ on maps determines a functor
$\bar{\cdot} : \PKo^n(\Si,m)^\op \to \PKo^{n}(\bar{\Si},m)$ because it
preserves the cap relations and disjoint unions.  Moreover, the relation
$\th^\op \mapsto \bar{\th'}$ implies that
$(\theta')^\op \mapsto \bar{\theta''}$ and
$(\theta'')^\op \mapsto \bar{\theta}$ so that triangles are mapped to
triangles and the functor $\bar{\cdot}$ descends to a map between formal
contact categories. By applying the same construction to the surface after
reversing its the orientation again, one obtains an inverse functor and so
the functor $\bar{\cdot}$, introduced above, is an isomorphism of formal
contact categories.
\end{proof}

\subsection{Relations for overtwisted contact structures}\label{overtwistedsec}
A theorem of E. Giroux \cite{G} states that a contact structure on
$\Sigma \times [0,1]$, when $\Sigma \ne S^2$, is overtwisted if and only if
its dividing set contains no homotopically trivial closed curves. When
$\Sigma = S^2$, a contact structure is overtwisted if and only if the
dividing set contains any two such curves. Corollary \ref{tightcor} states
that E. Giroux's criterion is satisfied for surfaces with boundary. The
surface $\Si$ is assumed to be connected in this section.

The lemma below shows that the local relations can be applied to parts of
more complicated dividing sets.

\begin{lemma}{(Local relations)}\label{locrellem}
  Suppose that $R$ and $\Si$ are orientable surfaces and $R \subset
  \Si$.
  Then a distinguished triangle in $\Ho(\Ko(R))$ yields a distinguished
  triangle in $\Ho(\Ko(\Si))$.
\end{lemma}
\begin{proof}
  The embedding $R\subset \Si$ determines a functor
  $i : \PKo(R)\hookrightarrow \PKo(\Si)$. A bypass triangle
  $\tilde{\theta} : \Tp \to \PKo(R)$ determines a bypass triangle
  $\Tp \to \PKo(\Si)$ after composing with $i$.
\end{proof}

\begin{defn}\label{nsonedef}
  If $\ga$ is a dividing set then we write $S^1\subset \ga$ when $\ga$
  contains a homotopically trivial closed curve. All such curves are
  isotopic when $\Si$ is connected. If $\ga$ contains any collection of
  $n\in\ZZ_+$ such curves then we write $nS^1\subset \ga$.
\end{defn}

\begin{proposition}\label{otprop}
The object represented by the dividing set pictured below is contractible.
\begin{center}\CPPic{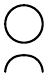}\hspace{-.75cm} $\cong$ \hspace{.4cm} $0$\end{center}
\end{proposition}
\begin{proof}
  The formal contact category $\Ho(\Ko(D^2,2))$ associated to the disk $D^2$
  with two boundary points contains a bypass move with equator indicated by
  the dashed line below.
$$\CPPic{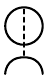}$$
All of the objects in the distinguished triangle associated to the bypass move are isotopic and the first relation in Definition \ref{pkosdef} implies two out of three of the maps are identity. 
\end{proof}

\begin{corollary}\label{tightcor}
\begin{enumerate}
\item If $\Sigma$ is a surface with boundary then for all dividing sets $\ga$ on $\Sigma$,
$$S^1\subset \ga \conj{ implies } \ga \cong 0 \normaltext{  in  } \HoKoS.$$
\item If $\Sigma$ is a closed surface then for all dividing sets $\ga$ on $\Sigma$,
  $$S^1\subset \ga \normaltext{ and } \ga \ne S^1 \conj{ implies } \ga \cong 0 \normaltext{ in } \HoKoS.$$
%% \item If $\Sigma$ is a closed surface then for all dividing sets $\ga$ on $\Sigma$,
%% $$nS^1\subset \ga \normaltext{  and  } n\geq 2 \conj{ implies } \ga \cong 0 \normaltext{  in  } \HoKoS.$$
\end{enumerate}
\end{corollary}
\begin{proof}
  The proposition above applies to surfaces with boundary as they are required to contain properly embedded arcs.
\end{proof}

Without further complicating the main construction this corollary appears to
be optimal: bypass moves do not imply that $S^1\cong 0$ in the disk category
$\Ho(\Ko(D^2,0))$, any such proof would contradict E. Giroux's theorem for
$\Sigma = S^2$.

\begin{corollary}\label{otpropcor}
The relation in Proposition \ref{otprop} above implies that a bypass move is zero in the homotopy category when it is capped in either the northeast or southwest:
$$\CPPic{NEcap} = 0 \conj{ and } \CPPic{SWcap} = 0.$$
\end{corollary}
\begin{proof}
The dividing set $\ga'$ resulting from either bypass move $\th : \ga \to \ga'$ must contain a
homotopically trivial curve. So the isomorphism $\ga'\cong 0$ is obtained by applying Lemma \ref{locrellem} and Proposition \ref{otprop}. This implies the relation $\th = 0$ in the homotopy category of the formal contact category.
\end{proof}

\begin{remark}
Two consecutive bypass moves occurring in a bypass triangle are disjoint:
$$\CPPic{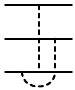}$$
The second bypass is capped when it is performed before the first, so the commutativity of disjoint bypasses and the corollary above suffice to imply that compositions of consecutive bypass moves must be zero in the homotopy category. 
\end{remark}

\subsection{Dividing sets containing disconnected regions are convolutions}\label{discosec}
Suppose $\ga$ is a dividing set on a surface $\Si$ with boundary and
$\Si\backslash \ga$ contains a connected component $B$ which is disjoint from the boundary of $\Si$.
 Then we will show that $\ga$ is homotopy equivalent to an iterated
cone construction on dividing sets which do not contain a region such as $B$.

\begin{defn}\label{discodef}
  A multicurve $\ga$ on a surface $\Si$ with boundary is {\em boundary disconnected} when there is a connected component $B$ of $\Si\backslash \ga$ which does not touch the boundary:
$$B\subset \Si\backslash\ga \conj{ and } B \cap \partial\Si = \emptyset$$
A dividing set $\ga$ is {\em boundary connected} when it is not boundary disconnected.
\end{defn}

\begin{theorem}\label{surfacetheorem}
  In the homotopy category of the formal contact category $\Ko(\Si,m)$ associated to a
  surface $(\Si,m)$ with boundary, every boundary disconnected dividing set
  $\ga$ is isomorphic to an iterated extension of dividing sets $\ga_i$ which are boundary connected.
%% a convolution of dividing sets $\ga_i$ which are boundary
%%   connected.
%% $$ \ga \cong (\oplus_{i=1}^n \ga_i,p)$$
%%  Moreover, the components of the twisted differential $p$ are bypass moves.
\end{theorem}
\begin{proof}
  Observe that boundary disconnected regions can be nested. For example, an annulus can be placed within the annulus illustrated below. For the purpose of this argument, the amount of nesting $n(\ga)$ is defined to be
  $$n(\ga) := \max_{\substack{B}} \min_{\substack{a}} \vnp{a \cap \ga}$$
where $a : (I,\{0\},\{1\})\to (B, \partial \Si, int(B))$ is an arc from the boundary $\partial \Si$ to an interior  point of a connected component $B\subset \Si\backslash \ga$.

The proof is by induction on the amount of nesting in boundary disconnected regions. Fix a dividing set $\ga$. If the nesting $n(\ga)=0$ and there are no boundary disconnected regions then there is nothing to show.  So assume that the statement of the theorem holds for all $\ga$ with $n(\ga) = N$ and suppose $n(\ga) = N+1$. 

There are innermost disconnected regions $B$ and arcs $a : I \to B$ in $\Si$ which satisfy $\vnp{a \cap \ga} = N+1$. Fix such a disconnected region $B$.

If this disconnected region is a disk then $\ga$ is isomorphic to zero because $\vnp{m}\geq 2$ by Proposition \ref{otprop}.  If $\ga$ is a dividing set on a surface with boundary and $\Si\backslash \ga$ contains an annulus or a punctured torus component then there are bypass moves:
  $$\hspace{-.5in} \CPPic{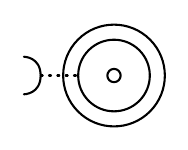}  \hspace{.8in} \conj{ and } \hspace{.5in} \CPPic{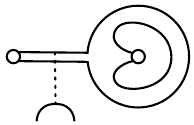} $$
respectively. The first picture above shows two concentric, homotopically non-trivial, circles in the annulus $(S^1\times [0,1], 2)$.
In the second picture above, the two small circles are identified by folding the page to form a torus with one boundary component $(T^2\backslash D^2,2)$. In either case, the triangle associated to the indicated bypass move results in two dividing sets which connect $B$ to either the boundary, when $n(\ga)=1$ or a region outside of $B$, when $n(\ga)>1$ in either case lowering $n(\ga)$ by $1$.

In general, the innermost region $B$ is an orientable surface with boundary. Any such surface is
obtained by attaching 1-handles to the boundary components of a disjoint
union of punctured tori $\Si_{1,1}$ and annuli $\Si_{0,2}$. If $B$ has
genus $g$ and $n+1$ boundary components then $B$ is abstractly
homeomorphic to $g$-copies of $\Si_{1,1}$ and $n$-copies of $\Si_{0,2}$
glued together in this fashion.  In particular, there is a $1$-handle $H$
which, when cut along its cocore $I$, produces a disjoint union of surfaces
with lower genus or number of boundary components. There in an interval $\ell$ in $\Si$ which is obtained by connecting $I$ to a point on the boundary of the region outside of $B$ (which is not in $\partial B$ itself). By construction, this interval $\ell$ intersects $\ga$ at three points.
The bypass move $\th$ determined by $\ell$ is determines a distinguished triangle
$$\ga \xto{\th} \ga' \to \ga'' \to \ga[1]$$ 
with objects $\ga'$ and $\ga''$
that must contain disconnected regions, $B$ and $B''$, with lower genus or
number of boundary components. This procedure can be iterated until the
result contains only annuli and tori to which one applies the bypasses in
the previous paragraph.

Applying the procedure in the previous two paragraphs to each innermost disconnected region expresses the result as an iterated extension of dividing sets for which $n(\ga)< N+1$. It follows by induction that $\ga$ can be further expressed as an iterated extension of dividing sets for which $n(\ga)=0$ which are boundary connected. So that the statement of the theorem holds.
\end{proof}

\subsection{The positive half of the contact category}\label{possec}

The decomposition of the formal contact category introduced by the
proposition below will clarify our discussion later.

\begin{proposition}\label{signdecompprop}
The formal contact category $\Ko(\Si,m)$ associated to a surface with boundary splits into a product of two pieces:
$$\Ko(\Si,m) \cong \Ko_+(\Si,m) \times \Ko_-(\Si,m),$$
supported on the dividing sets $\ga\in\Ko(\Si,m)$ in which the basepoint $z_1\in\partial_1 \Si$ is contained in a positive or negative region respectively.
\end{proposition}
\begin{proof}
  If two dividing sets $\ga$ and $\ga'$ are isotopic then the signs of the
  regions containing the basepoint must be equal. If $\th : \ga \to \ga'$ is
  a bypass move then it cannot change the sign of the region containing the
  basepoint $z_1$.  The rest of the proof follows along the same lines of
  the proof of Theorem \ref{splitthm}.
\end{proof}

By Proposition \ref{dualityfunprop}, the two pieces found in the
decomposition above are equivalent:
$$-^\vee : \Ko_+^n(\Si,m) \xto{\sim} \Ko_-^{-n}(\Si,m).$$
In Corollary \ref{rotationprop}, moving the basepoint $z_1$ to an adjacent region is shown to yield an equivalence $r : \Ko_+^n(\Si,m) \xto{\sim} \Ko_-^{n}(\Si,m)$. By composing the two maps we obtain an equivalence:
$$\Ko_+^n(\Si,m) \xto{\sim} \Ko^{-n}_+(\Si,m).$$
See also Proposition \ref{ytdualprop}.

\newpage
\section{Symmetries and generators of contact categories}\label{mcgMainsec}
The mapping class group of the surface $\Si$ is shown to act naturally on
the formal contact category $\KoS$. After introducing arc diagrams and
parameterizations of surfaces by arc diagrams, each parameterization of
$\Si$ by an arc diagram is shown to yield a system of generators for the
formal contact category.  Section \ref{kodecatsec} contains a discussion of
decategorification.

\subsection{The mapping class group action}\label{mcgsec}

In this section we show that the mapping class group $\mcgS$ acts naturally
on the formal contact category $\KoS$.

\begin{defn}
  Suppose that $\Sigma$ is an oriented surface. Then the mapping class group $\mcgS$ is the group of connected components of the group of orientation preserving and boundary fixing diffeomorphisms:
$$\mcgS = \pi_0 Diff^+(\Sigma,\partial\Sigma).$$
\end{defn}

Recall that an action of a group $G$ on a
dg category $\eC$ is a homomorphism from $G$ to the group
$\Aut(\eC)\subset \Endo_{\Hmo}(\eC)$ of derived equivalences.

\begin{theorem}\label{mcgthm}
  The mapping class group $\mcgS$ acts naturally on the formal contact category $\KoS$.
\end{theorem}
\begin{proof}
  The proof occurs in two steps: first we construct a natural $\mcgS$-action on the pre-formal contact category $\PKoS$ and second this group action is extended to the formal contact category $\KoS$.

  A diffeomorphism class $g\in\mcgS$, determines a functor $f_g : \PKoS\to
  \PKoS$ that is defined by its action on dividing sets and bypass moves. If
  $\ga$ is an isotopy class of dividing set on $\Sigma$ then there is a
  unique isotopy class of dividing set $g\ga$ and if $\theta = (T,\ga,\ga')$
  is a bypass move then there is a unique bypass disk $gT$ and associated
  bypass move $g\theta = (gT, g\ga,g\ga')$. Since the category $\PKoS$ is
  generated by bypass moves and the assignment $\th \mapsto g\th$ preserves disjointness of bypass moves
  and caps of bypass moves, there is a functor
$$f_g : \PKoS \to \PKoS \conj{ such that } f_g(\ga) = g\ga \normaltext{ and } f_g(\theta) = g\theta.$$
Both the composition law $f_{gg'} = f_g \circ f_{g'}$ and naturality follow directly from the definition. In particular, since the identity diffeomorphism $1\in\mcgS$ fixes both dividing sets and bypass moves the functor $f_1$ is the identity functor $1_{\PKoS}$.

Suppose that $f_g : \PKoS \to \PKoS$ is a functor occuring in the construction above. Composing with the localization functor $Q : \PKoS \to L_\Xi\PKoS$ from Equation \eqref{locfun} yields a functor $\PKoS \to L_{\Xi}\PKoS$. By Definition \ref{anlocdef}, the image of $Q^* : \Hom_{\Hqe}(L_{\Xi}\PKoS, L_{\Xi}\PKoS) \to \Hom_{\Hqe}(\PKoS,L_{\Xi}\PKoS)$ is the subset of functors $f : \PKoS \to L_{\Xi}\PKoS$ whose restriction to a bypass triangle extends to a distinguished triangle in the localization $L_{\Xi}\PKoS$.

If $\tilde{\theta} : \Tp \to \PKoS$ is the bypass triangle: 
$$\ga \xto{\theta} \ga' \xto{\theta'} \ga'' \xto{\theta''} \ga[1]$$
associated to a bypass move $\theta = (T,\ga,\ga')$ on $\Sigma$ by Proposition \ref{triprop} then $f_g(\theta) = (gT,g\ga,g\ga')$ and $f_g(\tilde{\theta})$ corresponds to the bypass triangle:
$$g\ga \xto{g\theta} g\ga' \xto{g\theta'} g\ga'' \xto{g\theta''} g\ga[1].$$
Since the criteria of Definition \ref{anlocdef} are satisfied, there is a unique
lift of the functor $Q\circ f_g$ to a functor:
$\tilde{f}_g : L_\Xi \PKoS \to L_\Xi \PKoS$. By Proposition \ref{pretriprop}, there
is an induced functor between the associated pretriangulated hulls:
$$h_g : \KoS \to \KoS \conj{ where } h_g = \tilde{f}_g^\pt.$$
Uniqueness of the lift and functoriality of $-^\pt$ imply that the stated group action is obtained.
\end{proof}

The same argument as above allows us to define an automorphism $r$ which
moves the first basepoint across the first adjacent boundary point.  The
corollary below records the existence of this map.

\begin{cor}\label{rotationprop}
  There is a distinguished automorphism $r$ of $\Ko(\Si,m)$ which moves the
  first basepoint $z_1\in\partial_1\Si$ on the first boundary component over
  the nearest boundary point in the direction of the orientation.
\end{cor}

The functor $r$ induces functors $r : \Ko^n_{\pm}(\Si,m) \to \Ko^n_{\mp}(\Si,m)$ with respect to the decomposition of $\Ko^n(\Si,m)$ found in Proposition \ref{signdecompprop}. See also Propositon \ref{ytdualprop}.

\subsection{Arc diagrams}\label{arcdiagramssec}

An arc diagram is a combinatorial way to record a handle decomposition of a
surface. The definitions below are due to R. Zarev \cite{Zarev1} and
constitute generalizations of ideas which were used by R. Lipshitz,
P. Ozsv\'{a}th and D. Thurston \cite[\S 3.2]{LOT}.

\begin{defn}\label{arcddef}
  An {\em arc diagram} $\Ad$ consists of three things:
\begin{enumerate}
\item an ordered collection $\lns = \{ \Ad_1,\ldots,\Ad_\ell\}$ of $\ell$ oriented line segments,
\item a set $\pts =\{ a_1,\ldots, a_{2k}\}$ of distinct points in the line segments $\lns$ and
\item a two-to-one function $M : \pts \to \{ 1,\ldots,k\}$ called the {\em
    matching}.  
\end{enumerate}
In order to apply to any version of the Bordered Heegaard-Floer package,
this data is required to be {\em non-degenerate}: after performing surgery
on $\lns$ at each $0$-sphere $M^{-1}(j)$, for $1\leq j \leq k$, the resulting $1$-manifold has no
closed components.
\end{defn}

The set of points $\pts$ receives a total ordering from the order on the set $\lns$ and the
orientations of the line segments. The numbers $\ell$ and $k$ are allowed to
be zero. Each arc diagram $\Ad$ determines a surface $\asuf$.

\begin{defn}\label{surfdef}
The {\em surface} $\asurf$ associated to an arc diagram $\Ad$ is given by thickening each line segment $\Ad_i$ to $\Ad_i \times [0,1]$ for $1\leq i \leq \ell$ and attaching oriented 1-handles $D^1\times D^1$ along the normal bundles of the 0-spheres $M^{-1}(j) \times \{0\}$ for $1\leq j \leq k$. The surface $\asurf$ is oriented by extending the orientation of the line segment $\Ad_1$ and its positive normal.

\end{defn}

\begin{rmk}\label{newrmk}
One can regard $\Ad_i$ as part of the boundary of $\Ad_i \times [0,1]$. In Definition \ref{elemdef}, an arc parameterization will be used to construct dividing sets $\z_C \in \Ko_+(F(\Ad))$ in which the positive regions correspond to the handles of $\Ad$. In particular, $\Ad_i$, when regarded as part of the boundary, will always be contained in a positive region of $\z_C \in \Ko_+(F(\Ad))$ (and a negative region of $z_C\in \Ko_-(F(\Ad))$.
  \end{rmk}

Recall that the points $m$ on a pointed oriented surface $(\Si,m)$ are also
ordered by the ordering of the boundary components and the order on each
boundary component is obtained by starting from each basepoint and following
in the direction of the orientation induced on the boundary.

\begin{defn}
Suppose that $m \subset \partial \Si$ is the set of {\em sutures} or points fixed along the boundary of $\Si$.  An {\em arc parameterization} $(\Ad,\vp_{\Ad})$ of a pointed oriented surface $(\Si,m)$ is an arc diagram $\Ad$ and a proper orientation preserving diffeomorphism 
$$\varphi_\Ad : (\asurf, \cup_{i=1}^\ell \partial\Ad_i) \to (\Si,m)$$
which preserves total order on the points $\pts$ and $m$ respectively. 
\end{defn}

\begin{remark}
An arc parameterization identifies $\cup_{i=1}^\ell \partial\Ad_i$ with $m$.  
The sets $m$ and $\pts$ play different roles, but under this identification, pairs in $m$ partition the points of $\pts$.
\end{remark}

\begin{example}\label{annulusparam}
The annulus $(S^1\times [0,1],(2,2))$ with two points fixed on each boundary component is parameterized by the arc diagram $\Ad$ pictured on the left below.
$$\BPic{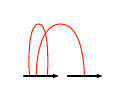}\conj{ } \BPic{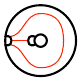}$$
This picture contains two oriented lines $\lns = \{\Ad_1,\Ad_2\}$ and four points $\pts = \{x,x',y,y'\}$ with $\Ad_1 = xyx'$ and $\Ad_2 = y'$. The matching function $M : \pts \to \{1,2\}$ is determined by the assignments $M(x) = 1 = M(x')$ and $M(y) = 2 = M(y')$. The picture on the right shows the surface $\asurf$ associated to $\Ad$.
\end{example}

\subsection{Generators from arc diagrams}\label{zarevgensec}
In this section we show that a parameterization
$\param = (\Ad, \varphi_{\Ad})$ of a pointed oriented surface $(\Si,m)$
determines a canonical collection $\Zi(\Ad)$ of generators for the
associated contact category $\Ko(\Si,m)$.  This material is motivated by a
reading of R. Zarev \cite{Zarev2}.

\begin{defn}\label{elemdef}
Suppose that a pointed oriented surface $(\Si,m)$ is parameterized by an arc diagram $\Ad$. Then for each subset
  $C \subset \{1,\ldots,k\}$ of matched pairs, there is an {\em elementary dividing set} 
$$\z_C = \partial R_C \conj{ on } \Si$$ 
where $R_C \subset \Si$ is the union of a thickening of the core of each 1-handle indexed by $C$ with the collection of thickened oriented arcs $\Ad_i \times [0,1]$. The region $R_C$ is the positive region of $\z_C$ and its complement $\Si\backslash R_C$ is the negative region of $\z_C$.

An elementary dividing set may be also be called a {\em positive elementary
  dividing set}.  The {\em set of positive elementary dividing sets} will be
denoted by $\Zi_+(\Ad)$.  The {\em set of negative elementary dividing sets}
$\Zi_-(\Ad) = \Zi_+(\Ad)^\vee$ are obtained by reversing the positive and
negative regions. The {\em set of elementary dividing sets} is the union
$$\Zi(\Ad)= \Zi_+(\Ad) \cup \Zi_-(\Ad).$$
\end{defn}

\begin{theorem}\label{zarevgenthm}
Suppose that $(\Si,m)$ is a pointed oriented surface with boundary and
$(\Si,m)$ is parameterized by an arc diagram $\Ad$. Then the set of
elementary dividing sets $\Zi(\Ad)$ classically generate the contact
category $\KoSm$: any dividing set $\ga$ is homotopy equivalent to an
iterated extension of dividing sets $\z \in \Zi(\Ad)$.
%% $$\ga \simeq (\opp_i \z_{C_i}, p)$$
%% and the map $p$ consists of bypass moves between elementary dividing sets.
\end{theorem}
\begin{proof}
Suppose that $\ga$ is a dividing set on $\Si$. We will show that $\ga$ can be expressed in terms of elementary dividing sets. The proof will be divided into a number of steps.

{\bf First}. By Theorem \ref{surfacetheorem} we can assume that $\ga$ is boundary connected. 

{\bf Second}. Here we simplify $\ga$ within the $1$-handles of $F(\Ad)$.

Let $\{ c_1,\ldots, c_k\}$ be the set of cocores of 1-handles of $F(\Ad)$.
If $c_i$ is a cocore of a 1-handle in $\asurf$ and the intersection number 
$\vert \ga\cap c_i \vert > 2$ then there is a bypass disk with equator parallel to $c_i$ with
associated bypass triangle $\ga \to \ga'\xto{\th_B} \ga''\to\ga[1]$ with $\vert \ga'\cap c_i\vert, \vert \ga''\cap c_i\vert < \vert \ga \cap c_i\vert$. So $\ga$ is isomorphic to a cone:
$$\ga \cong C(\th_B) \conj{ such that } \vert \ga'\cap c_i\vert, \vert \ga''\cap c_i\vert < \vert \ga \cap c_i\vert.$$

Since $\ga$ bounds an orientable surface contained within the $1$-handle,
$\vert \ga\cap c_i \vert$ is even. In more detail, $\ga$ bounds $R\subset
\Si\backslash \ga$ so $R\cap c_i$ is a disjoint union of intervals. Since
the cardinality of the boundary of an interval is two, $\ga\cap c_i =\partial (R\cap c_i)$ is even.

Therefore, after iterating this procedure some number of times, we can assume that 
\begin{equation}\label{occeq}
\vert \ga\cap c_i\vert = 0 \conj{ or } \vert \ga\cap c_i\vert = 2 \conj{ for } 1\leq i \leq k.
\end{equation}

If the intersection number is $0$ then the $i$th 1-handle is {\em unoccupied} and if the number is $2$ then the $i$th 1-handle is {\em occupied}.
% Equation \eqref{occeq} is preserved by the remainder of the proof.

{\bf Third}. Here we simplify $\ga$ within the $0$-handles of $F(\Ad)$. 

After removing the cocores from the surface, one obtains a disjoint union of disks
\begin{equation}\label{acceq}
  \asurf\backslash\{c_1,\ldots, c_k\} = \amalg_{i=1}^\ell D^2_i.
  \end{equation}
The positive regions of a dividing set $\ga$ produced by the second step intersects the boundary of each such disk along intervals where occupied 1-handles are attached and the end points of the oriented line segment $\Ad_i \times [0,1] \subset \partial D^2_i$.

Let us formalize the situation which we will simplify in the remainder of the proof.  
Suppose $R$ is a positive region bounded by $\ga$, and $D_i$ is a disk from Eqn. \eqref{acceq} then $R$ is {\em disconnected in $D_i$} if $R\cap \partial D_i \neq \emptyset$ and $(R\cap D_i) \cap \Ad_i \times [0,1] = \emptyset$. A region $R$ is {\em disconnected} if $R$ is disconnected in $D_i$ for some disk $D_i$ in Eqn. \eqref{acceq}.

A dividing set $\ga$ is elementary if and only if there is one positive region in each disk. So
in order to express $\ga$ produced by step two in terms of elementary dividing
sets, we must reduce the number of disconnected regions.
(This is just a version of Theorem \ref{surfacetheorem} with the boundary
components $\Ad_i \subset \partial \Si$ treated separately.)

Let $R_1, \ldots, R_N$ be the positive regions of $\ga$ which are disconnected. Our complexity function is
$$n(\ga) := \sum_{i=1}^N \sum_{j =1}^\ell |\pi_0(R_i\cap \partial D_j)| \in \ZZ_{\geq 0}$$ 
the total number of 1-handles occupied by the disconnected regions. Notice that if $N>0$ then there exists an $R$ such that $R\cap \partial D_i \neq \emptyset$ and so $n(\ga) > 0$. On the other hand, if $n(\ga) = 0$ then there are no disconnected regions and $N=0$.

We claim that any $\ga$ which satisfies Eqn. \eqref{occeq} with $n(\ga) > 0$ can be
expressed as a twisted complex in dividing sets $\ga'$ which satisfy
$n(\ga') = 0$. Suppose $n(\ga) > 0$, then there is a disk $D_i$ which contains a disconnected region. Let $*$ be the positive region which contains $\Ad_i\times [0,1]\subset D_i$. Now follow the orientation around $D_i$ to the region $R$ disconnected in $D_i$ which is adjacent to $*$ and consider the bypass move illustrated below.
\begin{center}
\begin{overpic}%[grid,tics=10]
{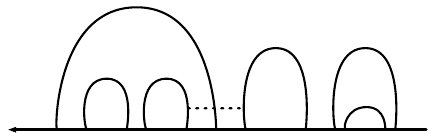}
 \put(129,25){$*$}
\put(60, 40){$R$}
\put(210, 0){$\partial D_i^2$}
\put(210, 60){$int(D_i^2)$}
\end{overpic}
\end{center}
This results in a triangle $\ga \to \ga' \to \ga''$ for which $n(\ga'), n(\ga'') < n(\ga)$.

Lastly, our dividing sets may still contain some positive regions which do
not intersect the boundary of any disk. Such regions can be removed with
Thm. \ref{surfacetheorem}.

\end{proof}

\begin{corollary}\label{zargencor}
  When a pointed oriented surface $\Si$ is parameterized by an
  arc diagram $\Ad$, the positive half of the formal contact category $\Ko_+(\Si)$ is generated
  by the positive elementary dividing sets $\Zi_+(\Ad)$.
\end{corollary}

\subsection{Decategorification}\label{kodecatsec}

\newcommand{\si}{\s} 

In this section we prove a variety of structural
properties and conjecture a decategorification statement for the formal
contact category.

\begin{prop}
A single bypass $\th : \ga \to \ga$ which takes $\ga$ to $\ga$ is capped. 
\end{prop}
\begin{proof}
One can make a small perturbation $a$ (or $b$) above (or below) of the equator $\ell$ of the bypass $\th$ as pictured on the lefthand side below. The bypasses associated to $a$ (or $b$) are isotopic to $\th$.
$$\begin{tikzpicture}[scale=10, node distance=2.5cm]
\node (B1) {$\CPPic{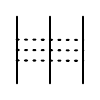}$};
\node (C1) [right=1.25cm of B1] {$\CPPic{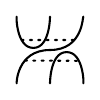}$};
\draw[->] (B1) to node  {$\th$} (C1);
\end{tikzpicture}$$
Now by assumption the righthand side, or the result of performing $\th$, is isotopic to the lefthand side. This isotopy takes the caps pictured on the righthand side to caps of the bypasses on the lefthand side. So $a$ and $b$ are capped. But $a$ and $b$ arose as perturbations of $\th$, so $\th$ is capped.
%% If $\th$ is capped so as to be zero then we are
%% done. Assume performing the bypass $\th$ gives a dividing set which is again
%% isotopic to $\ga$. 
%% Then the isotopy $H$ from $\th(\ga)$ to $\ga$ determines a cap.
%% %% push either
%% %% region $R$ or $S$ from their position on the righthand side to their
%% %% position on the lefthand side. $H_1(r)$ or $H_1(s)$ is the required cap on
%% %% $\th\ga$.
\end{proof}

\begin{prop}\label{decatprop}
Let $\Si$ be a surface with boundary together with a parameterization  $(\Ad,\vp_{\Ad})$.
  There is a surjective map 
$$\e : \FF_2\inp{\Ob(\Ko_+(\Si))} \to \Lambda^*H_1(F(\Ad), F(\partial \Ad);\FF_2)$$
 where $F(\partial \Ad) := \cup_i \Ad_i \subset \partial F(\Ad)$. 
This map satisfies the following property: if
  $$\ga\to \ga' \to \ga''$$
  is a bypass triangle then $\e(\ga'') = \e(\ga) + \e(\ga')$.
  \end{prop}
\begin{proof}
  A dividing set $\ga\subset \Si$ determines a collection of positive regions: if $\Si\backslash \ga = \sqcup_{i\in I} R_i$ then the set of positive regions is given by  $\aR := \{ i\in I : R_i \normaltext{ is positive } \}$. For each such region $R\in \aR$, let $\partial_+ R := \partial R \cap F(\partial \Ad)$, the pair $(R,\partial_+ R)$ gives an inclusion 
$$i_R : (R,\partial_+ R) \to (F(\Ad), F(\partial \Ad)).$$
Let $n_R := \dim H_1(R,\partial_+ R; \FF_2)$ so that $\Lambda^{n_R} H_1(R,\partial_+ R; \FF_2)$ is $1$-dimensional and there is a unique choice of non-zero vector $v_R \in \Lambda^{n_R} H_1(R,\partial_+ R; \FF_2)$.  Now tensoring gives a map 
$$\hat{i} : \otimes_{R\in\aR} \Lambda^{n_R} H_1(R,\partial_+ R; \FF_2) \xto{\bar{i}} \otimes_{R\in \aR} \Lambda^{n_R}H_1(F(\Ad),F(\partial \Ad);\FF_2) \hookrightarrow \Lambda^*H_1(F(\Ad),F(\partial \Ad);\FF_2)$$
  where $\bar{i} := \otimes_{R\in \aR} \wedge^{n_R} (i_R)_*$ and the last map is a composition of wedge products. The map $\e$ is defined to be
  $$\e(\ga) := \hat{i}(\wedge_{R\in\aR} v_R).$$

  The $1$-handles in $F(\Ad)$ span $H_1(F(\Ad),F(\partial \Ad); \FF_2)$. If $C$ corresponds to a subset of $1$-handles then by construction $\e(\z_C)$ is the wedge product of these classes in $\Lambda^*H_1(F(\Ad),F(\partial \Ad); \FF_2)$. Since wedge products of $1$-handles span the exterior algebra, $\e$ is onto.

  Additivity of $\e$ can be observed by examining how the bypass moves
  affect elements in the first homology. 
\begin{center}
\begin{tikzpicture}[scale=10, node distance=2.5cm]
\node (A2) {\CPic{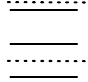}};
\node (B2) [right=1cm of A2] {};
\node (C2) [right=1cm of B2] {\CPic{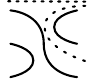}};
\node (D2) [below=1.41421356cm of B2] {\CPic{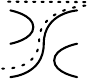}};
\draw[->, bend left=35] (A2) to node {} (C2);
\draw[->, bend left=35] (C2) to node {} (D2);
\draw[->, bend left=35] (D2) to node {} (A2);
\end{tikzpicture}
\end{center}
In the picture above the dashed arcs represent (local) choices of generators in a positive region. If the $\e(\ga) = A\wedge C$ and $\e(\ga') = B\wedge C$ are the wedge products of arcs depicted on the left and right respectively then $\e(\ga'') = (A+B)\wedge C$. The possible cases are handled similarly.

\end{proof}

\begin{cor}
Any bypass $\th : \z_C \to \z_{C'}$ between elementary dividing sets, the
third dividing set $\ga''$ in the associated bypass triangle,
\begin{equation}\label{trieq}
  \z_C \xto{\th}\z_{C'}\to \ga'',
  \end{equation}
is not an elementary dividing set.
\end{cor}
\begin{proof}
As above elementary dividing sets $\z_C$ determine basis vectors for $\Lambda^*H_1(\Si,\partial \Si;\FF_2)$ in a canonical way. Since $\e(\ga'')$ in Eqn. \eqref{trieq} must be a sum of the vectors determined by $\z_C$ and $\z_{C'}$ in this correspondence, it cannot be an elementary generator.
  \end{proof}

\begin{conjecture}
For any parameterization $\Ad$ of $\Si$, there is a map $\bar{\e}$, induced by $\e$, which is an isomorphism, as in the following diagram.
%  $$K_0(\Ko_+(\Si)) \xto{\sim} \FF_2\inp{\Zi_+(\Ad)}$$
\begin{center}
\begin{tikzpicture}[scale=10, node distance=2.5cm]
\node (A2) {$\FF_2\inp{\Ob(\Ko_+(\Si))}$};
\node (B2) [below=1cm of A2] {$K_0(\Ko_+(\Si))$};
\node (C2) [right=1cm of B2] {$\Lambda^*H_1(F(\Ad),F(\partial \Ad); \FF_2)$};
\draw[->] (A2) to node {$\e$} (C2);
\draw[->] (A2) to node {$\pi$} (B2);
\draw[->] (B2) to node {$\bar{\e}$} (C2);
\end{tikzpicture}
\end{center}
In the diagram above $\pi$ is the quotient map found in the definition of $K_0$.
\end{conjecture}

\subsubsection{Relation to work of J. Murakami and O. Viro}\label{MVsec}

The representation theory of the quantum group $U_q(\mathfrak{sl}_2)$ at
$q^4 = 1$ determines a degenerate instance of the Chern-Simons topological
field theory that has been related to the Alexander polynomial \cite{Murakami, Viro}. The Jones-Wenzl projector $p_3 \in \Endo_{U_q(\mathfrak{sl}_2)}(V^{\ott 3})$ takes the form:
$$p_3 = \MPic{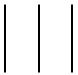} -\frac{d}{d^2-1}\left(\,\MPic{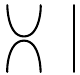} + \MPic{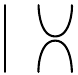}\!\! \right) + \frac{1}{d^2-1}\left(\,\MPic{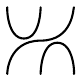} + \MPic{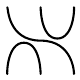}\!\! \right).$$
where $d = q+q^{-1}$. Taking $q = \sqrt{-1}$, gives $d = 0$ and $d^2 - 1 = -1$. This eliminates the middle term above, leaving
 the bypass triangle
$$p_3 = \MPic{n3-1} - \MPic{n3-e1e2} - \MPic{n3-e2e1}.$$
Since the righthand side should be zero, there is only a relationship between the contact geometry and representation theory after reducing by the Goodman-Wenzl ideal $\inp{p_3}$ \cite[Appendix]{Freedman}.

\newpage
\section{Comparison between categories associated to disks}\label{tiansec}

\newcommand{\Pow}{\aP}
\newcommand{\hi}{\bar}
\newcommand{\Q}{\aQ}
\newcommand{\R}{\aR}
\renewcommand{\P}{\aN}
\newcommand{\NL}{\aN}

In this section we show that the categories
associated to the disk $(D^2,2n)$ with $2n$ points by the Heegaard-Floer
theory $\alg(D^2,2n)$, the contact topology $\Co(D^2,2n)$ and the formal
contact construction are Morita equivalent.
$$\alg(D^2,2n) \cong \Co(D^2,2n) \cong \Ko_+(D^2,2n)$$
This is accomplished by choosing an arc parameterization $\zigzag_n$ of the
disk $(D^2,2n)$ so that the associated Heegaard-Floer category
$\alg(D^2,2n) \cong \alg(-\zigzag_n)$ has the same quiver presentation as
the algebraic contact category $\Co(D^2,2n) \cong \Yi_n$ studied by
Y. Tian. This equivalence is combined with Theorem \ref{zarevgenthm} to show
that both categories are Morita equivalent to the positive half of the
formal contact category $\Ko_+(D^2,2n)$. In this section, $n\geq 2$.

\subsection{The Heegaard-Floer categories associated to a disk}\label{HFdisksec}
In this section an arc diagram $\zigzag_n$ and an arc parameterization of
the disk $(D^2,2n)$ with $2n$ marked points by $\zigzag_n$ are
introduced. The Bordered Sutured Floer theory developed by R. Zarev
associates a dg category $\alg(\zigzag_n)$ to this parameterization. In
Section \ref{HFCodisksec}, we will find that this category is the same as
Y. Tian's quiver algebra $\R_n$. 

The disk will be oriented in the opposite direction of later sections. In
this way the boundary of the disk is oriented clockwise. When viewed from
above, as in the illustration below, each interval $\Ad_i\subset \partial D$
has a well-defined left direction (counterclockwise) and right direction
(clockwise). This terminology is used by the definition below.

\begin{defn}\label{zigzagdef}
  The {\em zig-zag arc diagram} $\zigzag_n$  is defined inductively as follows:
  \begin{enumerate}
  \item The arc diagram $\zigzag_2$ consists of two lines $\lns = \{ \Ad_1, \Ad_2\}$ and two points $\pts = \{a_1, a'_1\}$ where, $a_1 \in \Ad_1$, $a'_1\in \Ad_2$ and $M(a_1) = M(a'_1)$.
\item If $n$ is odd then $\zigzag_n$ is obtained from $\zigzag_{n-1}$ by adding a new line $\Ad_{n}$, containing the point $a_{n-1}$, to the right of the line $\Ad_{n-2}$ and adding the point $a'_{n-1}$ to the line $\Ad_{n-1}$ immediately to the left of $a'_{n-2}$.
\item If $n$ is even then $\zigzag_n$ is obtained from $\zigzag_{n-1}$ by adding a new line $\Ad_n$, containing the point $a'_{n-1}$, to the left of $\Ad_{n-2}$ and then adding the point $a_{n-1}$ to $\Ad_{n-1}$ to the right of the point $a_{n-2}$.
  \end{enumerate}
\end{defn}

If we imagine the line segments $\{\Ad_i\}_{i=1}^n$ to be embedded
sequentially along the real line $\RR$ then an orientation on each line
segment is induced by choosing an orientation of $\RR$; they all point
either to the left or to the right. The name zig-zag becomes clear after
rearranging the line segments into a zig-zag pattern.

\begin{center}
\begin{overpic}[scale=2]%,grid,tics=10]
{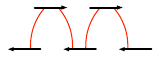}
\put(20,-5) {$\Ad_5$}
\put(70,-5) {$\Ad_3$}
\put(125,-5) {$\Ad_1$}
\put(42.5,52) {$\Ad_4$}
\put(95,52) {$\Ad_2$}

\put(125,25) {$h_1$}
\put(90,25) {$h_2$}
\put(51.5,25) {$h_3$}
\put(15,25) {$h_4$}
\end{overpic}
\end{center}
The arc diagram for $\zigzag_5$ is pictured above. The line labelled $\Ad_i$ is the $i$th  line segment in the construction from Definition \ref{zigzagdef}.  The lines $h_i$
connect the matched pairs $\{ a_i, a'_i \}$.
 If the illustration
above is understood to specify an embedding of the arc diagram into the
plane then thickening each of the components produces the parameterization
of the disk $(D^2, 2\cdot 5)$ with $10$ points pictured below.
\begin{center}
\begin{overpic}[scale=2]%,grid,tics=10]
{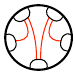}
\end{overpic}
\end{center}
Giving the plane the standard $\inp{x,y}$ orientation induces an orientation
on $(D^2,2n)$ in which the boundary is oriented clockwise.

The proposition below may be clear to readers who are more familiar with the algebras involved.

\begin{prop}
  The dg category $\alg(-\zigzag_n)$ has trivial differential: $d = 0$.
\end{prop}
\begin{proof}
This follows from the definition of the differential. In more detail, by construction, as an algebra with idempotents, the dg category
  $\alg(\zigzag_n)$ is a subalgebra of a tensor product of copies of strands
  algebras $\alg(1)$ and $\alg(2)$. Neither of these algebras have
  differentials. Any tensor product of algebras without differentials does
  not have a differential. Any subalgebra of an algebra without differential
  does not have a differential, see also \cite[Prop. 9.2]{Zarev1}.
\end{proof}

Without a differential, the dg category $\alg(-\zigzag_n)$ is a category. The definition below comes from \cite[\S 2.3]{Zarev1}.  It is summarized in Def. \ref{poopdef}.

First note that the idempotents in this construction correspond to the
objects of the category $\alg(-\zigzag_n)$, the idempotents are indexed by a
choice of a subset
$$S\subset \{ h_1,\ldots,h_{n-1} \}$$
of the 1-handles which identify matched pairs in the arc diagram $\zigzag_n$ \cite[Def. 2.5]{Zarev1}.  

In Definition \ref{zigzagdef} of $\zigzag_n$ above, there are $n$ line segments
 $\{\Ad_1,\ldots, \Ad_n\}$. On the segment $\Ad_1$, there is only one point $a_1$.
If $n$ is even then $\Ad_n$ contains only one point $a'_{n-1}$. If $n$ is odd then $\Ad_n$ contains only the point $a_{n-1}$. 
The line segment $\Ad_k\in \{\Ad_2,\ldots,\Ad_{n-1}\}$ contains the two points described below.  
\begin{equation}\label{eveneqn}
a'_k a'_{k-1} \normaltext{  for $k$ even }  \conj{ or } a_k a_{k+1} \normaltext{  for $k$ odd } \end{equation}
%% when $n$ is even. When $n$ is odd, either the points
%% \begin{equation}\label{oddeqn}
%% a'_k a'_{k-1} \normaltext{  for  }  \conj{ or } a_k a_{k+1} \normaltext{  for  } 1 < k \leq n-2.
%% \end{equation}
Since the algebra $\alg(1)$ only contains the identity element, the non-identity elements in the parts of $\alg(-\zigzag_n)\subset \alg(1)\ott \alg(2)^{\ott n-2} \ott \alg(1)$ correspond to the $\alg(2)$-tensor factors. Each such factor contains a Reeb chord $\rho_{k,k+1}$ or $\rho_{k+1,k}$. If the line segment contains the points $a'_{k+1} a'_k$ then the Reeb chord $\rho_{k,k+1}$ connects $\rho^-_{k,k+1} = a'_k$ to $\rho^+_{k+1,k} = a'_{k+1}$. If the line segment contains the points $a_{k} a_{k+1}$ then the Reeb chord $\rho_{k+1,k}$ connects $\rho^-_{k+1,k} = a_{k+1}$ to $\rho^+_{k+1,k} = a_{k}$. Since the $k$th 1-handle $h_k$ corresponds to the matching of the pair $a_k$ and $a'_k$, the Reeb chords $\rho_{k,k+1}$ and $\rho_{k+1,k}$ correspond to maps:
\begin{equation}\label{rhogeneqn}
\rho_{k,k+1} : h_k \to h_{k+1} \conj{and} \rho_{k+1,k} : h_{k+1} \to h_k
\end{equation}
Translating Equation \eqref{eveneqn}
%and \eqref{oddeqn} 
above into the language of Equation \eqref{rhogeneqn} tells us when such maps can be found in the category $\alg(-\zigzag_n)$. If $n$ is even then there are maps:
$$h_{n-1} \xto{\rho_{n-1,n-2}} h_{n-2} \xfrom{\rho_{n-3,n-2}} h_{n-3} \to \cdots \from h_3 \xto{\rho_{3,2}} h_2 \xfrom{\rho_{1,2}} h_1$$
and if $n$ is odd then there are maps:
$$h_{n-1} \xfrom{\rho_{n-2,n-1}} h_{n-2} \xto{\rho_{n-2,n-3}} h_{n-3} \from \cdots \from h_3 \xto{\rho_{3,2}} h_2 \xfrom{\rho_{1,2}} h_1.$$
Increasing the number $n$ by one has the effect of adding one new Reeb chord. 

The generators of the full category $\alg(-\zigzag_n)$ are obtained by extending each Reeb chord by identity in all possible ways \cite[Def. 2.9]{Zarev1}. In more detail, if $S = h_{i_1} h_{i_2}\cdots h_{i_j} \cdots h_{i_{k-1}} h_{i_k}$ is a subset of 1-handles which have been ordered so that $i_j<i_{j+1}$ then there is a generator:
\begin{equation}\label{diskgeneqn}
h_{i_1} h_{i_2}\cdots h_{i_j} \cdots h_{i_{k-1}} h_{i_k}\to h_{i_1} h_{i_2}\cdots h_{i_j \pm 1} \cdots h_{i_{k-1}} h_{i_k}
\end{equation}
in $\alg(-\zigzag_n)$ when there is a Reeb chord $\rho_{i_j,i_j \pm 1} : h_{i_j} \to h_{i_j \pm 1}$ as above and the 1-handle $h_{i_{j\pm 1}}$ isn't contained in set $S$:
$$i_{j \pm 1} \not\in \{ i_1,i_2,\ldots,i_k\}.$$

None of the relations satisfied by the strands algebras apply in our context because the Reeb chords are contained in independent strands algebras $\alg(2)$ of order two. There is only one relevant family of relations, stemming from the observation that maps applied to independent tensor factors commute.
\begin{equation}\label{comeqn}
\begin{tikzpicture}[scale=10, node distance=2.5cm]
\node (B)  {$\cdots h_{i_j} \cdots h_{i_\ell} \cdots $};
\node (Bp) [right=1.5cm of B] {};
\node (Bu) [above=.5cm of Bp] {$\cdots h_{i_j \pm 1} \cdots h_{i_\ell} \cdots $};
\node (Bd) [below=.5cm of Bp] {$\cdots h_{i_j } \cdots h_{i_\ell \pm 1} \cdots $};
\node (Bq) [right=1.5cm of Bp] {$\cdots h_{i_j \pm 1} \cdots h_{i_\ell \pm 1} \cdots $};
\draw[->] (B) to node {} (Bu);
\draw[->] (B) to node [swap] {} (Bd);
\draw[->] (Bu) to node {} (Bq);
\draw[->] (Bd) to node [swap] {} (Bq);
\end{tikzpicture}\end{equation}
Said differently, whenever generators can be applied out-of-order to form a square, as pictured above, this square must commute.

The definition below summarizes the discussion above.

\begin{definition}\label{poopdef}
$\alg(-\zigzag_n)$ is the dg category with $d=0$. The objects $\Ob(\alg(-\zigzag_n) = \{ S : S\subset \{h_1,\ldots,h_{n-1}\} \}$ are subsets of the set of arcs in Def. \ref{zigzagdef}. We write $S = \prod_{h_{i_k}\in S} h_{i_k}$ for any $S\in \Ob(\alg(-\zigzag_n))$. The category $\alg(-\zigzag_n)$ is generated by maps of the form Eqn. \eqref{diskgeneqn} subject to relations in Eqn. \eqref{comeqn}.
  \end{definition}

The examples below will be compared to Examples \ref{q3ex} and \ref{q4ex} in  Section \ref{HFCodisksec} later.

\begin{example}\label{hq3ex}
The structure of $\alg(-\zigzag_3)$ can be pictured in the following way:
$$\begin{tikzpicture}[scale=10, node distance=2.5cm]
\node (A) {$\emptyset$};
\node (B) [right=1cm of A] {$h_1$};
\node (C) [right=1cm of B] {$h_2$};
\node (D) [right=1cm of C] {$h_1 h_2$};
\draw[->] (B) to node {$\rho_{1,2}$} (C);
\end{tikzpicture}$$
\end{example}

\begin{example}\label{hq4ex}
The structure of $\alg(-\zigzag_4)$ is illustrated by the diagram below:
$$\begin{tikzpicture}[scale=10, node distance=2.5cm]
\node (A) {$\emptyset$};
\node (B) [right=1cm of A] {$h_1h_3$};
\node (Bp) [right=1.5cm of B] {};
\node (Bu) [above=.5cm of Bp] {$h_2h_3$};
\node (Bd) [below=.5cm of Bp] {$h_1 h_2$};

\node (Cp) [right=1cm of Bp] {};
\node (Cu) [above=.5cm of Cp] {$h_1$};
\node (Cd) [below=.5cm of Cp] {$h_3$};
\node (D) [right=1.5cm of Cp] {$h_2$};
\node (E) [right=1cm of D] {$h_1 h_2 h_3$};

\draw[->] (B) to node {$\rho_{1,2}$} (Bu);
\draw[->] (B) to node [swap] {$\rho_{3,2}$} (Bd);

\draw[->] (Cu) to node {$\rho_{1,2}$} (D);
\draw[->] (Cd) to node [swap] {$\rho_{3,2}$} (D);
\end{tikzpicture}$$
\end{example}

\begin{remark}
Bordered Sutured theory usually associates different algebras to
  different parameterizations of a surface. The categories of modules associated to these algebras are equivalent. In this sense, the algebras associated to surfaces are Morita equivalent, see Appendix.
In order to understand why this is the case, consider that the mapping cylinder 3-manifolds associated to a diffeomorphism between parameterizations and its inverse determine a pair of bimodules \cite[\S 8]{Zarev1}. Product with a bimodule determines a functor between modules over algebras. The composition of functors gives the bimodule associated to identity which is algebraically identity \cite[\S 8.6]{Zarev1}. See also \cite{Zarev2}.

In particular, there is an arc
  parameterization $\aW_n$ \cite[Ex. 9.1]{Zarev1} for which there is an
  isomorphism of dg categories $\alg(\aW_n) \cong \alg(n-1)^\op$
  \cite[Prop. 9.1]{Zarev1}, where $\alg(n-1)$ is the strands algebra
  \cite[\S 3.1]{LOT}. Therefore, $\alg(-\zigzag_n) \cong \alg(n-1)^\op$ in
  $\Hmo$.
\end{remark}

\subsection{The contact category associated to a disk}\label{Codisksec}
Here we introduce the category $\Yi_n$ that Y. Tian associates to
the disk with $2n$ boundary points \cite{YT1}. We will not discuss gradings.

\subsubsection{Indexing multicurves with nil-Temperley-Lieb notation}\label{nilTLsec}
Monomials in the nil-Temperley-Lieb algebra, will be used to denote
multicurves $\ga\subset (D^2,2n)$ in the disk. In particular, multicurves
determined by monomoials $e_{i_1} e_{i_2} \cdots e_{i_k}$, which have been
ordered, so as to satisfy $i_1 < i_2 < \cdots < i_k$, correspond to the
objects in Y. Tian's construction, see Definition \ref{YTquiverdef}.

\begin{defn}\label{Pdef}
    The {\em nil-Temperley-Lieb algebra} $\P_n$ is the $k$-algebra on generators:
    $e_i$, $1\leq i < n$, subject to the relations:
\begin{enumerate}
\item $e_i^2 = 0$ for $1\leq i < n$
\item $e_i e_j = e_j e_i$ for $\vnp{i - j} > 2$ and
\item $e_i e_{i\pm 1} e_i = e_i$.
\end{enumerate}
\end{defn}

If the ground ring $k$ is changed to $\ZZ[q,q^{-1}]$ and the first relation
is changed from $e_i^2 = 0$ to $e_i^2 = q+q^{-1}$ then the algebra $\P_n$
introduced above becomes the well-known Temperley-Lieb algebra
$\mathcal{T\! L}_n$, see \cite{KL}. 

The relationship
between the Temperley-Lieb algebra and the planar algebra of multicurves
extends to the nil-variant $\P_n$ introduced above. There is a basis for the algebra $\P_n$ consisting of monomials which is in
one-to-one correspondence with isotopy classes of boundary connected
multicurves in a pointed oriented disk $(D^2,2n)$. This can be seen after
each generator $e_i$ is identified with a multicurve $\ga(e_i)$.
$$e_i\mapsto \ga(e_i)$$
If the disk is pictured so that the first $n$ points are situated on the top of the disk and
the last $n$ points are situated on the bottom of the disk then all of the
strands of $\ga(e_i)$ are vertical except for two which connect the $i$th and
$(i+1)$-st points in each collection. The products, $\ga(e_ie_j) = \ga(e_i)\ga(e_j)$, of generators correspond to vertically stacking the multicurves. For instance, when $n=3$ we have the following
pictures:
$$\ga(1) = \CPPic{n3-1}\hspace{-.25in},\quad \ga(e_1) = \CPPic{n3-e1}\quad \normaltext{ or }\quad \ga(e_1 e_2) = \CPPic{n3-e1e2}\hspace{-.25in}.$$
In the image of the map $\ga$, the second and third relations in
Definition \ref{Pdef} correspond to isotopy and the first relation implies
that any multicurve containing a homotopically trivial component is zero.

This observation can be used to construct a set map $\ga$ from the monomials
the nil-Temperley-Lieb algebra $\P_n$ to positive dividing sets on
$(D^2,2n)$. Since all of the defining relations for $\P_n$ preserve
monomiality: the product of monomials is a monomial and each monomial
$x\in\P_n$ corresponds to a multicurve $\ga(x)$. After signing the regions
of $D^2\backslash \ga(x)$, this determines a dividing set on the disk.
Knowledge of the map $\ga$ is assumed throughout the next section.

\subsubsection{Y. Tian's disk category}\label{tianconstructionsec}

Y. Tian's category $\Yi_n$ is introduced by the sequence of definitions
below. The construction presented here is equivalent to the original \cite{YT1}. However, we will use the algebra $\P_n$ to
express the presentation in a more familiar notation.

\begin{defn}\label{YTquiverdef}
  The quiver $\Q_n$ has vertices $V := \{ S = \{ i_1 < i_2 < \cdots < i_k : 1\leq i_j < n, j = 1,\ldots,k \}$ and edges
  $$E(S,T) := \left\{ \begin{array}{ll}
    \{ \theta_p \} & \normaltext{ if } \vert T \vert = \vert S \vert + 2 \normaltext{ and } T = S \cup \{p,p+1\}\\
    \emptyset & \normaltext{ otherwise }
    \end{array}\right.$$

In more detail, the vertices $S$ of the quiver $\Q_n$ are the ordered monomials:
$$e_S = e_{i_1} e_{i_2} \cdots e_{i_k} \in \P_n \conj{ where } S = \{i_1< i_2 <\cdots <i_k\}.$$
and $1\leq i_j < n$ for $j = 1,\ldots,k$ in the nil-Temperley-Leib algebra. There is an edge
$\th_p : e_S \to e_T$ from $e_S$ to $e_T$ when the set $T$ can be obtained from the set $S$ by adjoining the disjoint subset $\{p,p+1\}$. 
%(In symbols: $\vert T \vert = \vert S \vert + 2$ and $T = S \cup \{p,p+1\}$.)
\end{defn}

  Before introducing the category $\Yi_n$, the definition above is
  illustrated by the examples below.

\begin{example}\label{q3ex}
  When $n=3$, the quiver $\Q_3$ assumes a rather unassuming form:
$$\begin{tikzpicture}[scale=10, node distance=2.5cm]
\node (A) {$e_1$};
\node (B) [right=1cm of A] {$1$};
\node (C) [right=1cm of B] {$e_1 e_2$};
\node (D) [right=1cm of C] {$e_2$};
\draw[->] (B) to node {$\th_1$} (C);
\end{tikzpicture}$$
\end{example}

\begin{example}\label{q4ex}
When $n = 4$, the quiver $\Q_4$ is more complicated:
$$\begin{tikzpicture}[scale=10, node distance=2.5cm]
\node (A) {$e_1 e_3$};
\node (B) [right=1cm of A] {$1$};
\node (Bp) [right=1.5cm of B] {};
\node (Bu) [above=.5cm of Bp] {$e_1 e_2$};
\node (Bd) [below=.5cm of Bp] {$e_2 e_3$};

\node (Cp) [right=1cm of Bp] {};
\node (Cu) [above=.5cm of Cp] {$e_1$};
\node (Cd) [below=.5cm of Cp] {$e_3$};
\node (D) [right=1.5cm of Cp] {$e_1 e_2 e_3$};
\node (E) [right=1cm of D] {$e_2$};

\draw[->] (B) to node {$\th_1$} (Bu);
\draw[->] (B) to node [swap] {$\th_2$} (Bd);

\draw[->] (Cu) to node {$\th_2$} (D);
\draw[->] (Cd) to node [swap] {$\th_1$} (D);
\end{tikzpicture}$$
\end{example}

Each arrow $\th_p : e_S \to e_T$
  corresponds to a bypass move $\ga(e_S)\to \ga(e_T)$ between the multicurves $\ga(e_S)$ and $\ga(e_T)$, involving the $p$th  and $p+1$st regions in the disk, see Equation \eqref{tianbypasseqn}.

The disk category $\R_n$ is the category generated by the graph $\Q_n$, modulo
the relation that compositions of disjoint bypass moves commute.

\begin{defn}\label{diskdef}
The {\em disk category} $\R_n$ is the $k$-linear category generated by the graph $\Q_n$ 
subject to the relations:
$$\th_p \th_q = \th_q \th_p \conj{ for each pair of arrows } \th_p \th_q, \th_q \th_p : e_S \to e_T \normaltext{ in } \Q_n. $$
\end{defn}

The disk category $\R_n$ can be viewed as a dg category with $d=0$. Recall the
notion of pretriangulated hull from Section \ref{trisec}.

\begin{defn}\label{hulldiskdef}
The category $\Yi_n$ associated to the disk $(D^2,2n)$ is the
  pretriangulated hull of the disk category $\R_n$:
$$\Yi_n = \R_n^\pt.$$
\end{defn}

\subsection{Relationship between the contact category and the Heegaard-Floer category}\label{HFCodisksec}
Here we show that the category $\alg(-\zigzag_{n})$ 
found in Section \ref{HFdisksec} is isomorphic to Y. Tian's disk category $\R_n$
from Section \ref{tianconstructionsec}.

\begin{theorem}
$$\R_n \xto{\sim} \alg(-\zigzag_n)$$
\end{theorem}
\begin{proof}
  The similarities between Examples \ref{hq3ex}, \ref{q3ex} and Examples
  \ref{hq4ex}, \ref{q4ex} are suggestive. We will discuss the case when $n$
  is even, the case when $n$ is odd is similar. We first give a bijective
  correspondence between the objects in either category. After this the
  generators in either category are related to one another by representing
  each by a geometric bypass moves.

 There is a one-to-one correspondence between the objects in each category.
 Recall that for $\R_n$ the objects $\Ob(\R_n) = V(\Q_n) = \{ e_S : S = \{
 i_1 < i_2 < \dots < i_k \} \}$ which correspond to multicurves in the disk
 determined by the product $e_{S} = e_{i_1} \cdots e_{i_k}$ in the
 nil-Temperley-Lieb algebra. For $\alg(-\zigzag_n)$, the objects are
 $\Ob(\alg(-\zigzag_n)) = \{ S : S \subset \{ h_1, h_2, \ldots, h_{n-1} \}
 \}$ which correspond to a selection of $1$-handles in the zig-zag diagram. The
 maps in the next two paragraphs are constructed using these two topological
 interpretations for $S$.

First we construct a map  $\Phi : \Ob(\R_n) \to \Ob(\alg(-\zigzag_n))$.
In this correspondence the identity diagram $1\in\P_n$ corresponds to selecting all of the odd 1-handles, $\Phi(1) = h_1 h_3 \cdots h_{n-1}$.
  Suppose that $e_S=e_{i_1} e_{i_2} \ldots e_{i_k}\in \P_n$ is an ordered
  monomial. Then to construct the selection of 1-handles in $\Ob(\alg(-\zigzag_n))$ associated to $e_S$ we perform surgery on this identity surface $h_1 h_3 \cdots h_{n-1}$ along the arcs pictured below for each $e_{i_k}$ appearing in $e_S$.

%cut the identity diagram along the odd generators $e_{2p+1}$ in $w$ and glue 1-handles along the even generators $e_{2q}$ as indicated by the picture below:
\begin{center}
\begin{overpic}[scale=2]%,grid,tics=10]
{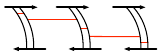}
\put(3,37) {$e_1$}
\put(43,37) {$e_2$}
\put(87,23) {$e_3$}
\put(105,22) {$e_4$}
\put(147,12) {$e_5$}
\end{overpic}
\end{center}
After performing this surgery, there is a uniquely determined set $S \subset \{h_1,\ldots,h_{n-1}\}$ of 1-handles in the arc diagram $\zigzag_n$ corresponding to this surface; this is the map from ordered monomials to subsets $S$ of the set of 1-handles.

Now we construct an inverse map  $\Psi : \Ob(\alg(-\zigzag_n)) \to \Ob(\R_n)$.
The empty set of 1-handles $\emptyset$ corresponds to the product of the odd generators $\Psi(\emptyset) = e_1e_3 \cdots e_{n-1}$. If $h_{i_1} h_{i_2} \cdots h_{i_k}$ is an arbitrary selection of 1-handles then gluing each 1-handle $h_{i_j}$ into the picture below, in the indicated fashion, uniquely determines a multicurve associated to a positive monomial.
\begin{center}
\begin{overpic}[scale=2]%,grid,tics=10]
{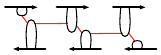}
\put(8,27) {$h_5$}
\put(43,35) {$h_4$}
\put(80,30) {$h_3$}
\put(98,25) {$h_2$}
\put(135,15) {$h_1$}
\end{overpic}
\end{center}

The maps introduced above are inverse. There is a bijection between the objects of either category. Observe that performing the odd $e_i$ surgeries in the first illustration above produces the picture below it. From this observation the following two rules below can be deduced: 
\begin{enumerate}
\item If $i$ is odd then the effect of choosing or not choosing $e_i$ corresponds to removing or adding $h_{n-i}$. 
\item If $i$ is even then the effect of choosing or not choosing $e_i$ corresponds to adding or removing $h_{n-i}$. 
\end{enumerate}
Here it is in algebraic notation.
$$\Phi(e_{i_1} e_{i_2} \cdots e_{i_k}) = \{ h_{n-s} : \exists r, s = i_r\,\, \mathrm{ and }\,\, s\,\, \mathrm{ even } \} \\
\cup \{ h_{n-s} : \forall r, s \neq i_r\,\, \mathrm{ and }\,\, s\,\, \mathrm{ odd } \}$$
$$\Psi(\{h_{i_1}, h_{i_2}, \ldots, h_{i_k}\}) = \{ e_{n-s} : \exists r, s = i_r\,\, \mathrm{ and }\,\, s\,\, \mathrm{ even } \} \\
\cup \{ e_{n-s} : \forall r, s \neq i_r\,\, \mathrm{ and }\,\, s\,\, \mathrm{ odd } \}$$
The variable $r$ is restricted to the relevant subset of indices and the subscripts of a word $e_{i_1} e_{i_2} \cdots e_{i_k}$ are placed in order so as to coincide with conventions. These rules determine a bijection.

If $w,w'\in\P_n$ are ordered monomials then an arrow $\th_p : ww' \to w e_p e_{p+1} w'$ in the graph $\Q_n$ corresponds to the bypass move $\th_p : \ga(ww') \to \ga(w e_p e_{p+1} w')$ pictured below,
\begin{equation}\label{tianbypasseqn}
\th_p = \CPPic{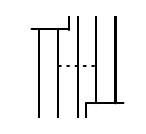}
\end{equation}
For example, after a rotation, the only arrow in the quiver $\Q_3$ corresponds to the
bypass illustrated before Definition \ref{disjdef}. On the other hand, the basic Reeb chords: $\rho_{k,k+1} : h_k\to h_{k+1}$ and $\rho_{k+2,k+1} : h_{k+2} \to h_{k+1}$ from Section \ref{HFdisksec} correspond to the pictures:
\begin{equation}\label{reebcorrespeqn}
\BCPPic{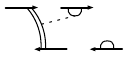}\conj{ and } \BCPPic{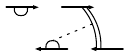}
\end{equation}
so that the two combinatorial notions perform the same function between
multicurves in the correspondence between the objects.

There are no relations in either category besides the commutativity of diagrams in Equation \ref{comeqn} and Definition \eqref{diskdef}.
\end{proof}

\subsection{Relationship between the disk category and the formal contact category}\label{reldisksec}

In this section we will construct a Morita equivalence between the
Heegaard-Floer category $\alg(-\zigzag_n)$ considered in Section
\ref{HFdisksec} and the formal contact category $\Ko_+(D^2,2n)$.

The discussion in prior sections sufficies to define a functor:
$$\mu : \alg(-\zigzag_n) \to \Ko_+(D^2,2n).$$
To each collection of 1-handles $C = h_{i_1} h_{i_2} \cdots h_{i_k}$ we
associate the elementary generator $\z_C \in \Ob(\PKo_+(D^2,2n))$. The basic
Reeb chords correspond to the bypass moves pictured in Equation
\eqref{reebcorrespeqn} above. Composing this functor with the quotient map $Q : \PKo_+(D^2,2n)\to \Ko_+(D^2,2n)$ yields $\mu$ above.

\begin{theorem}\label{tianthm}
The functor $\mu : \alg(-\zigzag_n) \to \Ko_+(D^2,2n)$ determines a Morita equivalence.
\end{theorem}

The proof of the theorem will use the fact that if $\eA$ and $\eC$ are small
dg categories then $\eA$ is Morita equivalent to $\eC$ when $\eC$ is
quasi-equivalent to a full dg subcategory $\eB$ of the category of $\eA$
whose objects form a set of small generators. This is a special case of a
more general statement \cite[Thm. 8.2]{Kellerder}.

\begin{proof}
Using Theorem \ref{zarevgenthm}, it suffices to check that for each pair of collections of 1-handles $C,C'$ the maps:
$$\mu_{C,C'} : \Hom_{\alg(-\zigzag_n)}(C,C') \to \Hom_{\Ko_+(D^2,2n)}(\z_C,\z_{C'})$$
are quasi-isomorphisms. Since the trivial bypasses must bound caps and are
removed by relation (1) in Definition \ref{pkosdef}.  The only bypasses
$\z_C\to \z_C'$ between elementary generators are those that appear in
Equation \eqref{reebcorrespeqn}. These bypasses and their compositions are the cycles in $\PKo_+(D^2,2n)$. It suffices to show that they remain cycles in the quotient.

 % Since the only bypasses $\z_C\to \z_C'$ between

The remainder follows from the commutativity of pushouts:
$$L_S L_{S'} \eC \cong L_{S\amalg S'} \eC \cong L_{S'} L_S \eC$$
and the observation that the maps
$Q_{C,C'} : \Hom_{\eC}(C,C')\to\Hom_{L_S\eC}(C,C')$ are quasi-isomorphisms
for any single Postnikov localization. The latter can be seen by identifying
a single Postnikov localization as an instance of Drinfeld localization
under the Yoneda embedding, see Proposition \ref{sesprop}. The Drinfeld
localization modifies the homological structure of the morphisms by adding a
single map $h$ which is a boundary $dh = 1_\kr$ where $\kr$ is as in the
proof of Proposition \ref{sesprop}. This makes any cycle to or from $\kr$
into a boundary, but does not create any other boundaries. Since $\kr$ is
not an elementary generator $\z_C$ for some $C$, the result follows.
\end{proof}

\subsection{Dualities}
Our discussion concludes with some mention of dualities. In Examples
\ref{q3ex} and \ref{q4ex}, duality is found in the lateral symmetry of the
graph $\Q_n$. If $[n]$ denotes an ordered set $\{1 < 2 < \cdots < n\}$ then
the assignment: 
$$e_{S}^y = e_{[n]\backslash S}$$ 
determines a contravariant involution:
$$-^y : \Yi_n^{\op} \to \Yi_n.$$

In $\Yi_n$ there are no signed regions and the lateral symmetry is
contravariant; so the functor $-^y$ cannot directly correspond to a functor,
such as $-^\vee$, between formal contact categories. The proposition records
the correct formulation. The proof is left to the reader.

\begin{prop}\label{ytdualprop}
  The diagram below commutes,
$$\begin{tikzpicture}[every node/.style={midway},node distance=2cm]
  \matrix[column sep={4cm,between origins}, row sep={2cm}] at (0,0) {
    \node(R) {$\Yi_n^\op$}  ; & \node(S) {$\Ko_+(D^2,2n)^\op$}; \\
    \node(R/I) {$\Yi_n$}; & \node (T) {$\Ko_+(D^2,2n)$};\\
  };
  \draw[<-] (R/I) -- (R) node[anchor=east]  {$-^y$};
  \draw[->] (R) -- (S) node[anchor=south] {$\fii_+^\op$};
  \draw[->] (S) -- (T) node[anchor=west] {$\a$};
  \draw[->] (R/I) -- (T) node[anchor=south] {$\fii_+$};
\end{tikzpicture}$$
where the functor $\a = (-)^\vee \circ \bar{(-)} \circ (r^{-1})^\op$ is the composition of three equivalences: $r$ is the element of the mapping class group which rotates the basepoint $z$ by one region clockwise (Corollary \ref{rotationprop}), $\bar{(-)}$ is reverses the orientation of the disk (Proposition \ref{orrevprop}) and $(-)^\vee$ changes the signs of the regions (Proposition \ref{dualityfunprop}).
\end{prop}

\newpage
\section{Linear Bordered Heegaard Floer categories}\label{BFalgsec}

Within the framework of the Bordered Heegaard Floer theory, a differential
graded category $\alg(\Ad)$ is associated to each arc diagram $\Ad$.  For
some choices this category satisfies $d=0$ and it is possible to write down
a quiver presentation.  In this section, these categories are related to the
corresponding formal contact categories. Functors are defined:
\begin{align*}
\alg(-\Ad_{0,n},1-n) \xto{\s_n} \Ho(\Ko_+^{2n-4}(\Si_{0,n},n\cdot 2)) \quad\normaltext{ and }\\
\alg(-\Ad_{g,1},2g-1) \xto{\tau_g} \Ho(\Ko_+^{2g-2}(\Si_{g,1},2))
\end{align*}
where $\Ad_{0,n}$ and $\Ad_{g,1}$ are arc diagrams which parameterize
surfaces, $\Si_{0,n}$ and $\Si_{g,1}$, of genus zero with $n$ boundary
components and of genus $g$ with one boundary component respectively. We fix
two points on every boundary component and require that $n>1$ and $g >0$. 

The bordered algebras studied in this section are the ``one moving strand'' algebras corresponding to the second largest weight, see \cite[\S 2]{Zarev1}, \cite[\S 2]{LOTMCG} or \cite[\S 3]{LOT}.

\subsection{A surface $\Si_{0,n}$ of genus $0$ with several boundary components}

When $n$ disks are removed from the 2-sphere
$$\Si_{0,n} = S^2\backslash \amalg_{i=1}^n D^2\conj{ } n > 1,$$
and two points are fixed on each of its boundary components, the resulting surface can be
pararameterized by the arc diagram $\Ad_{0,n}$ found in the definition
below. 

\begin{defn}\label{snparamdef}
  The arc diagram $\Ad_{0,n}$ consists of $n$ oriented line segments
  $\lns = \{ \Ad_1,\Ad_2,\ldots,\Ad_n\}$. On the first line segment $\Ad_1$
  there are $3n-3$ points and there is one point on each of the remaining line segments $\{\Ad_2\ldots \Ad_n\}$:
$$\Ad_1 = a_1 b_1 a'_1  a_2 b_2 a'_2 \cdots a_{n-1} b_{n-1} a'_{n-1} \conj{ and } \Ad_i = b'_{i-1} \normaltext{ for } 2\leq i \leq n.$$ 
The set of points is given by $\pts = \{ a_i,a'_i,b_i,b'_i : 1\leq i < n \}$.
The line $\Ad_1$ is oriented so that the subscripts of the points increase in value. The matching function is determined by the rules: $M(a_i) = M(a'_i)$ and $M(b_i) = M(b'_i)$.
\end{defn}

The annulus $\Si_{0,2}$ and its parameterization by $\Ad_{0,2}$ are pictured in Example
\ref{annulusparam}.

\begin{example}
When $n=4$, the definition above is illustrated
  by the picture below:
$$\BPic{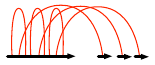}$$
\end{example}

\begin{defn}
  The category $\alg(-\Ad_{0,n},1-n)$ associated to the arc diagram
  $\Ad_{0,n}$ is the $k$-linear category determined by a quiver with
  vertices: $I_i$ and $J_i$ corresponding to the pairs $\{a_i,a'_i\}$ and
  $\{b_i,b'_i\}$ for $1\leq i < n$ respectively. There are arrows
  $\a_i : I_i \to J_i$, $\ga_i : J_i \to I_i$  and  $\nu_{i,i+1} : I_i \to I_{i+1}$ subject to the relations:
  \begin{enumerate}
  \item $\a_i \ga_i = 0 : J_i \to J_i$ and
  \item $\nu_{i+1,i+2}\nu_{i,i+1} = 0 : I_{i}\to I_{i+2}.$
  \end{enumerate}
\end{defn}

\begin{example}
The quiver underlying the category $\alg(-\Ad_{0,4},-2)$ in the definition above is illustrated below.
\begin{center}
\begin{tikzpicture}[scale=10, node distance=2.5cm]
\node (A1) {$J_1$};
\node (B1) [below=1.41421356cm of A1] {$I_1$};
\draw[->, bend left=35] (A1) to node {$\a_1$} (B1);
\draw[->, bend left=35] (B1) to node {$\ga_1$} (A1);

\node (A2) [right=3.5cm of A1] {$J_2$};
\node (B2) [below=1.41421356cm of A2] {$I_2$};
\draw[->, bend left=35] (A2) to node {$\a_2$} (B2);
\draw[->, bend left=35] (B2) to node {$\ga_2$} (A2);

\draw[->, bend right=35] (B1) to node [swap] {$\nu_{1,2}$} (B2);

\node (A3) [right=3.5cm of A2] {$J_{3}$};
\node (B3) [below=1.41421356cm of A3] {$I_{3}$};
\draw[->, bend left=35] (A3) to node {$\a_3$} (B3);
\draw[->, bend left=35] (B3) to node {$\ga_3$} (A3);
\draw[->, bend right=35] (B2) to node [swap] {$\nu_{2,3}$} (B3);

\end{tikzpicture}
\end{center}
\end{example}

The construction of the  functor 
$\s_n : \alg(\Ad_{0,n}, 1-n) \to \Ho(\Ko_+^{2n-4}(\Si_{0,n}, n\cdot 2))$
will occur in two stages.

First note that the parameterization of $\Si_{0,n}$ by the arc diagram
allows us to associate to each object, $I_i$ or $J_i$, $1\leq i< n$, a
dividing set contained in an annulus. In fact, Theorem \ref{zarevgenthm}
states that these dividing sets generate the contact category. In each
annulus we will describe bypass moves corresponding to the arrows
$\a_i : I_i \to J_i$ and $\ga_i : J_i \to I_i$. We will check that these
bypass moves satisfy the first collection of relations in the definition
above. After this has been done, bypass moves corresponding to the lateral
arrows $\nu_{i,i+1} : I_i \to I_{i+1}$ will be introduced and shown to
satisfy the second collection of relations.

\subsubsection{Step \#1}

For each annulus, the dividing sets $J_i$, $I_i$, and the bypass moves
corresponding to the maps 
$\ga_i  : J_i \to I_i$ and $\a_i : I_i \to J_i$ can be depicted by the curves:
$$\CPPic{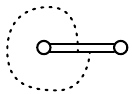} \quad\quad\leftrightarrows\quad \CPPic{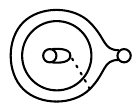}$$
The dividing set associated to $J_i$ is featured on the lefthand side and
the dividing set associated to $I_i$ is shown on the righthand side. The map
$\ga_i$ runs from left to right and the map $\a_i$ runs from right to
left. The equators of $\ga_i$ and $\a_i$ are determined by the dashed lines
in the dividing sets corresponding to $J_i$ and $I_i$ respectively.

\begin{prop}\label{step1prop1}
The relation $\a_i\ga_i = 0$ holds in the formal contact category $\Ho(\Ko_+(S^1 \times [0,1], (2,2)))$.
\end{prop}
\begin{proof}
  The map $\a_i\ga_i : J_i \to J_i$ is a composition of two disjoint bypass
  moves. This is illustrated below.
\begin{center}
\begin{overpic}%[grid,tics=10]
{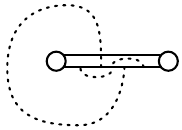}
 \put(20, 15){$\a_i$}
 \put(55,55){$\ga_i$}
\end{overpic}
\end{center}
Relation $(2)$ in the Definition \ref{formaldef} of the formal contact category implies that applying the two bypass moves in either order must commute:
$$\begin{tikzpicture}[scale=10, node distance=2.5cm]
\node (A) {$J_i$};
\node (B) [right=1.5cm of A] {$I_i$};
\node (C) [below=1.5cm of A] {$0$};
\node (D) [right=1.5cm of C] {$J_i,$};
\draw[->] (A) to node {$\ga_i$} (B);
\draw[->] (A) to node [swap] {} (C);
\draw[->] (C) to node {} (D);
\draw[->] (B) to node {$\a_i$} (D);
\end{tikzpicture}$$
but performing the bypass move $\a_i$ before the bypass move $\ga_i$ must be
zero since $\a_i$ is capped.
\end{proof}

The same argument shows that one of the terms in the commutative
diagram associated to the other composition $\ga_i \a_i$ is a
capped bypass equivalent to identity.

\subsubsection{Step \#2}

As pictured above, the idempotents $I_i$ correspond to the boundaries of
regular neighborhoods of loops about each boundary component of $\Si_{0,n}$.
We think of $\Si_{0,n}$ as a subset of the plane
$D^2\backslash \amalg_{i=1}^{n-1} D^2 \subset \RR^2$ with $n-1$ disks
removed from its interior. The arc parameterization
orders the boundary components and the associated idempotents. When two of them
are adjacent, $I_i$ and $I_{i+1}$, there is a bypass move
$\nu_{i,i+1} : I_i \to I_{i+1}$ determined by the equator of the bypass
disk in the illustration below.
\begin{center}
\begin{overpic}%[grid,tics=10]
{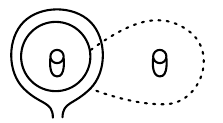}
 \put(-12.5, 30){$I_i$}
 \put(105,30){$\nu_{i,i+1}$}
\end{overpic}
\end{center}

\begin{prop}
  The relation: $\nu_{i+1,i+2}\nu_{i, i+1} = 0$ holds in the formal contact
  category $\Ho(\Ko_+(\Si_{0,n},n\cdot 2))$.
\end{prop}
\begin{proof}
  The proof is analogous to the proof of Proposition \ref{step1prop1}
  above. The bypass moves representing $\nu_{i,i+1}$ and
  $\nu_{i+1,i+2}$ are disjoint. Considering them simultaneously
  produces visual aid below.
$$\hspace{-2in}\CPPic{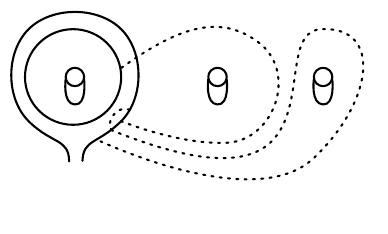}$$
The curve on the far right represents the equator of the bypass
$\nu_{i+1,i+2}$. Since this bypass move is capped the composition
factors through zero.
\end{proof}

\subsubsection{Y. Tian's annulus}
  As in Section \ref{tiansec} above, in Y. Tian's work \cite[\S 2.2]{YT2},
  the category associated to an annulus with two points on each boundary
  component is the pretriangulated hull on the free $k$-linear category
  associated to a quiver with five vertices: $I$, $E$, $F$ and $EF$. The
  dividing sets associated to $E$ and $F$ are Euler dual and are neither the
  source nor the target of any non-trivial edges.  There are two dividing
  sets $I$ and $EF$ generating the subcategory with Euler number zero
  via maps $\ga : I \to EF$ and $\a : EF \to I$ which are required to satisfy the relation:
$$\a\ga=0.$$
This description is  summarized by the illustration below.
$$\begin{tikzpicture}[scale=10, node distance=2.5cm]
\node (A) {$F$};
\node (B) [right=3cm of A] {$\ga : I$};
\node (C) [right=1.5cm of B] {$EF : \a$};
\node (D) [right=3cm of C] {$E$};
\draw[transform canvas={yshift=0.3ex},->] (B) -- (C);
\draw[transform canvas={yshift=-0.3ex},->] (C) -- (B);
\end{tikzpicture}$$
The quiver in the center is precisely $\alg(-\Ad_{0,2},0)$ above.

\begin{remark}
  It is natural to ask about surfaces $\Si_{0,n}$ with $n>2$. There are
  presently two constructions in the literature. In \cite{YT2}, the category
  associated to $\Si_{0,n}$ is a Bordered Heegaard Floer category by
  definition. Precisely the same can be said for the categories considered
  by I. Petkova and V. V\'{e}rtesi \cite{PV}. While the former chooses an
  arc parameterization which yields a heart encoding contact geometry, the
  latter chooses an arc parameterization which yields a Stendhalic  extension \cite{Webster}
  of the strands algebra \cite{LOT}. In both cases the arc parameterizations are {\em
    degenerate} so that the Border Heegaard Floer construction does not
  suffice to imply an equivalence between the two and the materials here 
  do not necessarily apply. 
\end{remark}

\subsection{A surface $\Si_{g,1}$ of genus $g$ with one boundary component}

\begin{defn}\label{symparcddef}
The arc diagram $\Ad_{g,1}$ consists of $4g$ points
$\pts = \{ a_i, a'_i, b_i,b'_i : 1\leq i \leq g \}$ on one line segment $\lns = \{ \Ad_1\}$
$$\Ad_1 = a_1 b_1 a'_1 b'_1 a_2 b_2 a'_2 b'_2 \cdots a_g b_g a'_g b'_g$$
which is oriented so that the indices above are increasing. The matching function is
determined by the rules $M(a_i) = M(a'_i)$ and $M(b_i) = M(b'_i)$ for $1\leq i \leq g$.
\end{defn}

\begin{example}
  The arc diagram $\Ad_{2,1}$ is illustrated below.
$$\BPic{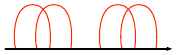}$$
\end{example}

\begin{defn}
The category $\alg(-\Ad_{g,1},2g-1)$ associated to the arc diagram $\Ad_{g,1}$ is the $k$-linear category determined by a quiver with vertices: $I_i$ and $J_i$ corresponding to the pairs $\{a_i,a'_i\}$ and $\{b_i,b'_i\}$ for $1\leq i \leq g$ respectively. There are arrows:
$$\a_i,\b_i : I_i \to J_i, \ga_i : J_i \to I_i  \conj{ and } \eta_{i,i+1} : I_i \to J_{i+1},$$
the compositions of which satisfy the relations below:
\begin{enumerate}
\item $\a_i \ga_i = 0 : J_i \to J_i$ and $\ga_i \b_i = 0 : I_i \to I_i$
\item $\eta_{i,i+1} \a_i = 0 : I_i \to I_{i+1}$ and  $\b_{i+1} \eta_{i,i+1} = 0 : J_i \to J_{i+1}$
\end{enumerate}
\end{defn}

Note that $\eta_{i,i+1} : I_i \to J_{i+1}$ is not the same as $\nu_{i,i+1} : I_i \to I_{i+1}$ in the previous section.

\begin{example}
 The quiver underlying the construction of the category $\alg(-\Ad_{2,1},3)$ is illustrated below.
\begin{center}
\begin{tikzpicture}[scale=10, node distance=2.5cm]
\node (A1) {$I_1$};
\node (B1) [right=1.41421356cm of A1] {$J_1$};
\draw[->, bend left=35] (B1) to node {$\ga_1$} (A1);
\draw[->, bend left=35] (A1) to node {$\a_1,\b_1$} (B1);

\node (A2) [right=2.26cm of B1] {$I_2$};
\node (B2) [right=1.41421356cm of A2] {$J_2$};
\draw[->, bend left=35] (B2) to node {$\ga_2$} (A2);
\draw[->, bend left=35] (A2) to node {$\a_2,\b_2$} (B2);

\draw[->, bend right=0] (B1) to node  {$\eta_{1,2}$} (A2);

\end{tikzpicture}
\end{center}
\end{example}

The construction of the  functor 
$\tau_g : \alg(-\Ad_{g,1}, 2g-1) \to \Ko^{2g-2}(\Si_{g,1},2)$
will occur in two stages.

First note that the parameterization of $\Si_{g,1}$ by the arc diagram
allows us to associate to each $i$, $1\leq i\leq g$, a pair of
dividing sets $I_i$ and $J_i$ contained in a torus
$\Si_{1,1} \subset \Si_{g,1}$ with one boundary component. In fact, Theorem
\ref{zarevgenthm} states that these dividing sets generate the category. In
each torus, we will describe bypass moves corresponding to the arrows
$\a_i,\b_i : I_i \to J_i$ and $\ga_i : J_i \to I_i$ and check that these
bypass moves satisfy the first collection of relations in the definition
above.

After this has been done, bypass moves corresponding to the lateral arrows
$\eta_{i,i+1}$ will be introduced and shown to satisfy the second collection
of relations.

\subsubsection{Step \#1}
For each torus, the dividing sets $I_i$, $J_i$, and the bypass moves
corresponding to the maps $\a_i,\b_i : I_i \to J_i$ and
$\ga_i : J_i \to I_i$ can be depicted by the curves:
$$\hspace{-.5in}\CPPic{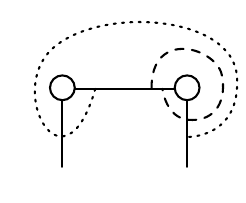} \hspace{1.25in}\leftrightarrows\hspace{.25in} \CPPic{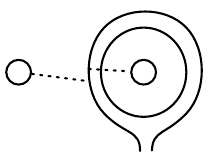}$$
On either side of the arrows in the picture above, the two small circles are
identified by folding the page to form the surface $(T^2\backslash D^2,2)$.  The
dividing set associated to $I_i$ is featured on the lefthand side and the
dividing set associated to $J_i$ is featured on the righthand side. The maps
$\a_i$ and $\b_i$ run from left to right and the map $\ga_i$ runs from right
to left. The equator of the map $\a_i$ is dotted and the equator of $\b_i$
is dashed.

\begin{prop}\label{Nextstep1prop1}
The relations $\b_i\ga_i = 0$ and $\ga_i\a_i = 0$ hold in the formal contact category $\Ho(\Ko_+(\Si_{1,1}, 2))$.
\end{prop}
\begin{proof}
The logic is analogous to the proof of Proposition \ref{step1prop1} above. The map $\b_i\ga_i$ is a composition of two disjoint bypass moves. When performed in the opposite order the bypass $\ga_i$ is capped implying that the composition $\b_i\ga_i$ factors through zero. This is illustrated below.
\begin{center}
\begin{overpic}%[grid,tics=10]
{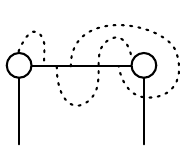}
 \put(35, 10){$\ga_i$}
 \put(80,60){$\b_i$}
\end{overpic}
\end{center}

The map $\ga_i\a_i$ is a composition of two disjoint bypass moves. When performed in the opposite order the bypass $\a_i$ is capped implying that the composition $\b_i\ga_i$ factors through zero. This illustrated below.
\begin{center}
\begin{overpic}%[grid,tics=10]
{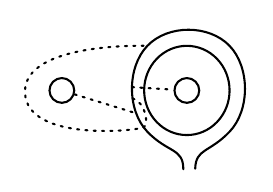}
 \put(20, 15){$\a_i$}
 \put(38.5,50){$\ga_i$}
\end{overpic}
\end{center}
\end{proof}

\subsubsection{Step \#2}

As pictured above, the idempotents $I_i$ correspond to the boundaries of
regular neighborhoods of loops about the first 1-handle and the idempotents
$J_i$ to the boundaries of regular neighborhoods of loops about the second
1-handle in the $i$th torus $\Si_{1,1} \subset \Si_{g,1}$.  The tori
$\Si_{1,1}$ are ordered by the arc parameterization and, when two tori are
adjacent, there is a bypass move $\eta_{i,i+1} : J_i \to I_{i+1}$ from the
dividing set about the second 1-handle of the first torus to the dividing
set about the first 1-handle of the second torus. The map $\eta_{i,i+1}$ is pictured below.
\begin{center}
\begin{overpic}%[grid,tics=10]
{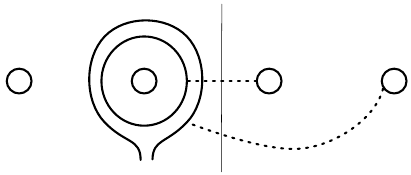}
\put(175,20){$\eta_{i,i+1}$}
\end{overpic}
\end{center}

In the illustration above, the first two and the second two smaller circles are connected by annuli $S^1\times [0,1]$ to form the $k$th and $k+1$st tori $\Si_{1,1} \subset \Si_{g,1}$.

\begin{prop}
  The relations: $\eta_{i,i+1} \a_i = 0 : I_i \to I_{i+1}$ and  $\b_{i+1} \eta_{i,i+1} = 0 : J_i \to J_{i+1}$ hold in the formal contact category $\Ho(\Ko_+(\Si_{g,1},2))$
\end{prop}
\begin{proof}
The logic is analogous to the proof of Proposition \ref{Nextstep1prop1} above. The map $\eta_{i,i+1} \a_i$ is a composition of two disjoint bypass moves. When performed in the opposite order the bypass $\eta_{i,i+1}$ is capped implying that the composition factors through zero. This is illustrated below.
\begin{center}
\begin{overpic}%[grid,tics=10]
{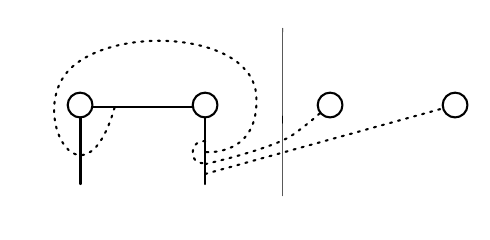}
 \put(105, 97){$\a_i$}
 \put(175,30){$\eta_{i,i+1}$}
\end{overpic}
\end{center}

The map $\b_{i+1} \eta_{i,i+1}$ is a composition of two disjoint bypass moves. When performed in the opposite order the bypass $\b_{i+1}$ is capped implying that the composition factors through zero. This is illustrated below.
\begin{center}
\begin{overpic}%[grid,tics=10]
{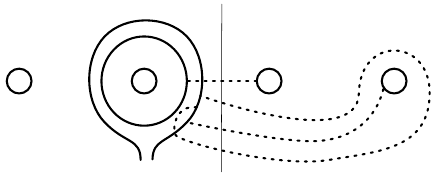}
 \put(215, 65){$\b_{i+1}$}
 \put(115,65){$\eta_{i,i+1}$}
\end{overpic}
\end{center}
\end{proof}

\newpage
\section{Comparison to geometric categories}\label{geosec}

One of the appealing qualities of the formal contact category $\Ko(\Si)$ is
that it has a universal property with respect to other dg categories by
construction. Although there is no underlying Floer theory or contact
geometry, this property allows us compare $\Ko(\Si)$ to other constructions
which stem from observations involving the former or the latter.  In this
section we will discuss why the univeral property of $\Ko(\Si)$ implies the
existence of maps:
$$\begin{tikzpicture}[scale=10, node distance=2.5cm]
\node (A) {$\Ko(\Si)$};
\node (B) [left=3cm of A] {};
\node (C) [right=3cm of A] {};
\node (B2) [below=1.25cm of B] {$\Co(\Si)$};
\node (C2) [below=1.25cm of C] {$\alg(-\Ad)\module$};
\draw[->,dashed] (A) to node {} (B2);
\draw[->] (A) to node [swap] {} (C2);
\end{tikzpicture}$$
in the homotopy category of dg categories which relate contact categories
$\Co(\Si)$ with the corresponding component of the Bordered Heegaard-Floer
theory.  See Sections \ref{contactsec} and \ref{bsfsec} for precise
statements.

\subsection{Relation to the Contact Category}\label{contactsec}

Much of the material in this paper was inspired by K. Honda's proposed {\em
  contact category} $\Co(\Si)$ \cite{KoHo}. Although a full account of this
construction is in preparation, in this section a modest comparison is drawn
between the formal and geometric contact categories.

The morphisms in the contact category $\Co(\Si)$ are tight contact
structures on $\Si\times [0,1]$.  More precisely, $\Co(\Si)$ is the
additivization \cite[\S 1.1.2.1]{L} of a category with objects given by
dividing sets $\ga$ on the surface $\Si$ and morphisms
$\theta : \ga \to \ga'$ given by contactomorphism classes of contact
structures on $\Si\times\I$, which induce $\ga$ and $\ga'$ on
$\partial\Si\times\I$, subject to the relation that an overtwisted contact
structure is zero. The composition is induced by the pullback of contact
plane fields along the rescaling diffeomorphism:
$\Si\times [0,1] \xto{\sim} \Si\times [0,1]\cup_{\Si} \Si\times [0,1]$.

The contact category $\Co(\Si)$ plainly exists. The maps in the contact
category $\Co(\Si)$ are generated by bypass moves between dividing sets
\cite[Lem. 3.10 (Isotopy discretization)]{KoGluing}. Since the bypass moves
satisfy the elementary relations $(1)$ and $(2)$ in Definition
\ref{pkosdef}, there is a functor: $\s : \PKoS \to \Co(\Si)$.  When
$(\Si,m)$ is a surface with boundary then the discussion in Section
\ref{overtwistedsec} suggests that these categories are very closely
related.

For the purposes of comparison we must make the non-trivial assumption
below.
\begin{assumption}
  The contact category $\Co(\Si)$ has pretriangulated dg enhancement $\Co^{\normaltext{dg}}(\Si)$ in which bypass triangles are distinguished triangles.
\end{assumption}

If this assumption is correct then there is a canonical lift 
$$\tilde{\s} : \KoS \to \Co^{\normaltext{dg}}(\Si)$$
of the dg functor $\s$ to a functor from the formal contact category to the dg category $\Co^{\normaltext{dg}}(\Si)$.

\begin{remark}\label{trirem}
In the formal contact category $\Ko(\Si)$, the bypass $\an$ involving the
annulus in the proof of Theorem \ref{surfacetheorem} determines a
distinguished triangle:
$$\ga \xto{\an} \ga' \xto{\an'} \ga'' \xto{\an''} \ga[1].$$
The map $\an'$ is not necessarily zero. However, in the geometric setting
$\an' = 0$, making the convolution $\ga \simeq C(\an')$ isomorphic to a
direct sum \cite{YTD}. As $\s(\an')=0$, it is possible to view $\Ho(\KoS)$
as a deformation.

\end{remark}

\subsection{Relation to the Bordered Sutured Floer Categories}\label{bsfsec}

\newcommand{\bsd}{\widehat{\normaltext{BSD}}}
\newcommand{\bsda}{\widehat{\normaltext{BSDA}}}
\newcommand{\DA}{\ga_A}
\newcommand{\DB}{\ga_B}
\newcommand{\DC}{\ga_C}
\newcommand{\tDA}{\tilde{\ga}_A}
\newcommand{\tDB}{\tilde{\ga}_B}

\newcommand{\DDA}{\aD_A}
\newcommand{\DDB}{\aD_B}
\newcommand{\DDC}{\aD_C}

\newcommand{\tDDA}{\widetilde{\aD}_A}
\newcommand{\tDDB}{\widetilde{\aD}_B}

\newcommand{\cp}{W}
\newcommand{\hd}{H}
\newcommand{\W}{\aW_3}

In this section we construct a functor $\widetilde{\Ko}_+(\Si,m)\to \alg(-\Ad)\module$
from a cofibrant replacement of the positive part of the formal contact
category to the category of left dg modules over an arc algebra of an arc
diagram $\Ad$ that parameterizes $\Si$. Assume that $(\Si,m)$ has at least
one boundary component and every boundary component $\partial_i\Si$ contains
a positive even number of points $m_i$. The ground ring $k$ of
$\Ko_+(\Si,m)$ is fixed to be the field $\FF_2$. We will not discuss
gradings here. The cofibrant replacement is a slightly larger, but
quasi-equivalent category, see Conj. \ref{cofibrem}. In particular, there
is a functor $\Ko_+(\Si,m) \to \alg(-\Ad)\module$ in $\Hqe$.

If $\ga$ is a dividing set on $\Si$ then R. Zarev associates a bordered
sutured manifold \cite[\S 3.2]{Zarev1} called the {\em cap} $\cp_{\ga}$ to
$\ga$. The cap $\cp_\ga$ is the 3-manifold $\Si\times [0,1]$ in which the
surface $\Si\times\{0\}$ is parameterized by the arc diagram $\Ad$, the
{\em sutures} $m$ are the $m$ boundary points, the dividing set
$\ga$ appears on $\Si\times \{1\}$ and the two sides are connected by
straight lines segments in $\partial\Si\times [0,1]$.
$$\cp_\ga = (\Si\times [0,1], \ga\times \{1\} \cup \Lambda \times [0,1], (-\Si\times \{0\},-\Lambda\times\{0\})).$$
For details concerning this definition consult \cite[Def. 2.5]{Zarev2}.

Associated to each bordered sutured manifold $Y$, there is a Heegaard
diagram $\hd(Y)$ \cite[\S 4]{Zarev1}. Associated to each Heegaard diagram
$\hd(Y)$ there is a left dg $\alg(-\Ad)$-module $\bsd(Y)$ \cite[\S
7.3]{Zarev1}. Notation for the module does not include the intermediate
Heegaard diagram because the homotopy type of the module is independent of
this choice.

If $\ga$ is a dividing set on $\Si$ such that the basepoint $z_1$ is
contained in the positive region $R_+\subset \Si\backslash \ga$ then $\ga$
determines an object $\ga\in \Ob(\Ko_+(\Si,m))$. To each such $\ga$ we associate
the left dg module $\bsd(\ga) = \bsd(\cp_\ga)$ associated to the cap for some choice of
Heegaard diagram.
\begin{equation}\label{obdef}
\ga \mapsto \bsd(\ga) \conj{ where } \bsd(\ga) = \bsd(\cp_{\ga})
\end{equation}

The disk $(D^2,6)$ can be parameterized by an arc diagram $\W$ pictured below:
$$\CPPic{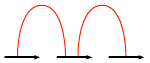}$$
The diagram $\W$ consists of three oriented line segments $\lns =\{ \Ad_1,\Ad_2,\Ad_3\}$ containing the points $\{a\}$, $\{a' < b \}$ and $\{b'\}$ respectively. The matching function $M$ is determined by $M(a) = M(a')$ and $M(b) = M(b')$.

As discussed in Proposition \ref{triprop}, the three important dividing sets
$\DA$, $\DB$ and $\DC$ in $(D^2,6)$ can be connected three bypass moves
$$\DA \xto{\th_A} \DB \xto{\th_B} \DC \xto{\th_C} \DC[1]\conj{ or } \MPic{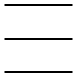} \xto{\th_A} \MPic{twob} \xto{\th_B} \MPic{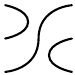} \xto{\th_C} \MPic{oneb}[1].$$
(The signs of the regions are fixed by requiring that the region containing
the basepoint is positive.)  Associated these three dividing sets, there are
three left $\alg(-\W)$-modules $\bsd(\DA)$, $\bsd(\DB)$ and $\bsd(\DC)$
corresponding to the bordered sutured diagrams given by the caps
$\cp_{\DA}$, $\cp_{\DB}$ and $\cp_{\DA}$.

In \cite[\S 6.2]{EVVZ}, the authors J. B. Etnyre, D. S. Vela-Vick and R. Zarev made a fundamental computation: after choosing Heegaard diagrams for the caps $\cp_{\DA}$, $\cp_{\DB}$ and $\cp_{\DC}$, they find that there are chain maps: $\phi_A : \bsd(\DA)\to \bsd(\DB)$, $\phi_B : \bsd(\DB) \to \bsd(\DC)$ and $\phi_C : \bsd(\DC) \to \bsd(\DA)$ such that 
$$\bsd(\DA) \xto{\phi_A} \bsd(\DB) \xto{\phi_B} \bsd(\DC) \xto{\phi_C} \bsd(\DA)[1]$$
is a distinguished triangle. They show explicitly that
\begin{enumerate}
\item $\bsd(\DA) = C(\phi_B)$
\item $\phi_A$ is projection and
\item $\phi_C$ is inclusion.
\end{enumerate}
(Alternatively, this follows from Section \ref{tiansec} and the Morita
invariance of the category associated to the disk by Bordered Sutured Floer
theory.) Our functor is defined using the pairing theorem to extend the
assignments: $\th_A \mapsto \phi_A$, $\th_B \mapsto \phi_B$ and
$\th_C \mapsto \phi_C$ to all of the other bypass moves between dividing
sets.

Throughout the remainder of this section, we will make repeated use the
pairing theorem.  Suppose that $\ga$ is a dividing set on $\Si$ and the first basepoint $z_1$ is contained in a positive region.
Then if
$D = (D^2,2m) \subset \Si$ is an embedded disk with $2m$ points on the
boundary such that $\ga^{\circ} = \ga\backslash (D\cap \ga)$ is a dividing set
on $\Si\backslash D$ then the pairing theorem \cite[Thm. 8.7]{Zarev1} gives
a homotopy equivalence:
$$\bsd(\ga)\xto{\sim} \bsda(\ga^{\circ})\boxtimes \bsd(\ga\cap D)\conj{ } \ga = \ga^\circ \cup_{\ga\cap \partial D} (\ga \cap D)$$
where $\ga\cap \partial D = 2m$, $\bsd(\ga) = \bsd(\cp_{\ga})$ is the left dg $\alg(-\Ad)$-module
assigned to the dividing set $\ga$,
$\bsda(\ga^{\circ}) = \bsda(\cp_{\ga^{\circ}})$ is a left
$\alg(-\Ad)$-module and right $A_{\infty}$ $\alg(-\W)$-module,
$\bsd(\ga\cap D) = \bsd(\cp_{\ga\cap D})$ is the left $\alg(-\W)$-module
determined by $\ga$ in the interior of the disk $D$ and the box product
$\boxtimes$ is an analogue of the derived tensor product, see \cite[\S
2.4]{LOT}.

\begin{defn}\label{mapsdef}
If $\th : \ga \to \eta$ is a bypass move then the map $\th_* : \bsd(\ga) \to \bsd(\eta)$ of dg modules associated to $\th$ is determined by the commutative diagram:
$$\begin{tikzpicture}[scale=10, node distance=2.5cm]
\node (A) {$\bsd(\ga)$};
\node (B) [right=1.5cm of A] {$\bsda(\ga^{\circ})\boxtimes \bsd(\DA)$};
\node (C) [below=1.5cm of A] {$\bsd(\eta)$};
\node (D) [right=1.5cm of C] {$\bsda(\ga^{\circ})\boxtimes \bsd(\DB)$};
\draw[->] (A) to node {$\sim$} (B);
\draw[->] (A) to node {$\th_*$} (C);
\draw[->] (B) to node {$1\boxtimes \phi_A$} (D);
\draw[->] (C) to node {$\sim$} (D);
\end{tikzpicture}$$
where $\ga^\circ = \ga\backslash D$, introduced above, denotes the dividing set minus the region
containing the equator of the bypass disk associated to $\th$. 
\end{defn}

In order for the maps chosen above to yield a functor from the pre-formal
contact category, we must check that relations (1) and (2) in Definition
\ref{pkosdef} above are satisfied. Since these relations hold up to homotopy
in the category $\alg(-\Ad)\module$, this determines a functor from the
cofibrant replacement of the pre-formal contact category. Lastly we will
show that this functor factors through the Postnikov localization introduced
by Proposition \ref{triprop}.

\subsubsection{Relation (1)}

If $\th$ is capped in the northwest or southeast then relation (1) must hold
up to homotopy by the invariance of the bordered sutured theory \cite[\S 7]{Zarev1}.

In more detail, suppose that $\th : \ga\to \eta$ is a bypass move and $D$ is a neighborhood of the equator of the underlying bypass disk. Then when there is a cap, the region $D$ can be enlarged to a region $\tilde{D}$ which contains the cap disk in $\Si$. Two applications of the pairing theorem give:
$$\begin{tikzpicture}[scale=10, node distance=2.5cm]
\node (A) {$\bsd(\ga)$};
\node (C) [below=1.5cm of A] {$\bsd(\eta)$};
\draw[->] (A) to node {$\th_*$} (C);

\node (B) [right=1.5cm of A] {$\bsda(\ga^{\circ})\boxtimes \bsd(\DA)$};
\node (D) [below=1.5cm of B] {$\bsda(\ga^{\circ})\boxtimes \bsd(\DB)$};
\draw[->] (B) to node {$1\boxtimes \phi_A$} (D);

\draw[->] (A) to node {$\sim$} (B);
\draw[->] (C) to node {$\sim$} (D);

\node (E) [right=1.5cm of B] {$\bsda(\tilde{\ga}^{\circ})\boxtimes \bsd(\tDA)$};
\node (F) [below=1.5cm of E] {$\bsda(\tilde{\ga}^{\circ})\boxtimes \bsd(\tDB)$};
\draw[->] (E) to node {$1\boxtimes \tilde{\phi}_A$} (F);

\draw[->] (B) to node {$\sim$} (E);
\draw[->] (D) to node {$\sim$} (F);
\end{tikzpicture}$$
where $\ga^\circ = \ga\backslash D$ and $\tilde{\ga}^{\circ} = \ga\backslash\tilde{D}$. The dividing sets  $\tDA$ and $\tDB$, on the righthand side above, are identical when the cap is either northwestern or southeastern. They are both represented by the same Heegaard diagram and the map $\tilde{\phi}_A$ is identity. It follows that $\th_*$ is homotopic to identity.

\subsubsection{Relation (2)}

In order to see that disjoint bypass moves $\th \amalg \th' : \ga \to \eta$
commute we must cut the dividing set $\ga$ along the two disjointly embedded
disks corresponding to neighborhoods of the equators of our bypass moves to form $\ga^{\circ\circ} = \ga\backslash (D\amalg D')$. The arc algebra associated to a disjoint union splits, $\varphi : \alg(-(\W\amalg \W)) \xto{\sim} \alg(-\W)\ott_k \alg(-\W)$, the module $\bsd(\DA)\ott_k\bsd(\DA)$ appears in the pairing theorem:
$$\bsd(\ga) \xto{\sim} \bsda(\ga^{\circ\circ})\boxtimes\left[\bsd(\DA)\ott_k\bsd(\DA)\right]$$
and the disjoint union of Heegaard diagrams splits as a tensor product compatible with the isomorphism $\varphi$ above. Under this identification, the maps $\th_*$ and $\th'_*$ induced by $\th$ and $\th'$ correspond to different tensor factors and must commute by the standard algebraic fact that:
$$(1_{\ga^{\circ \circ}} \boxtimes [1_{A} \ott \th'_*]) (1_{\ga^{\circ\circ}} \boxtimes [\th_* \ott 1_{A}]) = (1_{\ga^{\circ\circ}} \boxtimes [\th_* \ott 1_{A}]) (1_{\ga^{\circ \circ}} \boxtimes [1_{A} \ott \th'_*])$$
where $1_{\ga^{\circ \circ}}$ and $1_A$ are used to denote the identity maps $1_{\bsda(\ga^{\circ\circ})}$ and $1_{\bsd(\DA)}$ respectively.

\subsubsection{Triangles}
Finally, it is necessary to see that the objects and the maps assigned by Equation \eqref{obdef} and Definition \ref{mapsdef} factor through the Postnikov localization constructed in Proposition \ref{triprop}

These choices form distinguished triangles because 
\begin{align*}
\bsd(\ga) &= \bsd(\ga\cup \DA) \\
          &\simeq \bsda(\ga^{\circ})\boxtimes \bsd(\DA)\\
          &\simeq \bsda(\ga^\circ) \boxtimes C(\phi_B)\\
          &\simeq C(1_{\bsda(\ga^{\circ})} \boxtimes \phi_B)\\
          &\simeq C(\th'_*)
\end{align*}
where the last equivalence corresponds to the commutative diagram in
Definition \ref{mapsdef} above after rotating the triangle. An analogue of
this argument appears in \cite[Thm. 4.1]{LOTBRANCHED}.

\newpage

\section{Appendix: Dg categories}\label{dgcatsec}

This section contains some materials about dg categories and the model structures. All of the definitions below are from the literature. More information about 
differential
graded categories can be found in \cite{Kellerdg, Toen} or \cite[\S 1]{DK};
consult \cite{TabAdd, Tabuada, T} for technical details. The language of
model categories is reviewed in the reference \cite[\S A.2]{LT}, more
details can be found in the references \cite{Hovey, Quillen}.

\begin{definition}\label{dgcatdef}
  A {\em dg category $\eC$ over $\aA$} is a category enriched in the monoidal category of chain complexes:
  $$\Hom_{\eC}(x,y) \in \Kom_k(\aA) \conj{ for all } x,y \in \Ob(\eC),$$
  such that composition in $\eC$ is a map in $\Kom_k(\aA)$. A {\em functor} $f : \eC \to \eD$ between two such dg categories is required to consist of maps in $\Kom_k(\aA)$:
\begin{equation}\label{dgfuneq}
  f_{x,y} : \Hom_\eC(x,y) \to \Hom_{\eD}(f(x),f(y)) \in \Kom_k(\aA)
\end{equation}  
  \end{definition}

A dg functor $f : \eC\to \eD$ is {\em fully faithful} when for any pair $x,y\in \Ob(\eC)$ the map $f_{x,y}$ in Eqn. \eqref{dgfuneq} is a isomorphism of chain complexes. If the homology $H^*(f_{x,y})$ induces an isomorphism for all pairs then $f_{x,y}$ is called {\em quasi-fully faithful}. A functor $f : \eC\to \eD$ is a {\em quasi-isomorphism} of dg categories when $H^*(f) : H^*(\eC) \to H^*(\eD)$ induces an equivalence of graded $k$-linear categories.

\newcommand{\Chd}{\Kom_k^*}
\newcommand{\Mod}{Mod}
\begin{example}
  The category of chain complexes $\Kom_k(\aA)$ is a subcategory $\Kom_k(\aA) \subset \Kom_k^*(\aA)$ of a dg category. The objects of $\Kom_k^*(\aA)$ are the chain complexes $(C,\partial_C)\in \Kom_k(\aA)$. The maps are now given by the chain complex $(Hom^*((C,\partial_C),(D,\partial_D)),\delta)$
  $$\Hom^n((C,\partial_C),(D,\partial_D)) := \prod_{m\in\ZZ} Hom(C^m, D^{n+m})$$
  with differential $\delta(f) := d_D f + (-1)^{n+1} f d_C$ for $f$ of degree $n$.
  \end{example}

When $\aA$ is $\Vect_k$, the category of dg categories will be denoted by $\dgcat$. Important for this paper is a sequence of localizations obtained by different model category structures on $\dgcat$.
\begin{equation}\label{loceq}
  \dgcat \xto{(1)} \Hqe \xto{(2)} \Hmo
  \end{equation}

{\bf Hqe:}
The first category $\Hqe := \dgcat[W^{-1}]$ is obtained by requiring quasi-isomorphisms $f\in W$ to be isomorphisms. In this model structure cofibrations are determined by the left lifting property with respect to fibrations and fibrations are dg functors $f : \eC \to \eD$ for which $f_{x,y}$ in Eqn. \eqref{dgfuneq} are surjective and
\begin{itemize}
\item For $x\in \Ob(\eC)$ and any homotopy equivalence $\b : f(x) \to y$ in $\eD$ there is a homotopy equivalence $\a : x \to z$ in $\eC$ so that $f(\a) = \b$.
  \end{itemize}
The inital object is the empty category $\emptyset$ with no objects and the
final object $0$ is the zero dg category consisting of one object with no
endomorphisms. In $\Hqe$ non-trivial dg categories are fibrant and cofibrant resolutions are can be obtained from Cobar-Bar construction.

{\bf Modules: } For any dg category there are associated categories of modules over that dg category.

A right dg module $M$ over a dg category $\eC$ is a dg functor $\eC^{\op}\to \Chd(\Vect_k)$. The dg category of such functors will be denoted by $\Mod_{\eC}$. The homology $H^*(M) : \eC^{op}\to \Vect_k^{\ZZ}$ of a dg module $M$ is the functor $c\mapsto H^*(M(c))$ taking values in graded vector spaces. A {\em quasi-isomorphism} $g : M\to N$ of dg modules is a map inducing an isomorphism between their respective homologies. The derived category $D(\eC)$ of dg modules over a dg category $\eC$ is obtained by inverting the quasi-isomorphisms $Q$
$$D(\eC) := \Mod_{\eC}[Q^{-1}]$$
This is a triangulated category \cite{Kellerder}.
If $f : \eC \to \eD$ is a dg functor then there is a pushforward functor $f_! : \Mod_\eC \to \Mod_\eD$ which is left adjoint to the pullback $f^* : \Mod_\eD\to \Mod_\eC$. These functors induce functors between derived categories
$$ f_! : D(\eC) \leftrightarrow D(\eD) : f^*.$$
A dg functor $f : \eC\to \eD$ is a {\em Morita equivalence} when $f^* : D(\eD) \to D(\eC)$ is an equivalence of triangulated categories.

{\bf Hmo:}
The category $\Hmo$ is obtained by inverting Morita equivalences $M$. 
$$\Hmo := \Hqe[M^{-1}]$$
The category $\Hmo$ is pointed: the dg category $1$ consisting of a single object and a single morphism is both initial and terminal. The cofibrant objects of $\Hmo$ and $\Hqe$ remain the same. Fibrant objects become pretriangulated dg categories as discussed in the next paragraph.

There is a full subcategory $\eC^{\perf} \subset \Mod_{\eC}$ consisting of
modules $M$ which are compact in $D(\eC)$. Since representable modules are compact the Yoneda embedding factors through the subcategory of perfect modules, giving a dg functor
$$\gamma : \eC\to \eC^{\perf}$$
A dg category $\eC$ is called {\em perfect} when $\gamma$ is a quasi-equivalence. A dg category $\eC$ in $\Hmo$ is fibrant if and only if it is perfect. So $\gamma$ is fibrant replacement. An explicit model for $\eC^{\perf}$ is given by the idempotent completion of the category of one-sided twisted complexes over $\eC$ \cite[\S 2.4]{D}.

{\bf Maps:}
To\"{e}n's theorem shows that maps $\eC \to \eD$ in $\Hqe$ and are given by bimodules $\eC\ott \eD^{\op}\to \Chd(\Vect_k)$ satisfying certain cofibrancy and representability conditions \cite{T}. If $\eD$ is fibrant then these are also the maps in $\Hmo$. Dg functors described above define maps in each of these settings.
%In this article we are somewhat relaxed with the language and may refer to maps in these localizations as functors.

{\bf Constructions in $\Hqe$ vs $\Hmo$:}
If $\eC\to \eD$ and $\eC\to \eE$ in $\Hqe$ then the homotopy pushout $\eD\sqcup^h_\eC \eE$ can be constructed by using the coproduct of dg categories on the associated pushout of cofibrant replacements. Since cofibrant objects in $\Hqe$ and $\Hmo$ agree the quotient $\Hqe\to \Hmo$ commutes with homotopy pushout.

Since all of the localization constructions in this article are homotopy pushouts, they are indifferent to the distinction between $\Hqe$ and $\Hmo$ in the manner described above.

\section{Glossary of notation}

 After Section 2 dg categories are ungraded over a field of characteristic 2.
 The homotopy category of dg categories $\Ho(\dgcat)$ over $k$ will be
 denoted by $\Hqe$ or $\Hmo$ when the equivalence relation is
 quasi-equivalence or Morita equivalence respectively. All surfaces denoted
 by $\Si$ are connected unless otherwise mentioned. $\Si_{g,n}$ is the
 orientable surface of genus $g$ with $n$ boundary components.

\noindent
\begin{multicols}{1}
\begin{list}{}{
  \renewcommand{\makelabel}[1]{#1\hfil}
}
\item[$-^\vee$] Prop. \ref{dualityfunprop}.
\item[$-^\op$] opposite category.
\item[$\pts$] points $\{ a_1,\ldots, a_{2k}\}$ in arc diagram, Def. \ref{arcddef}.
\item[$a_k, a'_k$] points in an arc diagram, Def. \ref{arcddef}.
\item[$\alg(\Ad)$] arc algebra, \cite{LOT, Zarev1}.
\item[$B$] bottom of $D^2$.
\item[$B$] $B\subset \Si$, Def. \ref{discodef}.
\item[$\bsd(\ga)$] Eqn. \ref{obdef}.
\item[$\eC$] dg category, After \S2 ungraded, see \S\ref{grading}.
\item[$c_i$] cocore of 1-handle.
\item[$\Co(\Si)$] geometric contact category or Y. Tian algebraic contact category.
\item[$d$] differential, $d^2 = 0$.
\item[$\dgcat$] category of dg categories, \cite{D, Toen}.
\item[$D', \bar{D}, \tilde{D}$] Def. \ref{Ttridef}, Def. \ref{didef}.
\item[$D^2$] unit disk.
\item[$e_k$] generator of $\P_n$, Def. \ref{Pdef}.
\item[$\spc(\ga)$] Def. \ref{eudef}.
\item[$\asurf$] surface of arc diagram, Def. \ref{surfdef}.
\item[$F(\partial \Ad)$] Prop. \ref{decatprop}
\item[$\ga$] dividing set, Def. \ref{divsetdef}.
\item[$\ga(\e_i)$] dividing set associated to $e_i$, \S\ref{nilTLsec}.
\item[$\ga_A,\ga_B,\ga_C$] bypass triangle, Prop. \ref{triprop}, \S\ref{bsfsec}.
\item[$\ga^\vee$] dual dividing set, Def. \ref{dualdef}, Prop. \ref{dualityfunprop}.
\item[$(\oplus_{i=1}^n \ga_i,p)$] convolution of dividing sets, Def. \ref{ptdef}.
\item[$\Ga(\Si)$] mapping class group, \S\ref{mcgsec}.
\item[$h_k$] 1-handle in $F(\zigzag_n)$.
\item[$\Ho(\eC)$] $[\eC]$ or $H^0(\eC)$, \cite{Toen}.
\item[$Hom^I$] Def. \ref{toenlocdef}.
\item[$Hom^T$] Def. \ref{anlocdef}.
\item[$Hom^{\inp{K}}$] Prop. \ref{sesprop}.
\item[$\Hmo$] Morita homotopy category, \cite{TabAdd}.
\item[$\Hqe$] homotopy category, \cite{Tabuada}.
\item[$i(x,y)$] geometric intersection number.
\item[$int(X)$] interior of $X$.
\item[$I,I',\bar{I}$] Def. \ref{intcatdef} and Def. \ref{Kdef}.
\item[$k$] ground field. After \S2, $char(k)=2$.
\item[$\kappa,\kappa'$] Def. \ref{intcatdef} and Def. \ref{didef}.
\item[$\inp{\kr}$] Prop. \ref{sesprop}.
\item[$K_0(\eC)$] Grothendieck group, \cite{TabAdd}.
\item[$\KoS$] Def. \ref{formaldef}.
\item[$\Ko^n(\Si,m)$] Thm. \ref{splitthm}.
\item[$\Ko^n_{\pm}(\Si,m)$] $\pm$-halves of $\Ko^n(\Si,m)$, \S\ref{possec}.
\item[$L_R\eC$] Def. \ref{intcatdef}.
\item[$L_S\eC$] Def. \ref{anlocexprop}.
\item[$m$] boundary points $m\subset \partial\Si$, \S \ref{curvessec}.
\item[$M$] a matching $M : \pts \to \{ 1,\ldots,k\}$ in arc diagram, Def. \ref{arcddef}.
\item[$\mu$] \S \ref{HFCodisksec}.
\item[$\zigzag_n$] zig-zag diagram for $(D^2,2n)$, Def. \ref{zigzagdef}.
\item[$\Mat(\eC)$] the additive closure, \S \ref{trisec}.
\item[$\P_n$] nil-Temperley-Lieb algebra, Def. \ref{Pdef}.
\item[$\NN$] $\NN = \{ 0 \} \cup \ZZ_+$.
\item[$nS^1$] Def. \ref{nsonedef}.
\item[$N(T)$] neighborhood of disk, Def. \ref{bypassattdef}.
\item[$\PKoS$] Def. \ref{pkosdef}.
\item[$\PPKoS$] Conj. \ref{cofibrem}.
\item[$\Q_n$] Y. Tian quiver, Def. \ref{YTquiverdef}.
\item[$r$] Basepoint automorphism, Cor \ref{rotationprop}.
\item[$\rho_{k,k\pm 1}$] Eqn. \eqref{rhogeneqn}.
\item[$\R_n$] Y. Tian disk category, Def. \ref{diskdef}.
\item[$R_\pm$] positive and negative regions, Def. \ref{divsetdef}.
\item[$S^2$] the 2-sphere.
\item[$\Si_{g,n}$] connected surface of genus $g$ with $n$ boundary components.
\item[$(\Si,m)$] pointed oriented surface, Def. \ref{psurfdef}.
\item[$\bar{\Si}$] orientation reversal, Prop. \ref{orrevprop}.
\item[$\partial_i \Si$] $i$th boundary component of $\Si$, Def. \ref{psurfdef}.
\item[$\th : \ga\to\ga'$] bypass move, Def. \ref{bypassattdef}.
\item[$\th_{i,j}$] Def. \ref{Ttridef}, Def. \ref{didef}.
\item[$T$] top of $D^2$.
\item[$(T,\ga,\ga')$] bypass attachment, Def. \ref{bypassattdef}.
\item[$\cp_\ga$] cap associated to $\ga$, \cite[Def. 2.5]{Zarev2}, \S \ref{bsfsec}.
\item[$\Xi$] Eqn. \eqref{locfun}, Def. \ref{formaldef}.
\item[$\Yi_n$] $\R_n^\pt$, Def. \ref{hulldiskdef}.
\item[$z$] basepoints $z = \{ z_1,\ldots, z_n\}$, $z_i\in\partial_i \Si$, Def. \ref{psurfdef}.
\item[$\z_C$] $z_C\in \Zi(\Ad)$, Def. \ref{elemdef}.
\item[$\ZZ_+$] $\{ 1,2, 3,\ldots\} \subset \ZZ$.
\item[$\ZZ/2$] $\ZZ/2\ZZ$.
\item[$\lns$] ordered set of lines, Def. \ref{arcddef}.
\item[$\Ad$] arc diagram, Def. \ref{arcddef}.
\item[$\Ad_i$] arc in arc diagram, Def. \ref{arcddef}.
\item[$\Ad_{0,n}$] arc diagram for $\Si_{0,n}$, Def. \ref{snparamdef}.
\item[$\Ad_{g,1}$] arc diagram for $\Si_{g,1}$, Def. \ref{symparcddef}.
\item[$\Zi(\Ad)$] set of elementary dividing sets, Def. \ref{elemdef}.
\end{list}
\end{multicols}


\begin{thebibliography}{10}


 \bibitem{BV} A. Beilinson and V. Vologodsky, {\em A Dg Guide to Voevodsky's Motives}, see \url{http://arxiv.org/abs/math/0604004}
\bibitem{BK} A. I. Bondal and M. M. Kapranov, {\em Enhanced Triangulated Categories}, Math. USSR. Sbornik, 70, 1991, no. 1, 93--107.

\bibitem{HondaGroup} V. Colin, P. Ghiggini and K. Honda, {\em The equivalence of Heegaard Floer homology and embedded contact homology via open book decompositions I-III}, arXiv:1208.1074, arXiv:1208.1077, arXiv:1208.1526.

\bibitem{C} B. Cooper, {\em Formal contact categories II}, in preparation.

\bibitem{Donaldson} S. K. Donaldson, {\em The Seiberg-Witten equations and 4-manifold topology}, Bull. Amer. Math. Soc., 33, 1996, 45--70.

\bibitem{D} V. Drinfeld, {\em DG quotients of DG categories}, J. Algebra, 272, 2004, no. 2, 643--691.

\bibitem{Dyckerhoff} T. Dyckerhoff, {\em Compact generators in categories of matrix factorizations}, Duke Math. J., 159, 2011, no. 2, 223--274.
\bibitem{DK} T. Dyckerhoff and M. M. Kapranov, {\em Triangulated surfaces in triangulated categories}, arXiv:1306.2545v3.

\bibitem{EVVZ} J. B. Etnyre, D. S. Vela--Vick and R. Zarev, {\em Sutured Floer homology and invariants of Legendrian and transverse knots}, arXiv:1408.5858.

\bibitem{FarbMargalit} B. Farb and D. Margalit, {\em A primer on mapping class groups.} Princeton Mathematical Series, 49. Princeton University Press, 2012.

\bibitem{Freedman} M. H. Freedman, {\em A magnetic model with a possible Chern-Simons phase}, Appendix by F. Goodman and H. Wenzl,  Comm. Math. Phys., 234, 2003, no. 1, 129--183.

\bibitem{GM} S. I. Gelfand, Y. Manin, {\em Methods of homological algebra}, Springer 1988.
\bibitem{G} E. Giroux, {\em Structures de contact sur les vari\'{e}t\'{e}s fibr\'{e}es en cercles au-dessus d'une surface}, Comment. Math. Helv., 76, 2001, 218--262.
\bibitem{HondaTian} K. Honda and Y. Tian, {\em Contact categories of disks}, arXiv:1608.08325.
\bibitem{KoHo} K. Honda, {\em Contact structures, Heegaard Floer homology and triangulated categories}, In preparation.
\bibitem{KoGluing} K. Honda, {\em Gluing tight contact structures}, Duke Math. J., 3, 2002, 435--478.
\bibitem{KoOn} K. Honda, {\em On the classification of tight contact structures I}, Geometry and Topology, vol. 4, 2000, 309--368.

\bibitem{Hovey} M. Hovey, {\em Model Categories}, American Mathematical Society, Mathematical Surveys and Monographs, Vol. 63, 2007.

\bibitem{H1} M. Hutchings, {\em An index inequality for embedded pseudoholomorphic curves in symplectizations}, J. of Eur. Math. Soc., 4, 2002, no. 4  313--361.

\bibitem{H2} M. Hutchings and C. H. Taubes, {\em Gluing pseudoholomorphic curves along branched covered cylinders I}, J. Symplectic Geom., 5, 2007, no. 1, 43--137.

\bibitem{H3} M. Hutchings and C. H. Taubes, {\em Gluing pseudoholomorphic curves along branched covered cylinders II}, J. Symplectic Geom., 7, 2009, no. 1, 29--133.


\bibitem{AJ} A. Juh\'{a}sz, {\em A survey of Heegaard Floer homology}, New Ideas in Low Dimensional Topology, World Scientific, 2014, 237--296.
\bibitem{KL} L.H. Kauffman and S.L. Lins, {\em Temperley-Lieb recoupling theory and invariants of $3$-manifolds}, Princeton University Press, 1994.
\bibitem{Kellerder} B. Keller, {\em Deriving DG categories}, Annales scientifique de l'\'{E}.N.S., 4, 27, 1997, no. 1, 63--102.
\bibitem{Kellerdg} B. Keller, {\em On differential graded categories}, ICM 2006.

\bibitem{K} M. Kontsevich, {\em Symplectic geometry of homological algebra}, 
see \url{http://www.ihes.fr/~maxim/TEXTS/Symplectic_AT2009.pdf}

\bibitem{KM} P. Kronheimer and T. Mrowka, {\em Monopoles and Three-manifolds}, New Math. Monogr. 10, Cambridge Univ. Press, 2007.
\bibitem{TaubesGroup} \c{C}. Kutluhan, Y-J. Lee and C. H. Taubes, {\em HF=HM I-V : Heegaard Floer homology and Seiberg--Witten Floer homology}, arXiv:1007.1979, arXiv:1008.1595, arXiv:1010.3456, arXiv:1107.2297,  arXiv:1204.0115.

\bibitem{LOT} R. Lipshitz, P. Ozsv\'{a}th and D. Thurston, {\em Bordered Heegaard Floer homology: Invariance and Pairing},  arXiv:0810.0687.

\bibitem{LOTBRANCHED} R. Lipshitz, P. Ozsv\'{a}th and D. Thurston, {\em Bordered Floer homology and the branched double cover I}, Journal of Topology, 7, 2014, no. 4, 1155--1199.

\bibitem{LOTMCG} R. Lipshitz, P. Ozsv\'{a}th and D. Thurston, {\em A faithful linear-categorical action of the mapping class group of a surface with boundary}, J. of the EMS,  15, 2013, no. 4, 1279--1307.

\bibitem{L} J. Lurie, {\em Higher Algebra}, see \url{http://www.math.harvard.edu/~lurie/papers/higheralgebra.pdf}
\bibitem{LT} J. Lurie, {\em Higher Topos Theory}, see \url{http://www.math.harvard.edu/~lurie/papers/croppedtopoi.pdf}

\bibitem{Mathews} D. V. Mathews, {\em Strand algebras and contact categories}, arXiv:1608.02710.

\bibitem{Murakami} J. Murakami, {\em The multivariable Alexander polynomial and a one-parameter family of representations of $U_q(\mathfrak{sl}(2,\CC))$ at $q^2 = -1$}, Lecture Notes in Math., 1510, Springer, Berlin, 1992.

\bibitem{Nadler} D. Nadler, {\em Cyclic symetries of $A_n$-quiver representations}, arXiv:1306.0070v2.

\bibitem{OZ2} P. Ozsv\'{a}th and Z. Szab\'{o}, {\em Holomorphic disks and topological invariants for closed three-manifolds}, Ann. of Math. 159, 2004, no. 3, 1027--1158.

\bibitem{OZ1} P. Ozsv\'{a}th and Z. Szab\'{o}, {\em Holomorphic disks and 3-manifold invariants: properties and applications}, Ann. of Math., 159, 2004, no. 3, 1159--1245.

\bibitem{PV} I. Petkova and V. V\'{e}rtesi, {\em Combinatorial tangle Floer homology}, Geometry \& Topology to appear, arXiv:1410.2161.


\bibitem{Quillen} D. Quillen, {\em Homotopical Algebra}, Lecture Notes in
  Mathematics, 43, Springer--Verlag, Berlin--New York, 1967.

\bibitem{TabAdd} G. Tabuada, {\em  Invariants Additifs de dg-cat\'{e}gories}, IMRN, 53, 2005, 3309--3339.

\bibitem{TabuadaLoc} G. Tabuada, {\em On Drinfeld's dg quotient}, Journal of Algebra, 323, 2010, 1226--1240.

\bibitem{Tabuada} G. Tabuada, {\em Une structure de cat\'{e}gorie de mod\`{e}les de Quillen sur la cat\'{e}gorie des dg-cat\'{e}gories}, C. R. Acad. Sci. Paris S\'{e}r. I Math., 340, 2005, no. 1, 15--19.

\bibitem{Taubes1} C. H. Taubes, {\em Embedded contact homology and Seiberg-Witten Floer cohomology I--V}, Geom. Topol., 14, 2010, no. 5, 2497--3000.

\bibitem{YT1} Y. Tian, {\em A categorification of $U_qsl(1\vert 1)$ as an algebra}, see \url{http://arxiv.org/abs/1210.5680}
\bibitem{YT2} Y. Tian, {\em A categorification of $U_T sl(1,1)$ and its tensor product representations}, Geom. Topol., 18, 2014, 1635--1717.

\bibitem{YTD} Y. Tian, {\em Private communication}, 2015.





\bibitem{Toen} B. To\"{e}n, {\em Lectures on dg-categories}, Topics in algebraic and topological K-Theory, Lect. Notes in Math., 2008, 243--302.
\bibitem{T} B. To\"{e}n, {\em The homotopy theory of dg-categories and derived Morita theory}, Invent. Math., 167, 2007, no. 3, 615--667.

\bibitem{Viro} O. Viro, {\em Quantum relatives of the Alexander polynomial}, Algebra i Analiz, 18, 2006, no. 3, 63--157.

\bibitem{Walker2} K. Walker, {\em TQFTs [early incomplete draft]}, 2006, see \url{http://canyon23.net/math/}.

\bibitem{Webster} B. Webster, {\em Tensor product algebras, Grassmannians and Khovanov homology}, arXiv:1312.7357v1.

\bibitem{W} E. Witten, {\em Monopoles and 4-manifolds}, Math. Res. Letters, 1, 1994, 769--796.


\bibitem{Zarev1} R. Zarev, {\em Bordered Floer homology for sutured manifolds}, arXiv:0908.1106.
\bibitem{Zarev2} R. Zarev, {\em Joining and gluing sutured Floer homology}, arXiv:1010.3496v1.


\end{thebibliography}
\end{document}